\documentclass[11pt,reqno]{amsproc}

\usepackage{mathtools}
\usepackage{amssymb}
\usepackage{amsthm}

\usepackage[abbrev,lite,nobysame]{amsrefs}

\usepackage[margin=1in]{geometry}
\usepackage{graphicx}
\usepackage[dvipsnames]{xcolor}
\usepackage{enumitem}
\usepackage{comment}

\usepackage{tikz}
\usetikzlibrary{calc}

\usepackage[
  colorlinks=true,
  pdfstartview=FitV,
  linkcolor=BrickRed,
  citecolor=black,
  urlcolor=black
]{hyperref}

\renewcommand{\eprint}[1]{%
  \href{https://arxiv.org/abs/#1}{\nolinkurl{#1}}}

\hypersetup{
  pdftitle={Smooth autonomous fast dynamo action on the three-torus},
  pdfauthor={Michele Coti Zelati; Massimo Sorella; David Villringer}
}

\newtheorem{proposition}{Proposition}[section]
\newtheorem{lemma}[proposition]{Lemma}
\newtheorem{corollary}[proposition]{Corollary}
\newtheorem{theorem}[proposition]{Theorem}

\theoremstyle{definition}
\newtheorem{definition}[proposition]{Definition}

\newtheorem{remark}[proposition]{Remark}

\numberwithin{equation}{section}

\newcommand{\Id}{\mathrm{Id}}
\newcommand{\xiangle}{\langle \xi \rangle}
\newcommand{\cA}{\mathcal{A}}
\newcommand{\Op}{\text{Op}}
\newcommand{\cQ}{\mathcal{Q}}
\newcommand{\cR}{\mathcal{R}}
\newcommand{\cE}{\mathcal{E}}
\newcommand{\w}{{\tilde u_{1,2}}}

\makeatletter

\renewcommand{\subsubsection}{%
  \@startsection{subsubsection}{3}%
    {\normalparindent}%
    {.5\linespacing\@plus.7\linespacing}%
    {-.5em}%
    {\normalfont\bfseries}%
}

\def\@tocline#1#2#3#4#5#6#7{%
  \relax
  \ifnum #1>\c@tocdepth
  \else
    \par
    \addpenalty\@secpenalty
    \addvspace{#2}%
    \begingroup
      \hyphenpenalty\@M
      \@ifempty{#4}{%
        \@tempdima\csname r@tocindent\number#1\endcsname\relax
      }{%
        \@tempdima#4\relax
      }%
      \parindent\z@
      \leftskip#3\relax
      \advance\leftskip\@tempdima\relax
      \rightskip\@pnumwidth plus4em
      \parfillskip-\@pnumwidth
      #5\leavevmode
      \hskip-\@tempdima
      \ifcase #1
        \or
        \or \hskip1em
        \or \hskip2em
        \else \hskip3em
      \fi
      #6\nobreak\relax
      \dotfill
      \hbox to\@pnumwidth{\@tocpagenum{#7}}%
      \par\nobreak
    \endgroup
  \fi
}

\makeatother

\newcommand{\eps}{\varepsilon}
\newcommand{\e}{\mathrm{e}}
\newcommand{\dd}{\mathop{}\!\mathrm{d}}

\newcommand{\PP}{\Pi_{\rm div}}
\newcommand{\RR}{\mathbb{R}}
\newcommand{\NN}{\mathbb{N}}
\newcommand{\CC}{\mathbb{C}}
\newcommand{\TT}{\mathbb{T}}
\newcommand{\ZZ}{\mathbb{Z}}
\newcommand{\cL}{\mathcal{L}}
\newcommand{\cG}{\mathcal{G}}
\newcommand{\cU}{\mathcal{U}}
\newcommand{\cK}{\mathcal{K}}
\newcommand{\cM}{\mathcal{M}}
\newcommand{\cS}{\mathcal{S}}
\newcommand{\Av}{\Pi_0}

\def\Re{{\rm Re}}

\begin{document}
\title[Smooth autonomous fast dynamo action on the three-torus]{\vspace*{-3cm}Smooth autonomous fast dynamo action on the three-torus} 

\author[M. Coti Zelati]{Michele Coti Zelati}
\address{Department of Mathematics, Imperial College London}
\email{m.coti-zelati@imperial.ac.uk}

\author[M. Sorella]{Massimo Sorella}
\email{m.sorella@imperial.ac.uk}

\author[D. Villringer]{David Villringer}

\email{d.villringer22@imperial.ac.uk}



\begin{abstract}
We construct a nonempty $C^k$-open family of smooth, autonomous, divergence-free velocity fields on $\mathbb{T}^3$ that generate fast dynamos. The proof proceeds by first establishing fast dynamo action for a large class of smooth, time-periodic velocity fields exhibiting the classical stretch--fold--shear mechanism. We then realize the associated period map within a smooth autonomous flow using carefully tuned return dynamics to a global Poincar\'e section. In both settings, the proof first isolates unstable distributional eigenmodes for the ideal dynamo operator on an anisotropic Banach space of distributions, and then shows that this spectral instability persists under the singular perturbation $\eps \Delta$. The resulting estimates are uniform as $\eps\to0$ and stable under $C^k$ perturbations of the
velocity field, yielding an open set of smooth fast dynamo vector fields. This result resolves the Fast Dynamo Conjecture of Zeldovich and Sakharov, as recorded in Arnold's book of problems.
\end{abstract}

\maketitle

\tableofcontents

\section{The fast dynamo problem}
A fundamental equation in mathematical magnetohydrodynamics (MHD) is the \emph{kinematic dynamo equation}
\begin{equation}\tag{KDE}\label{passive-vector}
\begin{cases}
\partial_t B+ (u \cdot\nabla)B-(B \cdot\nabla)u=\eps \Delta B, \\[1mm]
\nabla \cdot B =0,
\end{cases}
\qquad (t,x)\in (0,\infty)\times \mathbb T^3.
\end{equation}

Here $u\in C^\infty(\TT^3\times [0,\infty))$ is a \emph{prescribed}, divergence-free velocity field, and $\eps>0$ is the magnetic resistivity. Equation~\eqref{passive-vector} describes the evolution of a weak magnetic field when its feedback on the fluid velocity is neglected. Equivalently, it is the magnetic component of the linearisation of the three-dimensional MHD equations around a (forced) Euler or Navier-Stokes solution with vanishing magnetic field.

The \emph{dynamo mechanism} goes back to Larmor \cite{larmor1919possible}, who proposed that fluid motion within the Sun could sustain its magnetic field. Since magnetic resistivity is extremely small in astrophysical regimes---for the Sun, it is estimated to be of order $\eps\sim10^{-6}$ (see \cite{hood2011solarmagneticfields})---the central mathematical question is whether instability persists uniformly as $\eps\to0$. For a fixed $\eps>0$, a velocity field $u$ is said to generate a \emph{kinematic dynamo} if \eqref{passive-vector} admits a nontrivial solution that grows exponentially in $L^2$. If the corresponding exponential growth rate remains bounded away from zero as $\eps\to0$, then $u$ is called a \emph{fast dynamo}. 

Proving even kinematic dynamo action at fixed positive resistivity is  nontrivial. Classical \emph{anti-dynamo theorems} exclude several settings with excessive symmetry: Zeldovich's theorem \cite{Zeldovich_1992} rules out planar flows, while Cowling's theorem \cite{Cowling33} rules out the maintenance of an axisymmetric magnetic field. Moreover, the results of Vishik and Klapper--Young \cites{Vishik89,Klapper_Young_1995} show that smooth fast dynamo action requires chaotic Lagrangian dynamics; in particular, the flow generated by a smooth autonomous fast dynamo must have positive topological entropy. For broader accounts of dynamo theory, we refer to the monographs \cites{CG95,AK98}.

These results leave open a conjecture originally posed by Ya.~B.~Zeldovich and A.~D.~Sakharov in the 1970s and later included in Arnold's collection of problems \cite{Arnold04}*{Problem~1994--28}.

\medskip

\noindent\textbf{Fast Dynamo Conjecture.}
\textit{There exist a smooth, time-independent, divergence-free velocity field $u:\TT^3\to\RR^3$ and constants $\eps_0,\gamma_0>0$ such that, for every $\eps\in(0,\eps_0]$, there exists a nonzero divergence-free initial datum $B_{\mathrm{in}}^\eps$ for which the corresponding solution of \eqref{passive-vector} satisfies}
\begin{equation}\label{eq:expoGro}
\liminf_{t\to\infty}
\frac1t
\log
\|B^\eps(t)\|_{L^2}
\geq \gamma_0.
\end{equation}
\textit{Such a velocity field is called a fast dynamo.}

\medskip

Our main result resolves this conjecture and, moreover, constructs a mechanism of smooth autonomous fast dynamo action that is robust under finite-regularity perturbations of the velocity field.

\begin{theorem}[Autonomous fast dynamo]
\label{thm:autonomous}
There exist $k\in\NN$ and a nonempty set
$\mathcal D\subset C^\infty_\sigma(\TT^3)$, open in the $C^k$
topology, such that every $u\in\mathcal D$ is a fast dynamo.
\end{theorem}
\begin{remark}
Although we have made no attempt at optimising this constant, a careful reading of the proof of Theorem \ref{thm:autonomous} reveals that $k=4$ suffices for the openness statement.
\end{remark}
\subsection{The ideal spectral strategy}
\label{sec:discussion of proof}
Our proof is guided by the heuristic, informed by classical notions in the theoretical physics literature \cite{CG95}, that fast dynamo action should leave behind a spectral ``remnant'' in the ideal equation, i.e.\ \eqref{passive-vector} with $\eps=0$. More precisely, we seek an unstable eigenmode for \eqref{passive-vector} with $\eps=0$, and only afterwards ask whether this mode survives the addition of diffusion. 
However, such a limiting mode should not generally be expected to remain smooth: repeated stretching and mixing transfer the magnetic field to increasingly fine scales, suggesting that the natural ideal object is a distribution rather than an $L^2$ function \cite{Moffatt_Proctor_1985}. Hence, our goal will be to construct a space of distributions $\mathcal{X}$ that not only contains the limiting eigenmode for the ideal dynamo problem, but also provides \eqref{passive-vector} with sufficiently robust spectral structure to control the singular perturbation of adding $\eps \Delta$.

We implement this principle through two distinct but related constructions. For a time-periodic velocity field, the evolution over one complete period provides a natural discrete operator, and our first task is to isolate an expanding mode of this operator. In order to do so, we develop a general anisotropic framework to deal with time-periodic \emph{stretch--fold--shear} fast dynamos, first introduced in \cite{baylychildress}. Similarly, although an autonomous flow has no distinguished global period, a two-dimensional global transverse section $ \Sigma  \subset \TT^3$ equips it with a period-like return structure: each trajectory is decomposed into successive passages between crossings of the section, with a generally non-constant time assigned to each passage. Using this structure, for an autonomous vector field with strictly positive third component, we evaluate the dynamo evolution at the return time to the \emph{Poincar\'e section} $\Sigma = \{z=0\}$. This yields an operator that can be treated using our stretch–fold–shear framework.

\subsection{Time-periodic dynamos} 
\label{sec:intro time periodic}
As in the recent work \cite{CZSV26}, the time-periodic construction is based on the classical \emph{stretch--fold--shear} (SFS) mechanism \cites{baylychildress,ChildressLargeSFS,Gilbert1993}, which we study in a rather general setting. Let $f_1,f_2\in C^\infty(\TT;\RR)$ be nonconstant mean-free functions, let $g\in C^\infty(\TT^2;\RR)$, and let $N\geq1$ be an integer. Consider the idealized velocity field
\begin{equation}\label{eq:velocity}
u_N(t,x,y,z)=
\begin{cases}
(0,2Nf_1(x),0), & t\in[0,\frac{1}{2}),
\\[1mm]
(2Nf_2(y),0,0), & t\in[\frac{1}{2},1),
\\[1mm]
(0,0,-g(x,y)), & t\in[1,2),
\end{cases}
\end{equation}
extended periodically in time with period $2$. We remark that the field \eqref{eq:velocity} is divergence-free and smooth in space, and its switching times can be smoothed without changing the essential features of the construction.

The first two stages of \eqref{eq:velocity} are strong horizontal shears. Their time-$1$ flow map, denoted by $T_N:\TT^2\to\TT^2$, has derivative $DT_N$ with a term of order $N^2$, which provides the dominant stretching mechanism. The third stage is a vertical shear. On the first nonzero vertical Fourier mode it produces the phase factor $\e^{2\pi i g}$, so that the evolution of the first two components of the magnetic field over one period is described by \begin{equation}\label{eq:periodic-dynamo-operator} 
\mathcal M_{T_N,g}b = \e^{2\pi i g}(DT_N\,b)\circ T_N^{-1}. \end{equation} 
Thus the horizontal shears provide stretching, while the vertical shear controls the interference created by the rapidly oscillating composition with $T_N^{-1}$. This construction yields the following result, which is also of independent interest.

\begin{theorem}[Time-periodic SFS fast dynamos] \label{thm:time-periodic} 
Let $f_1,f_2\in C^\infty(\TT;\RR)$ be nonconstant mean-free functions. Then there exist $g\in C^\infty(\TT^2;\RR)$ and $N_0\geq1$ such that, for every integer $N\geq N_0$, the SFS construction \eqref{eq:velocity} admits a smooth, $2$-periodic, divergence-free realization 
$$
u_N^{\rm per} \in C^\infty_{\rm per}(\RR\times\TT^3;\RR^3) $$
which is a fast dynamo. 
\end{theorem}
Several properties we derive for the operator in \eqref{eq:periodic-dynamo-operator} actually hold for a general map $T:\TT^2\to \TT^2$. The specific choice of $T_N$, arising from $u_N$ in \eqref{eq:velocity}, is particularly amenable to constructing an unstable eigenvalue in the strong-chaos limit $N\to\infty$. In principle, our methodology can be employed for a plethora of measure-preserving diffeomorphisms on $\TT^2$. See Section \ref{sec:strategy} for more details.

\subsection{From periodic to autonomous dynamos}
As hinted at before, the proof of Theorem \ref{thm:time-periodic} follows from a rather general anisotropic Banach space framework for SFS-type velocity fields. Hence, the passage to an autonomous velocity field requires identifying such ``SFS-type'' dynamics. To achieve this, we construct an autonomous velocity field whose vertical component is strictly positive. Consequently, every trajectory repeatedly crosses the Poincar\'e section $\Sigma= \{ z=0\}$, always in the same direction. Following a trajectory from one crossing to the next produces two pieces of data: the point at which it next crosses the Poincar\'e section and the time taken to reach it. These are respectively the \emph{Poincaré return map} and the \emph{return-time function}. The distinction between these two objects is the key to the autonomous construction--- the return map and return time should be thought of as, respectively, the analogues of $T_N$ and $g$ in the time-periodic construction.
This use of a return-time function is closely related to the theory of twisted transfer operators \cites{Pollicott1985,Dolgopyat1998Prevalence,Dolgopyat1998Decay,BaladiVallee2005}, but its role here is reversed. In mixing problems, a large imaginary spectral parameter is combined with variation of the roof function to produce \emph{destructive} interference. Here the roof function is deliberately chosen so that the same high-frequency mechanism produces \emph{constructive} interference and recovers the phase required by the periodic dynamo operator.
\subsection{Restoring resistivity and obtaining openness}
At this stage, both constructions provide an isolated growing distributional mode for the ideal equation. It remains to restore resistivity. This cannot be treated by ordinary $L^2$ perturbation theory: diffusion is an unbounded operator, and the limiting eigenmode is itself distributional, see also \cite{Chicone_Latushkin_Montgomery-Smith_1995}. Instead, following the approach introduced in \cite{CZSV26}, we compare the ideal and resistive evolutions in a pair of strong and weak anisotropic topologies. Uniform estimates in the strong space, together with convergence in the weak one, allow the isolated ideal spectral data to persist for all sufficiently small $\eps>0$.
An analogous construction also explains the openness statement in Theorem~\ref{thm:autonomous}. A perturbation in the underlying velocity field behaves not dissimilarly from a perturbation in the resistivity, failing to be continuous in natural strong topologies, but being controllable in the weak-strong topology. Hence, the same philosophy that allows us to control the diffusive perturbation also ensures that our construction is robust under perturbations in the underlying velocity field.

\subsection{Comparison with existing literature}

The rigorous theory of fast dynamo action has developed rapidly in recent years. Among the principal advances are our construction of an autonomous Lipschitz fast dynamo on $\RR^3$ \cite{CZSV25}, the time-periodic Lipschitz fast dynamo on $\TT^3$ constructed in our previous work \cite{CZSV26}, the autonomous Lipschitz fast dynamo on $\TT^3$ of \cite{niebel2026autonomouslipschitzfastdynamo}, and the smooth random construction on $\TT^3$ of \cite{rowan2026aidiscoveredsmoothrandomfast}. Related progress includes weaker subsequential or limsup-in-time forms of fast dynamo action \cites{RowanDynamo25,sorellaVillringer2025}, as well as their ``almost sure'' counterparts \cite{delnin2026turbulentdynamosboundeddomains} established via convex integration techniques, slow dynamo instability and its nonlinear consequences for MHD \cites{NF_DV2025spectral,NF_DV2025nonlinear}, and results concerning the growth and decay of particular solutions to the ideal dynamo equation \cite{navarrofernandez2026exponentialgrowthdecayideal}. Within the well-posed Lipschitz or smooth setting, three broad approaches have emerged: rescaling, probabilistic averaging, and singular spectral perturbation.

\subsubsection*{The rescaling method}
The rescaling method was introduced in our work \cite{CZSV25}. On $\RR^3$, if $B(t,x)$ solves \eqref{passive-vector} with resistivity $\eps=1$ and velocity field $U$, then
$B^\eps(t,x)=B(t,x/\sqrt{\eps})$ solves \eqref{passive-vector} with resistivity $\eps$ and velocity field $u^\eps(x)=\sqrt{\eps}\,
U(x/\sqrt{\eps})$. On the torus, the same identity is compatible with the periodic geometry along a discrete sequence of scales. Since the transformation $U\mapsto u^\eps$ preserves the Lipschitz seminorm, a dynamo at one positive resistivity provides building blocks at arbitrarily small resistive scales. In \cite{CZSV25}, spatially separated copies of flows exhibiting the alpha effect, related to the classical example of \cite{Roberts70}, are combined to construct an autonomous fast dynamo in $W^{1,\infty}(\RR^3)$.

The same principle underlies two subsequent constructions. In \cite{sorellaVillringer2025}, spatial separation is replaced by temporal separation, producing $u\in L^\infty_tW^{1,\infty}_x(\TT^3)$ for which 
$$
\liminf_{\eps\to0}\limsup_{t\to\infty}
\frac1t\log\|B^\eps(t)\|_{L^2(\TT^3)}>0.
$$
This gives a limsup-in-time form of fast dynamo action. More recently, \cite{niebel2026autonomouslipschitzfastdynamo} combined rescaling with a spectral gluing argument, using the velocity field of \cite{Gilbert88} as a building block, to obtain an autonomous fast dynamo in $W^{1,\infty}(\TT^3)$ in the full sense of equation \eqref{eq:expoGro}.

The currently available rescaling constructions are Lipschitz, but not smoother. Indeed, although the rescaling is bounded in $W^{1,\infty}$, its $C^{1,\alpha}$ seminorm grows like $\eps^{-\alpha/2}$ for every $\alpha>0$. Thus, the existing multiscale gluing schemes do not provide uniform control in any $C^{1,\alpha}$ space. This distinction is particularly relevant in view of \cites{Vishik89,Klapper_Young_1995}, which connect smooth fast dynamo action to dynamical complexity of the ideal flow. The example of \cite{niebel2026autonomouslipschitzfastdynamo} shows that their conclusions do not extend verbatim to the Lipschitz category.

\subsubsection*{The probabilistic approach}
A second strategy uses randomness to obtain averaged stretching estimates for which the effect of diffusion is particularly transparent in Fourier variables; see \cite{CG95} for the classical framework. Rowan \cite{RowanDynamo25} used this perspective to construct a time-dependent velocity field with smooth temporal dependence and high spatial regularity satisfying
$$
\limsup_{\eps\to0}\limsup_{t\to\infty}
\frac1t\log\|B^\eps(t)\|_{L^2(\TT^3)}>0.
$$

Compared with \cite{sorellaVillringer2025}, this comes with greater spatial and temporal regularity, but a weaker notion of limit in $\eps$.

More recently, \cite{rowan2026aidiscoveredsmoothrandomfast} constructed a random velocity field $u\in C^\infty_tC^\infty_x(\TT^3)$ such that, for every fixed sufficiently small $\eps>0$, exponential growth at an $\eps$-independent rate holds for almost every realization. The exceptional null set may depend on $\eps$, so the result does not assert that, with probability one, a single realization works simultaneously for every sufficiently small resistivity. This order of quantifiers is familiar in the theory of mixing by random velocity fields; see, for instance, \cite{BBPS21}. Nevertheless, the result provides the first smooth construction exhibiting a uniform fast-dynamo growth rate through a probabilistic mechanism.

\subsubsection*{The singular-perturbative approach}
The third approach, adopted in the present paper, is the \emph{singular-perturbative approach} introduced in our previous work \cite{CZSV26}. It is motivated by the heuristic of Moffatt and Proctor \cite{Moffatt_Proctor_1985} that growing resistive eigenmodes should, after suitable normalization, possess a zero-resistivity limit in a generalized sense. In the setting relevant here, the ideal evolution on $L^2$ does not provide the isolated unstable spectral data required by classical perturbation theory; see \cites{Moffatt_Proctor_1985,Chicone_Latushkin_Montgomery-Smith_1995}. One is therefore led to seek such spectral data on suitable spaces of distributions.

In \cite{CZSV26}, we gave the first rigorous implementation of this principle for the dynamo problem. For a time-periodic Lipschitz velocity field, anisotropic spaces related to those of Demers and Liverani \cite{demers_liverani} were used to isolate an unstable distributional eigenmode of the ideal evolution. The Keller--Liverani perturbation theorem \cite{Keller_Liverani} then showed that the associated spectral data persist after the addition of sufficiently small resistivity, yielding a time-periodic fast dynamo in $L^\infty_tW^{1,\infty}_x(\TT^3)$. The principal contribution of \cite{CZSV26} was therefore to turn the Moffatt--Proctor heuristic into a workable two-stage scheme: first establish spectral instability for the ideal dynamics on a space of distributions, and only afterwards introduce diffusion.

This spectral scheme provides the conceptual and perturbative starting point of the present paper. The conclusions obtained here, however, do not follow by simply regularizing or suspending the construction of \cite{CZSV26}. That work concerns a time-periodic Lipschitz velocity field and uses an anisotropic functional framework inspired by ideas from uniformly hyperbolic dynamics, see e.g.\ \cites{Baladi_quest,BaladiGouezel2010PiecewiseConeHyperbolic,demers_liverani,zworski_2015}. While the present paper is heavily inspired by the anisotropic Banach construction of \cite{CZSV26}, the precise mechanism we exploit is global, rather than the local expansion and contraction of hyperbolic dynamics. The current paper also addresses three features absent from \cite{CZSV26}: smoothness, autonomy, and stability under perturbations of the velocity field. Each requires a new dynamical and analytic construction, as described in the preceding Section \ref{sec:discussion of proof}. Thus, \cite{CZSV26} supplies the spectral blueprint, while its realization for a $C^k$-open family of smooth autonomous flows is the principal advance of the present work.

A recurrent feature of rigorous fast-dynamo constructions---recently termed ``bespokeness'' in \cite{rowan2026aidiscoveredsmoothrandomfast}---is their reliance on velocity fields designed around exact algebraic, geometric, or averaging identities. This applies to the rescaling constructions \cites{CZSV25,sorellaVillringer2025,niebel2026autonomouslipschitzfastdynamo}, the probabilistic constructions \cites{RowanDynamo25,rowan2026aidiscoveredsmoothrandomfast}, and also to the seed flow constructed here. Theorem~\ref{thm:autonomous}, however, shows that the resulting dynamo property is not confined to this particular example: the exact identities are needed to produce one velocity field satisfying the relevant spectral criterion, but nearby velocity fields need not preserve them.

For each fixed $\eps>0$, standard perturbation theory gives stability of isolated spectral data under sufficiently small perturbations of the velocity field \cite{K76}. The corresponding neighbourhood may, however, degenerate as $\eps\to0$. The present argument instead works with an ideal, diffusivity-independent spectral criterion which is stable under $C^k$ perturbations and from which resistive growth can be recovered uniformly as $\eps\to0$. In this sense, the openness result develops the singular-perturbative program of \cite{CZSV26} in a genuinely new direction.

\section{Main ideas and outline of the proof}\label{sec:strategy}
We now develop the quantitative framework behind the strategy summarized in Section~\ref{sec:discussion of proof}, introducing the operators, anisotropic spaces, and perturbative estimates used in the proof. We begin with the stochastic Lagrangian representation of the kinematic dynamo equation.

\subsection{The Lagrangian formulation}

As a diffusive Lie-transport equation, \eqref{passive-vector} admits a
stochastic Lagrangian representation. Given a divergence-free velocity
field $u(t,x)$ on $[0,\infty)\times\TT^d$, let
$\Phi_{u,\eps}^t$ denote the stochastic flow solving
\begin{equation}
\label{eq:stochastic flow}
\dd\Phi_{u,\eps}^t(\omega,x)
=
u\bigl(t,\Phi_{u,\eps}^t(\omega,x)\bigr)\,\dd t
+\sqrt{2\eps}\,\dd W_t(\omega),
\qquad
\Phi_{u,\eps}^0(\omega,x)=x,
\end{equation}
where $(W_t)_{t\geq0}$ is a standard $d$-dimensional Brownian motion.
It\^o's formula gives the stochastic Cauchy formula
\cite{Eyink2009}
\begin{equation}
\label{eq:flow representation}
B(t,\cdot)=\mathbb E\left[
\bigl(D\Phi_{u,\eps}^t(\omega,\cdot)\,B_0\bigr)
\circ \bigl(\Phi_{u,\eps}^t(\omega,\cdot)\bigr)^{-1}
\right].
\end{equation}
For $\eps=0$, we denote by $\Phi_u^t$ the deterministic flow. In this case,
\eqref{eq:flow representation} remains valid with the expectation omitted.

The proofs of Theorems~\ref{thm:autonomous} and
\ref{thm:time-periodic} begin by constructing unstable eigenvalues of the ideal dynamo operators whose associated eigenfunctions are distributions. We therefore first consider \eqref{eq:flow representation} with $\eps=0$.
The two-dimensional ideal dynamics will play a central role. Given a smooth volume-preserving diffeomorphism $T:\TT^2\to\TT^2$, arising as the time-$1$ flow map of a divergence-free velocity field, define
\begin{equation}
\label{eq:general 2d dynamo operator}
\mathcal M_T b:=(DT\,b)\circ T^{-1} \,, \qquad \cM_{T,g} b : =\e^{2\pi i g} \cM_T b \,,
\qquad
b\in L^2(\TT^2;\CC^2).
\end{equation}
We refer to $\mathcal M_T$ as the \emph{two-dimensional ideal dynamo operator}. The central observation is that these operators satisfy exact algebraic identities which separate the divergence-free and gradient components of a vector field. This allows the two components to be measured at different Sobolev regularities and leads to the anisotropic Banach spaces used below.

\subsection{The dynamo problem in two dimensions}
Somewhat paradoxically, the two-dimensional structure relevant to our construction is most clearly revealed by Zeldovich's anti-dynamo theorem, which rules out resistive dynamo action in dimension $2$.

\begin{theorem}[Zeldovich's anti-dynamo theorem
\cite{Zeldovich_1992}]
\label{thm:zeldovich}
Let $\eps>0$ and let $u\in C^\infty_\sigma(\TT^2)$. Consider
\begin{equation}
\label{eq:2d dynamo zeldovich}
\partial_t b+u\cdot\nabla b-b\cdot\nabla u=\eps\Delta b.
\end{equation}
Then, for every $b_0\in L^2(\TT^2;\CC^2)$, not necessarily
divergence-free, the corresponding solution remains bounded in $L^2$ as $t \to \infty$.
\end{theorem}
We recall only the structural identities underlying the theorem, as these will also be central to our construction. By the Hodge decomposition, write $b=\overline b+\nabla^\perp\psi+\nabla\phi$, where $\overline b$ denotes the spatial average of $b$. The gradient
component is determined by the divergence through
$$
\phi=\Delta^{-1}\nabla\cdot b.
$$
Taking the divergence of \eqref{eq:2d dynamo zeldovich} gives
$$
\partial_t(\nabla\cdot b) +u\cdot\nabla(\nabla\cdot b) =\eps\Delta(\nabla\cdot b).
$$
Thus, $\nabla\cdot b$ evolves autonomously according to the passive
scalar equation. Moreover, a direct computation shows that, if
$\widetilde\psi$ solves
\begin{equation}
\partial_t\widetilde\psi
+u\cdot\nabla\widetilde\psi
=\eps\Delta\widetilde\psi,
\end{equation}
then $\nabla^\perp\widetilde\psi$ solves
\eqref{eq:2d dynamo zeldovich}. For general initial data, the Hodge components do not completely decouple: the divergence determines the gradient component, which in turn forces the divergence-free component. The resulting system is nevertheless triangular. Together with the exponential decay of mean-free passive scalars when
$\eps>0$, this structure yields the boundedness asserted in
Theorem~\ref{thm:zeldovich}.

At the level of the ideal dynamo operator $\mathcal M_T$, the same structure is encoded by the exact identities
\begin{equation}
\label{eq:zeldovich identities}
(DT\,\nabla^\perp\psi)\circ T^{-1}
=\nabla^\perp(\psi\circ T^{-1}),\qquad
\nabla\cdot\bigl((DT\,b)\circ T^{-1}\bigr)
=(\nabla\cdot b)\circ T^{-1}.
\end{equation}
Thus, although $\mathcal M_T$ contains the potentially large
stretching factor $DT$, its action on the two diagonal components of the Hodge decomposition is governed by scalar composition with $T^{-1}$, i.e. by solutions to the passive scalar equation
\begin{equation}
\label{eq:passive scalar}
\partial_t f+u \cdot \nabla f=\eps \Delta f.
\end{equation}
For $\eps>0$, Zeldovich's theorem exploits the diffusive decay of the corresponding passive scalars. Such decay is unavailable in the ideal limit $\eps=0$. Instead, we use \eqref{eq:zeldovich identities} to measure the divergence-free and gradient components in Sobolev spaces of opposite regularities. Indeed, this becomes particularly effective for the
strong-shear maps $T_N$ of Section \ref{sec:intro time periodic}, see Section \ref{sub:SFS_map_strategy}.

\subsection{The Lasota--Yorke framework}

As in \cite{CZSV26}, our main tool for controlling the essential
spectral radius is a \emph{Lasota--Yorke inequality}. We first recall
the abstract framework.

\begin{proposition}[\cite{Hennion}]
\label{prop:hennion}
Let $X$ and $Y$ be Banach spaces such that $X$ embeds continuously and
compactly into $Y$. Let $\cL$ be a bounded operator on $X$ which
extends to a bounded operator on $Y$. Suppose that there exist
constants $C>0$ and $0<\alpha<M$ such that, for every $m\in\NN$,
\begin{align}
\|\cL^m b\|_X
&\leq
C\alpha^m\|b\|_X+CM^m\|b\|_Y,
\qquad b\in X,
\\
\|\cL^m b\|_Y
&\leq
CM^m\|b\|_Y,
\qquad b\in Y.
\end{align}
Then the essential spectral radius of $\cL$ on $X$ satisfies
$r_{\mathrm{ess},X}(\cL)\leq\alpha$.
\end{proposition}

Consequently, every spectral point of $\cL$ on $X$ with modulus strictly larger than $\alpha$ is an isolated eigenvalue of finite algebraic multiplicity; see
\cite{K76}*{Chapter~IV, Section~6}. Proposition~\ref{prop:hennion} therefore gives a precise target for the construction of our anisotropic spaces: the strong Lasota--Yorke scale must be strictly smaller than the scale on which we construct an eigenvalue.

\subsection{The anisotropic spaces and the horizontal Lasota--Yorke estimate}
Let $T:\TT^2\to\TT^2$ be a smooth, volume-preserving diffeomorphism
with $\det DT=1$. Since composition with $T^{-1}$ is unitary on $L^2$, one has $\|\mathcal M_T\|_{L^2\to L^2}=\|DT\|_{L^\infty}$.
Moreover, in dimension two,
$\|DT^{-1}\|_{L^\infty}=\|DT\|_{L^\infty}$, and hence, for
$|s|\leq1$,
\begin{equation}
\label{eq:composition estimate strategy}
\|f\circ T^{-1}\|_{H^s}
\lesssim_s
\|DT\|_{L^\infty}^{|s|}\|f\|_{H^s}.
\end{equation}
The identities \eqref{eq:zeldovich identities} allow us to exploit
this smaller composition scale separately on the two components of
the Hodge decomposition.
Fix
\begin{equation} \label{parameter:r-delta}
0<r<\frac12,
\qquad
0<\delta<
\min\left\{
r,\frac12\left(\frac12-r\right)
\right\}.
\end{equation}
We define the \emph{strong horizontal anisotropic space} $X$ as the
completion of $C^\infty(\TT^2;\CC^2)$ with respect to
\begin{equation}
\label{eq:horizontal strong norm}
\|b\|_X^2
=|\overline b|^2
+\|\PP(b-\overline b)\|_{H^{-r}}^2
+\|(\Id-\PP)(b-\overline b)\|_{H^r}^2,
\end{equation}
where
\begin{align} \label{eq:Leray-def}
    \overline b:=\int_{\TT^2}b(x,y)\,\dd x\,\dd y,
\qquad
\PP:=\Id-\nabla\Delta^{-1}\nabla\cdot
\end{align}
is the Leray projector. Similarly, the \emph{weak horizontal
anisotropic space} $Y$ is the completion of
$C^\infty(\TT^2;\CC^2)$ with respect to
\begin{equation}
\label{eq:horizontal weak norm}
\|b\|_Y^2
=|\overline b|^2
+\|\PP(b-\overline b)\|_{H^{-r-\delta}}^2
+\|(\Id-\PP)(b-\overline b)\|_{H^{r-\delta}}^2.
\end{equation}
By Rellich's theorem, the embedding $X\hookrightarrow Y$ is compact.

For later use with the autonomous section operators, we also introduce
three-component versions of these spaces. Given
$B\in C^\infty(\TT^2;\CC^3)$, write
\[
\overline B
:=
\int_{\TT^2}B(x,y)\,\dd x\,\dd y,
\qquad
B=\overline B+(b,b_v),
\]
where 
\[
b
=
(B_1-\overline B_1,B_2-\overline B_2),
\qquad
b_v=B_3-\overline B_3.
\]
We define
\begin{equation}
\label{eq:strong norm}
\|B\|_{\mathcal X}^2
=|\overline B|^2
+\|\PP b\|_{H^{-r}}^2
+\|(\Id-\PP)b\|_{H^r}^2
+\|b_v\|_{H^{-r-\delta}}^2
\end{equation}
and
\begin{equation}
\label{eq:weak norm}
\|B\|_{\mathcal Y}^2
=|\overline B|^2
+\|\PP b\|_{H^{-r-\delta}}^2
+\|(\Id-\PP)b\|_{H^{r-\delta}}^2
+\|b_v\|_{H^{-r-2\delta}}^2.
\end{equation}
We denote by $\mathcal X$ and $\mathcal Y$ the corresponding completions of $C^\infty(\TT^2;\CC^3)$. Again, Rellich's theorem gives the compact embedding $\mathcal X\hookrightarrow\mathcal Y$.

We now explain how the identities \eqref{eq:zeldovich identities}
lead to the required Lasota--Yorke estimate. Write
\[
b
=
\overline b+\nabla^\perp\psi+\nabla\phi,
\]
where $\psi$ and $\phi$ are mean-free. Then
\[
\PP(b-\overline b)=\nabla^\perp\psi,
\qquad
(\Id-\PP)(b-\overline b)
=
\nabla\phi
=
\nabla\Delta^{-1}\nabla\cdot b.
\]
The first identity in \eqref{eq:zeldovich identities}, together with
\eqref{eq:composition estimate strategy}, gives
\begin{align}
\|\PP\mathcal M_T\PP(b-\overline b)\|_{H^{-r}}
=
\|\nabla^\perp(\psi\circ T^{-1})\|_{H^{-r}}
\lesssim_r
\|DT\|_{L^\infty}^{1-r}
\|\PP(b-\overline b)\|_{H^{-r}}.
\end{align}
Similarly, the second identity in
\eqref{eq:zeldovich identities} yields
\begin{equation}
\|(\Id-\PP)\mathcal M_T(b-\overline b)\|_{H^r}
\lesssim_r
\|DT\|_{L^\infty}^{1-r}
\|(\Id-\PP)(b-\overline b)\|_{H^r}.
\end{equation}
Thus the diagonal components of the Hodge decomposition grow at the
strong scale $\|DT\|_{L^\infty}^{1-r}$, rather than at the full
stretching scale $\|DT\|_{L^\infty}$.

For the untwisted operator $\mathcal M_T$, there is no coupling from
the divergence-free component to the gradient component. There are,
however, contributions to the divergence-free component from the
gradient and mean components, as well as contributions to the spatial
mean. These terms may have size $\|DT\|_{L^\infty}$, but they can be
estimated in the weaker norm.

Multiplication by the phase $\e^{2\pi i g}$ introduces an additional
coupling from the divergence-free component to the gradient component.
For a divergence-free vector field $v$,
\begin{equation}
(\Id-\PP)\bigl(\e^{2\pi i g}v\bigr)
=
[\Id-\PP,\e^{2\pi i g}]v.
\end{equation}
The commutator is an operator of order $-1$. Since
$2r+\delta<1$, it satisfies
\begin{equation}
\|[\Id-\PP,\e^{2\pi i g}]v\|_{H^r}
\lesssim_{g,r,\delta}
\|v\|_{H^{-r-\delta}}.
\end{equation}
Moreover, since the phase is of modulus $1$, multiplication by
$\e^{2\pi i g}$ preserves the principal part of the relevant Sobolev
norms; its commutators with the Sobolev multipliers are of lower order
and can again be estimated in the weak norm. Consequently, all
dependence on $g$ enters the weak term of the Lasota--Yorke estimate.

We shall prove the following consequence of these observations.

\begin{proposition}
\label{prop:horizontal Lasota-yorke}
Let $T:\TT^2\to\TT^2$ be a smooth, volume-preserving diffeomorphism,
and let $g\in C^\infty(\TT^2;\RR)$. Then
$\e^{2\pi i g}\mathcal M_T$ extends to a bounded operator on $X$ and
$Y$. Moreover, there exist constants $c_r>0$ and
$C_{g,r,\delta}>0$, independent of $T$, such that
\begin{align*}
\|\e^{2\pi i g}\mathcal M_Tb\|_X
&\leq
c_r\|DT\|_{L^\infty}^{1-r}\|b\|_X
+
C_{g,r,\delta}\|DT\|_{L^\infty}\|b\|_Y,
\\
\|\e^{2\pi i g}\mathcal M_Tb\|_Y
&\leq
C_{g,r,\delta}\|DT\|_{L^\infty}\|b\|_Y.
\end{align*}
\end{proposition}

This follows from Proposition~\ref{prop:LY}. Iterating the preceding
inequalities and applying Proposition~\ref{prop:hennion} gives
\begin{equation}
r_{\mathrm{ess},X}
\bigl(\e^{2\pi i g}\mathcal M_T\bigr)
\leq
c_r\|DT\|_{L^\infty}^{1-r}.
\end{equation}

\subsection{The stretch--fold--shear construction in the strong chaos limit}
\label{sub:SFS_map_strategy}
We now quantify the stretch--fold--shear mechanism described in Section~\ref{sec:discussion of proof}, following the spectral strategy introduced in \cite{CZSV26}. Let $u_N$ be the idealized velocity field defined in \eqref{eq:velocity}. The time-$1$ horizontal flow map is
\begin{equation}
\label{eq:TN}
T_N(x,y)=\Bigl(
x+Nf_2\bigl(y+Nf_1(x)\bigr),
y+Nf_1(x)
\Bigr),
\end{equation}
and its derivative is given explicitly by
\begin{equation}
\label{eq:DTN}
DT_N(x,y)
=
\begin{pmatrix}
1+N^2f_1'(x)f_2'\bigl(y+Nf_1(x)\bigr)
&
Nf_2'\bigl(y+Nf_1(x)\bigr)
\\[1mm]
Nf_1'(x)
&
1
\end{pmatrix}.
\end{equation}
In particular, $\det DT_N=1$, so $T_N$ preserves Lebesgue measure on $\TT^2$. Since $f_1$ and
$f_2$ are nonconstant, $\|DT_N\|_{L^\infty}\sim N^2$.
On the first nonzero vertical Fourier mode, the first two components of the magnetic field evolve over one period according to
$$
\mathcal M_{T_N,g}b =\e^{2\pi i g}(DT_N\,b)\circ T_N^{-1},
$$
introduced in \eqref{eq:periodic-dynamo-operator}, and so Proposition~\ref{prop:horizontal Lasota-yorke} gives
\begin{align}
\|\mathcal M_{T_N,g}b\|_X 
\leq c_rN^{2-2r}\|b\|_X +C_{g,r,\delta}N^2\|b\|_Y,\qquad
\|\mathcal M_{T_N,g}b\|_Y 
\leq C_{g,r,\delta}N^2\|b\|_Y.
\end{align}
After normalization by $N^{-2}$, the corresponding iterated estimates and Proposition~\ref{prop:hennion} therefore imply
$r_{\mathrm{ess},X}
\bigl(N^{-2}\mathcal M_{T_N,g}\bigr)
\lesssim N^{-2r}$.
In particular, the essential spectral radius of the normalized operator converges to zero as $N\to\infty$.

It remains to identify spectral data which survive this limit. For the alternating shears in \eqref{eq:velocity}, the rapidly oscillating composition with $T_N^{-1}$ eliminates all but one leading-order contribution. Although $DT_N$ also depends on $N$, its leading $N^2$ term has a simple product structure, and the resulting normalized operators converge to an explicitly computable rank-one operator.

\begin{proposition}
\label{prop:rank 1 limit simple}
Let $u_N$ be given by \eqref{eq:velocity}, and define
$$
\ell(b)
:=
\int_{\TT^2}f_1'(x)b_1(x,y)\,\dd x\,\dd y,
\qquad
h_g(x,y)
:=
\e^{2\pi i g(x,y)}
\begin{pmatrix}
f_2'(y)\\
0
\end{pmatrix}.
$$
Then, as $N \to \infty$, $N^{-2}\mathcal{M}_{T_N,g}$ converges in the $X \to Y$ operator norm to the rank-one operator
$$
\mathcal M_\infty b:=\e^{2\pi i g}\mathcal L_\infty b=\ell(b)h_g.
$$
The only possible nonzero eigenvalue of $\mathcal M_\infty$ is
$$
\lambda_g
:=
\ell(h_g)
=
\int_{\TT^2}
\e^{2\pi i g(x,y)}
f_1'(x)f_2'(y)\,\dd x\,\dd y.
$$
If $\lambda_g\neq0$, then it is an algebraically simple eigenvalue of
$\mathcal M_\infty$. Moreover, for every pair of nonconstant functions
$f_1,f_2$, there exists $g\in C^\infty(\TT^2;\RR)$ such that
$\lambda_g\neq0$.
\end{proposition}

\begin{proof}
We only prove the statement regarding the (nonzero) eigenvalue.
Since $\mathcal M_\infty$ has range
$\operatorname{span}\{h_g\}$, one has
$$
\mathcal M_\infty h_g
=
\ell(h_g)h_g
=
\lambda_g h_g.
$$
Thus, whenever $\lambda_g\neq0$, it is the unique nonzero eigenvalue of
$\mathcal M_\infty$ and is algebraically simple.
It remains to choose $g$ so that $\lambda_g\neq0$. For $t\in\RR$, set
$$
g_t(x,y):=t f_1'(x)f_2'(y).
$$
Since $f_1$ and $f_2$ are periodic,
$$
\lambda_{g_0}
=
\int_{\TT^2}f_1'(x)f_2'(y)\,\dd x\,\dd y
=0.
$$
On the other hand,
$$
\left.\frac{\dd}{\dd t}\right|_{t=0}\lambda_{g_t}
=
\left.\frac{\dd}{\dd t}\right|_{t=0}
\int_{\TT^2}\e^{2\pi i t f_1'(x)f_2'(y)}f_1'(x)f_2'(y)\,\dd x\,\dd y 
=2\pi i\int_{\TT^2}\bigl(f_1'(x)\bigr)^2\bigl(f_2'(y)\bigr)^2
\,\dd x\,\dd y
\neq0.
$$
Consequently, $\lambda_{g_t}\neq0$ for all sufficiently small nonzero
$t$, which completes the proof.
\end{proof}
Hence, since $\mathcal M_\infty$ has a non-zero eigenvalue for an appropriate choice of $g$, in view of the Keller--Liverani Theorem \ref{thm:keller-liverani 1} the same holds for $N^{-2} \mathcal{M}_{T_N,g}$ for all $N$ large enough. Thus, taking $N$ large, it holds that $\mathcal{M}_{T_N,g}$ must have an isolated eigenvalue of modulus strictly larger than $1$, yielding the desired growing mode.

\subsection{The autonomous framework} \label{subsec:autonomous-framework}
 Having constructed an ideal growing mode for the period map associated with the time-periodic velocity field $u_N$, we now encode the time-periodic dynamics in an autonomous flow. We restrict our attention to velocity fields
$u\in C^\infty_\sigma(\TT^3)$ satisfying
\begin{equation}
u_3(x,y,z)>0 \qquad \text{for every }(x,y,z)\in\TT^3.
\end{equation}
To distinguish successive crossings of a horizontal Poincaré section $\{ z=0\}$, we regard $u$
as a $\ZZ$-periodic vector field on $\TT^2\times [0,\infty)$ and denote its (lifted)
flow by $\Phi_u^t$. We start by studying the eigenvalue problem 
\begin{equation}
\label{eq:inviscid eigenvalue relation}
u\cdot\nabla B-B\cdot\nabla u+\lambda B=0\,.
\end{equation}
Applying the solution formula \eqref{eq:flow representation} for the dynamo equation, we obtain, for any $t \geq 0$,
\begin{align} \label{eq:formula-explicit-eigen}
    B(\Phi_u^t)=\e^{-\lambda t}D_x\Phi_u^t B \,.
\end{align} 
Write $u=(u_{1,2},u_3)$ and $a=(x,y)\in\TT^2$, and use $s\in[0,\infty)$ in place of $z$. Since $u_3>0$, dividing \eqref{eq:inviscid eigenvalue relation} by $u_3$ gives the following evolution PDE: 
$$ \partial_s B+\frac{u_{1,2}}{u_3}\cdot\nabla_a B = \frac{1}{u_3}(\nabla u-\lambda I)B \,.$$
It is now clear that we can view $z=s \in [0, \infty)$ as a ``time variable''. To apply the method of characteristics we now need to study the flow map 
\begin{align} \label{d:vector-change}
    \frac{\dd}{\dd s} \Phi_\w^s= \w(\Phi_\w^s,s)\,, \qquad  \w= \frac{u_{1,2}}{u_3} : \TT^2 \times [0,\infty) \to \RR^2 \,.
\end{align}
A key step is to ensure that the solution $B$ of
\eqref{eq:inviscid eigenvalue relation} is periodic in the
vertical variable $s=z$. In order to study this periodicity, we define the Poincaré return time 
\begin{equation}
\overline\tau_u(a) := \inf\left\{ t>0: \bigl(\Phi_u^t(a,0)\bigr)_3=1 \right\},
\end{equation}
and the Poincaré return map $\Phi_u^{\bar \tau_u}$. The preceding discussion then yields the following identities (see Lemma \ref{lemma:properties autonomous 1}):
\begin{align} \label{eq:identities}
      \overline\tau_u (a) = \int_0^1 \frac{1}{u_3(\Phi_{\w}^s(a),s)}\,\dd s \,, \qquad \Phi_u^{\overline\tau_u(a)}(a,0) = (\Phi_{\w}^1(a),1 )\,.
\end{align}
Thus, we define the operator $\cK_{\lambda,\eps}$ for the ideal evolution $\eps=0$ as\footnote{We use $D_x$ to specify that we differentiate the flow map $\Phi_u^t$ in the space variable and then we evaluate it at $t= \overline\tau_u$.} 
$$ \cK_{\lambda,0} : B_0 \to  ( \e^{- \lambda \bar \tau_u} D_x \Phi_u^{\bar \tau_u} B_0) \circ (\Phi_u^{\bar \tau_u})^{-1}$$
where $B_0 : \TT^2 \to \CC^3$ is a vector field defined on the surface $\{ z=0\}$. The preceding discussion shows that, for $\lambda\in\mathbb C$, the eigenvalue problem \eqref{eq:inviscid eigenvalue relation} admits a non-trivial solution if and only if there exists a non-trivial $B_0$ satisfying $\mathcal K_{\lambda,0}B_0=B_0$. In other words, one needs $1 \in \sigma_p (\cK_{\lambda,0})$.

It is now tempting to choose a vector field $u: \TT^3 \to \RR^3$ so that $\w$ generates the flow map $\Phi_\w^1 = T_N$ given by the time-periodic setting, and then to choose $u_3$ so that the return time $\e^{ -\lambda\bar{\tau}_u } \circ (\Phi_\w^1)^{-1}$ generates the contribution $\e^{2\pi i g}$ needed to reproduce the operator \eqref{eq:periodic-dynamo-operator} in the time-periodic setting.

However, there are two technicalities to deal with:
\begin{enumerate}
    \item The map $\Phi^1_\w : \TT^2 \to \TT^2$ is, in general, not measure preserving, unlike the map $T_N : \TT^2 \to \TT^2$ in the time-periodic setting (see Lemma \ref{lemma:properties autonomous 1}).
    \item Applying the chain rule to the identity in \eqref{eq:identities} gives
    $$ 
    D_x\Phi_u^{\overline\tau_u(a)}(a,0)
    \begin{pmatrix}
    I_2
    \\
    0
    \end{pmatrix}
    + \partial_t\Phi_u^{\overline\tau_u(a)}(a,0) \otimes D \overline\tau_u(a) =
    \begin{pmatrix}
    D\Phi_{\w}^1(a)\\ 
    0
    \end{pmatrix}
    $$
    hence in general 
    $$D_x\Phi_u^{\overline\tau_u(a)}(a,0)
    \begin{pmatrix}
    I_2
    \\
    0
    \end{pmatrix} \neq \begin{pmatrix}
    D\Phi_{\w}^1(a)\\ 
    *
    \end{pmatrix} $$
    since the derivative of the return time produces an additional term. So long as $u$ is of the form $(0,0,u_3)$ in a neighbourhood of the section $\{z=0\}$, this correction only affects the third component. Although such a condition may be imposed in a particular construction, it is not preserved throughout the open set $\mathcal D$ on which we aim to establish fast dynamo action (see Definition \ref{def:D}).
\end{enumerate}
 To deal with the first  technicality we apply Lemma \ref{lemma:quantitative-moser} to find a map $\kappa_u: \TT^2 \to \TT^2$ so that 
 $$P_u:=\kappa_u^{-1}\circ \Phi_\w^1 \circ\kappa_u:\TT^2 \to \TT^2  $$
 is measure preserving and we accordingly define the changed Poincaré return time 
 $$ \tau := \overline{\tau} \circ \kappa_u \,. $$

 The second technicality is now affected by the change of variables $\kappa_u$. To deal with it, we change frames, replacing the orthonormal basis $\{\mathbf{e}_1,\mathbf{e}_2,\mathbf{e}_3\}$ with the basis
 $$\left \{ \begin{pmatrix}
     D\kappa_u \mathbf{e}_1
     \\
     0
 \end{pmatrix}, \begin{pmatrix}
     D\kappa_u \mathbf{e}_2
     \\
     0
 \end{pmatrix}, u(\kappa_u (a),0) \right \}\,.$$

 In particular, we define the matrix $E_u :\CC^3 \to \CC^3$
 $$ E_u = \begin{pmatrix}
     D\kappa_u & u_{1,2} (\kappa_u (a))
     \\
     0 & u_3(\kappa_u (a))
 \end{pmatrix}$$

which induces an operator $\mathcal{E}_u$ on vector fields $B \in C^\infty (\TT^2; \CC^3)$ defined by
 $$\mathcal{E}_u : B (\cdot ) \mapsto E_u(\kappa_u^{-1}(\cdot )) B(\kappa^{-1}_u(\cdot)) \,. $$ 
In this way, after conjugation with the operator $\mathcal{E}_u$ we recover the structure of the time-periodic dynamo operator. More precisely,  we may factorise the operator $\cK_{\lambda, 0}$ as
$$ \cK_{\lambda, 0} = \mathcal{E}_u \circ \mathcal{D}_{u,\lambda} \circ \mathcal{E}_u^{-1}\,,$$
 where the operator $\mathcal{D}_{u, \lambda}$ is defined as
 $$ \mathcal{D}_{u,\lambda} B=\e^{-\lambda \tau_u\circ P_u^{-1}}\begin{pmatrix}
DP_u\circ P_u^{-1} & 0 \\
-D \tau_u \circ P_u^{-1} & 1
\end{pmatrix} B\circ P_u^{-1} \,,$$
see Lemma \ref{lemma:flow conjugation}.
We now note that the following ``flexibility'' result holds true (see Lemma \ref{lemma:autonomous mimicking time-periodic}).
\begin{lemma}
\label{lemma:flexibility of section maps}
Let $T_N:\TT^2\to\TT^2$ be a smooth, volume-preserving diffeomorphism
isotopic to the identity, and let $g\in C^\infty(\TT^2;\RR)$ be mean-free. Then there exists $n_0\geq1$ such that, for every integer $n\geq n_0$, there is a vector field
\begin{equation}
u^{(n)}\in C^\infty_\sigma(\TT^3; \RR^3), \qquad
\bigl(u^{(n)}\bigr)_3>0,
\end{equation}
for which
\begin{equation}
P_{u^{(n)}}=T_N,
\qquad
\tau_{u^{(n)}}\circ P_{u^{(n)}}^{-1}
=1-\frac{g}{n}.
\end{equation}
\end{lemma}

From here, the comparison with the time-periodic construction is rather striking: isolating the top $2 \times 2$ block of $\mathcal{D}_{u,\lambda}$ yields precisely the two-dimensional dynamo operator, and the role of the phase $\e^{2\pi ig}$ is played by $\e^{-\lambda \tau_u\circ P_u^{-1}}$. Indeed, we look for an eigenvalue $\lambda = 2\pi i n + \theta$ with $n \in \NN$ large and $\theta \in \CC$ so that $\cM_{T_N, g} $ has an eigenvalue $\e^{\theta}$ satisfying $|\e^{\theta}|>1$. Then, denoting by $\Pi_h : \CC^3 \to \CC^2$ the projection onto the first two components we have reconstructed the identity
$$ \Pi_h \mathcal{D}_{u^{(n)}, \lambda}  = \e^{\frac{\theta g}{n}} \e^{- \theta} \cM_{T_N, g}\Pi_h\,,$$
where from the previous discussion we have $1 \in \sigma (\e^{- \theta}  \cM_{T_N, g} )$. We notice that the factor $\e^{\frac{\theta g}{n}}$ is a small perturbation for $n \gg 1$, which we can control in our anisotropic Banach space.  Consequently, reintroducing the third component, we will deduce by spectral perturbation theory and the implicit function theorem that, for any $n \in \NN$ sufficiently large, there exists $\tilde \theta_n \approx \theta$ such that  
$\mathcal{D}_{u^{(n)},\,2\pi i n+\tilde \theta_n}$
has the algebraically simple eigenvalue $1$. The maps $u \in C^{k+2}\mapsto P_u \in C^k$ and $u \in C^{k+2} \mapsto \tau_u \in C^{k}$ are continuous. Moreover, the map $(P_u,\tau_u) \in C^k \mapsto \mathcal{D}_{u,\lambda}$ is continuous when $\mathcal{D}_{u,\lambda}$ is viewed as a map  $ \mathcal{D}_{u,\lambda}: \mathcal{X}\to\mathcal{Y} $. Furthermore, by Lemma \ref{lemma:D implies lasota yorke}, the operator $\mathcal{D}_{u, \lambda}$ satisfies a Lasota--Yorke inequality that is locally uniform in $u$. Therefore, the Keller--Liverani theorem implies that, upon slightly changing $\lambda$, $\mathcal{D}_{u,\lambda}$ retains the eigenvalue $1$ under small velocity perturbations. This yields the following openness result.

\begin{proposition}
\label{prop:openness 1}
There exist $k\in\NN$ and a non-empty subset
$$
\mathcal D
\subset
\left\{
u\in C^\infty_\sigma(\TT^3):u_3>0
\right\},
$$
open in the $C^k$ topology, such that the following holds. For every
$u\in\mathcal D$, there exists $\lambda_0\in\CC$, with
$\Re\lambda_0>0$, for which $\mathcal{D}_{u,\lambda_0}$ has an algebraically
simple eigenvalue $1$, and such that its essential spectral radius on $\mathcal{X}$ is strictly less than $1$. In particular, for any $u \in \mathcal{D}$ there exists a non-trivial solution to \eqref{eq:inviscid eigenvalue relation} with $\lambda = \lambda_0$ satisfying $\Re (\lambda)>0$.
\end{proposition}

It remains to prove that this spectral structure persists after the addition of magnetic resistivity.

\subsection{The diffusive perturbations} Finally,  we deal with the case of positive diffusion $\eps>0$ by means of a perturbation argument. As in \cite{CZSV26}, the central abstract tool we rely on is the Keller--Liverani perturbation theorem \cite{Keller_Liverani}. While the time-periodic case can be handled rather directly using the representation \eqref{eq:flow representation}, the autonomous construction requires substantially more care. We consider the eigenvalue problem 
\begin{align} \label{eq:elliptic evolution}
    \eps \Delta B - u\cdot\nabla B+B\cdot\nabla u = \lambda B\,.
\end{align}
As before, dividing the PDE by $u_3$ and viewing $z=s \in [0, \infty)$ as a time variable and $a = (x,y)\in \TT^2$ as a spatial variable, we obtain 
\begin{equation}
\label{eq:intro elliptic evolution}\partial_s B+\frac{u_{1,2}}{u_3}\cdot\nabla_a B = \frac{1}{u_3}(\nabla u-\lambda I)B + \frac{\eps}{u_3}  \Delta_{a} B + \frac{\eps}{u_3} \partial_{ss} B\,,\end{equation}

with ``initial datum'' at  $s=0$, given by $B_0 : \TT^2\to \CC^3$. However, at this point the interpretation of $B_0$ as an initial datum must be weakened: the addition of $\eps \partial_s^2$ makes the equation elliptic, so that specifying a boundary condition at $s=0$ no longer gives a unique solution. Hence, in order to restore the interpretation of a temporal evolution equation, we first restore uniqueness by imposing an exponential growth bound as $s \to \infty$ in Lemma \ref{lemma:notion of solution}. Hence, we may denote by 
$\cK_{\lambda, \eps}B_0$ the trace of the solution to \eqref{eq:intro elliptic evolution} with boundary condition $B(a,0)=B_0(a)$, evaluated at $s=1$. 

As before, restoring periodicity in $s$ requires searching for eigenvectors $B$ with eigenvalue $1$ for the operator $\cK_{\lambda, \eps}$. To do so, we seek to isolate the ideal contribution $\mathcal{D}_{u,\lambda}$ from the resistive effects, which may be done via the following factorisation
\begin{equation} 
\label{eq:factorisation intro}\cK_{\lambda, \eps} =  \mathcal{E}_u \circ \mathcal{D}_{u, \lambda} \cU_{\lambda, \eps} \circ \mathcal{E}_u^{-1} \,.
\end{equation}
Indeed, the resistive effects are isolated and fully encoded in the operator  $\cU_{\lambda, \eps} $, which denotes the trace at $s=1$ of the solutions of the elliptic PDE
\begin{equation}
\label{eq:abstract pseudodifferential intro}
\partial_s v-\eps\left(\alpha^{0,0}\partial_s^2+2\alpha^{0,j}\partial_s\partial_j+\alpha^{i,j}\partial_i\partial_j+H^0\partial_s+H^j\partial_j+H\right)v=0\,, 
\end{equation}
given a prescribed boundary condition at $s=0$. Here, repeated indices $i,j\in\{1,2\}$ are summed, and the coefficients depend smoothly on $u$ and $\lambda$.  We suppress the dependence of $\mathcal U_{\lambda,\eps}$ on the fixed velocity field $u$. In view of \eqref{eq:factorisation intro}, and the fact that $\mathcal{D}_{u,\lambda}$ admits precisely the spectral structure we are after, it suffices to prove that $\mathcal{U}_{\lambda, \eps}$ is a perturbation of the identity as $\eps \to 0$ in our anisotropic Banach spaces. The main technical difficulty in proving this is controlling the couplings between the Hodge components of the solution uniformly as $\eps \to 0$: if a coupling from the solenoidal component to the gradient component were to remain as $\eps \to 0$, this would destroy precisely the mechanism that our anisotropic spaces exploit in order to obtain good spectral estimates. However, this possibility may be ruled out thanks to the crucial fact that the highest-order coefficients $\alpha^{\cdot, \cdot}$ of \eqref{eq:abstract pseudodifferential intro} are \emph{scalars}. This implies that the coupling between Hodge components occurs at a lower pseudodifferential order, and hence vanishes as $\eps \to 0$. All in all, we obtain the following result, proved using parameter-dependent pseudodifferential calculus (see Proposition \ref{prop:pseudodifferential abstract} and the estimates derived in the proof of Corollary \ref{cor:uniform lasota yorke}).

\begin{proposition}
Let $K\Subset\{\lambda\in\CC:\Re\lambda>0\}$ be compact. Then there exist $\eps_0>0$ and $C>0$ such that, for every
$\lambda\in K$, $0<\eps\leq\eps_0$, and
$B_0\in\mathcal E_u\mathcal X$, the operator
$
\mathcal U_{\lambda,\eps}\mathcal E_u^{-1}
:
\mathcal E_u\mathcal X\longrightarrow\mathcal X
$
is bounded and satisfies
\begin{align}
\bigl\|
\mathcal U_{\lambda,\eps}\mathcal E_u^{-1}B_0
\bigr\|_{\mathcal X}
&\leq
\left(
1+C\eps^{\frac{1-2r-\delta}{2}}
\right)
\|B_0\|_{\mathcal E_u\mathcal X}
+
C\|B_0\|_{\mathcal E_u\mathcal Y},
\\
\bigl\|
\mathcal U_{\lambda,\eps}\mathcal E_u^{-1}B_0
\bigr\|_{\mathcal Y}
&\leq
C\|B_0\|_{\mathcal E_u\mathcal Y},
\\
\bigl\|
\mathcal U_{\lambda,\eps}\mathcal E_u^{-1}B_0
-
\mathcal E_u^{-1}B_0
\bigr\|_{\mathcal Y}
&\leq
C\left(
\eps^{\frac{1-2r}{2}}
+
\eps^{\frac{\delta}{2}}
\right)
\|B_0\|_{\mathcal E_u\mathcal X}.
\end{align}
\end{proposition}

The first two estimates preserve the Lasota--Yorke inequality, with strong constant strictly less than $1$, for all sufficiently small $\eps$. The third gives the strong-to-weak convergence required by the parameter-uniform Keller--Liverani theorem \ref{thm:keller liverani}.
Combining this result with Rouch\'e's theorem, we can deduce that for all $\eps \ll 1$, there exists $\lambda_\eps\in\CC$, with $\lambda_\eps\xrightarrow[\varepsilon\to0]{}\lambda_0$, such that  $1 \in \sigma (\cK_{\lambda_\eps, \eps})$. Hence, there exist non-trivial solutions to \eqref{eq:elliptic evolution} with $\Re(\lambda_\eps)\geq \frac{\Re (\lambda_0)}{2}>0$ for $\eps \ll 1$ and we conclude the proof of 
Theorem \ref{thm:autonomous}.

\section{A family of time-periodic SFS fast dynamos} \label{sec:periodic-fast-dynamo}

The main objects we study throughout this paper are operators acting on $\CC^3$-valued vector fields on $\TT^2$, of the following form:
\begin{equation}
\label{eq:general form map}
\cS_{T,h, g}B:=\e^{2 \pi i g}\left (\begin{pmatrix}
DT & 0 \\
h & 1
\end{pmatrix} B\right )\circ T^{-1},
\end{equation} 

where $T: \TT^2 \to \TT^2$ is a volume-preserving diffeomorphism. We shall henceforth refer to this class of operators as \emph{SFS operators associated to $T$}. Indeed, operators in this class arise naturally as time-$1$ solution maps for time-periodic velocity fields of the form $u(t,x,y)$, in which case $T$ is the time-$1$ flow map generated by the first two components of $u$. The key insight is that, owing to the two-dimensional nature of the map $T:\TT^2 \to \TT^2$, the behaviour of the map \eqref{eq:general form map} is \emph{substantially} simpler than that of solutions of the kinematic dynamo equations on $\TT^3$. Indeed, it is well known that for two-dimensional velocity fields, the divergence-free solution $b$ to the kinematic dynamo equations is given by the perpendicular gradient of the solution to the passive scalar equation, a fact which is notably exploited in the Zeldovich antidynamo theorem \cite{Zeldovich_1992}.

In much the same vein, we shall exploit the algebraic simplifications afforded by two-dimensional dynamo action to define a class of \emph{anisotropic Banach spaces} of distributions, on which the operator \eqref{eq:general form map} admits favourable spectral properties. In this sense, the present work is heavily inspired by the recent result \cite{CZSV26}, which constructed anisotropic Banach spaces adapted to a specific underlying hyperbolic velocity field. The spaces constructed here are more global in nature, and adapted to a wider class of maps. We now introduce the scale of spaces we work on.

\subsection{Sobolev scalar and multiplier estimates}
Before moving on to studying the dynamo problem with the help of our anisotropic Banach spaces, we first note the following elementary results regarding the boundedness of composition with smooth, volume-preserving diffeomorphisms, as well as multiplication by smooth functions.
\begin{lemma}
\label{lemma:volume preserving diffeo on H^r}
Let $T: \TT^2 \to \TT^2$ be a volume-preserving diffeomorphism. Fix $r \in [-1,1]$. Then, for any $f \in C^\infty(\TT^2)$, there holds the bound
$$
\|f \circ T^{-1}\|_{H^r} \leq C_r\max\{1,\|DT\|_{L^\infty}\}^{|r|} \|f\|_{H^r}.
$$
\end{lemma}
\begin{proof}
Since $T$ is volume preserving, we have
$$
\|f \circ T^{-1}\|_{L^2} =\|f\|_{L^2}.
$$
Next, note that 
$$
\|\nabla (f\circ T^{-1})\|_{L^2} \leq \|DT^{-1}\|_{L^\infty} \|(\nabla f)\circ T^{-1}\|_{L^2} =\|DT^{-1}\|_{L^\infty}\|\nabla f \|_{L^2}.
$$
Hence, interpolating immediately yields 
$$
\|f\circ T^{-1}\|_{H^r} \leq \max\{1, \|DT^{-1}\|_{L^\infty}\}^r \|f\|_{H^r},
$$
for $r >0$.  For $r<0$, we have by duality
$$
|\langle f\circ T^{-1},g\rangle_{L^2}|=|\langle f, g\circ T\rangle_{L^2}|\leq \|f\|_{H^{-|r|}} \|g\circ T\|_{H^{|r|}}\leq \|f\|_{H^{-|r|}}\|g\|_{H^{|r|}}\max\{1,\|DT\|_{L^\infty}\}^{|r|}.
$$

Finally, note that since we are in finite dimensions, all norms are equivalent. Thus, we may take $|DT(x)|$ to be the spectral norm $\sqrt{\sigma_{max}(DT(x)DT^T(x))}$. But then, in dimension $2$, a square matrix with determinant $1$ has singular values that are reciprocal. Hence, it follows that $|DT(x)|=|DT^{-1}(T(x))|$, which completes the proof.
\end{proof}
\begin{lemma}
\label{lemma:smooth sobolev multiplier}
Let $a \in C^\infty(\TT^2)$ be a smooth function, and fix $r \in [-1,1]$, $\delta \in [0,1-|r|]$. Then, there exist $C_r, C_{r,\delta}>0$ so that for all $f \in C^\infty(\TT^2)$ we have
$$
\|a f\|_{H^{r}} \leq C_r \|a\|_{L^\infty}\|f\|_{H^r}+C_{r,\delta}\|\nabla a\|_{L^\infty}\|f\|_{H^{r-\delta}}.
$$
\end{lemma}
\begin{proof}
For $r=0$, we immediately have
$$
\|a f\|_{L^2} \leq \|a\|_{L^\infty} \|f\|_{L^2}.
$$
Next, applying \cite{LenzmannSchikorra2020}*{Theorem 6.1, (6.1)}, we bound
$$
\||\nabla|^r (a f)\|_{L^2}  \leq \|a |\nabla|^r f\|_{L^2}+\|[|\nabla|^r,a]f\|_{L^2} \leq \|a\|_{L^\infty}\|f\|_{H^r}+C_r \|\nabla a\|_{L^\infty}\|f\|_{H^{r-1}},
$$
yielding the desired result for positive $r$. For negative $r$, we have
$$
\|\langle D \rangle^{-|r|} a f\|_{L^2}=\|\langle D \rangle^{-|r|}a \langle D \rangle^{|r|} \langle D \rangle^{-|r|}f \|_{L^2} \leq \|a\langle D \rangle^{-|r|}f\|_{L^2}+\|\langle D \rangle^{-|r|}[a,\langle D \rangle^{|r|}] \langle D \rangle^{-|r|} f\|_{L^2}.
$$
The first of these terms yields a bound $\|a\|_{L^\infty}\|f\|_{H^{r}}$, and the second one is bounded by 
$$
\|[a,\langle D \rangle^{|r|}]\langle D \rangle^{-|r|}f \|_{L^2} \leq C\|\nabla a\|_{L^\infty}\|\langle D \rangle^{-|r|}f \|_{H^{|r|-1}}\leq C\|\nabla a\|_{L^\infty}\|f\|_{H^{-1}}.
$$
Hence, we obtain the desired result.
\end{proof}
\subsection{The SFS operator on anisotropic spaces}

The key aspect of our anisotropic spaces is that they, in some sense, spectrally encode the geometric cancellations inherent to two-dimensional anti-dynamo theorems. Indeed, they allow us to prove the following \emph{Lasota--Yorke} inequality for SFS operators, which is the main technical result of this section.
\begin{proposition}
\label{prop:LY}
Let $\cS_{T,h, g}$ be the following SFS operator:
\begin{equation}
\label{eq:general SFS operator}
\cS_{T,h,g} B=\e^{2\pi i g}\begin{pmatrix}
DT(T^{-1}) & 0\\
h(T^{-1}) & 1
\end{pmatrix}B(T^{-1}),
\end{equation}

where $h: \TT^2 \to \CC^2$ is a mean-free, smooth vector field, $g :\TT^2 \to \CC$ is a smooth function, and $T:\TT^2 \to \TT^2$ is a volume-preserving diffeomorphism isotopic to the identity. Then, there exist $C_r, C_{r,\delta}>0$ independent of $T,h$, so that there holds the following Lasota--Yorke inequality
\begin{align}
\|\cS_{T,h,g} B\|_{\mathcal X} \leq C_{LY,S}(T,h,g)\|B\|_{\mathcal X}+C_{LY,W}(T,h,g)\|B\|_{\mathcal Y}
\end{align}
where 
\begin{equation}
\label{eq:strong LY constant}
C_{LY,S}(T,h,g)=C_{r,\delta}\|\e^{2\pi i g}\|_{L^\infty}\bigl (\|DT \|_{L^\infty}^{1-r}+\|DT \|_{L^\infty}^{r+\delta} \bigr) ,
\end{equation}
\begin{align}
\label{eq:weak LY constant}
C_{LY,W}(T,h,g)& =C_{r,\delta}(\|\e^{2\pi i g}\|_{L^\infty}\bigl (\|DT\|_{L^\infty}+\|DT\|_{L^\infty}^{r+\delta}\| h\|_{C^1} \bigr )\\
&\quad +\|\nabla (\e^{2\pi i g})\|_{L^\infty} \bigl(\|DT\|_{L^\infty}^{1-r+\delta}+\|DT\|_{L^\infty}+\|DT\|_{L^\infty}^{r+2\delta}(1+\|h\|_{C^1}) \bigr )) \notag.
\end{align}
Finally, there holds the weak-to-weak bound 
$$
\|\cS_{T,h,g} B\|_{\mathcal Y} \leq C_{LY,WW}(T,h,g)\|B\|_{\mathcal Y},
$$
where 
\begin{align}
\label{eq:weak weak LY constant}
&C_{LY,WW}(T,h,g)=C_{r,\delta}\|\e^{2\pi i g}\|_{C^1}\bigl(\|DT\|_{L^\infty}^{1-r+\delta}+\|DT\|_{L^\infty}+\|DT\|_{L^\infty}^{r+2\delta}(1+\|h\|_{C^1}) \bigr ).
\end{align}
\end{proposition}
In order to prove the proposition, it is convenient to first prove it in the special case where $g \equiv 0$. 
\begin{lemma}
The result of Proposition \ref{prop:LY} holds when $g \equiv 0$.
\end{lemma}
\begin{proof}
Throughout the proof, write $\cS:=\cS_{T,h,0}$.\\

\textbf{Step 1: the spatial average.}
We begin by discussing the top-left component. In order to apply Leray projectors, we must first study what happens to the spatial average. The full horizontal component is $\Pi_h B=\Pi_h\overline B+b$. Write $T(x)=x+\phi(x)$, where $\int_{\TT^2}D\phi(x) \dd x=0$. This decomposition follows since $T$ is isotopic to the identity, i.e.\ it arises as the time-$1$ flow map of some smooth, time-dependent velocity field. Then, changing variables and integrating by parts yields
\begin{align*}
\int_{\TT^2} DT(T^{-1}(x,y)) (\Pi_hB)(T^{-1}(x,y)) \dd x \dd y
&=\int_{\TT^2} DT(x,y) (\Pi_hB)(x,y) \dd x \dd y\\
&=\Pi_h\overline B-\int_{\TT^2} \phi(x,y) (\nabla \cdot b)(x,y) \dd x \dd y.
\end{align*}
As such, if $b=\PP b$, the spatial average is conserved, and the only non-zero average is created by the gradient part of $b$. Furthermore, we bound 
$$
\left |\int_{\TT^2} DT(x,y) ((\Id-\PP)b)(x,y)\dd x \dd y \right | \leq \|DT\|_{L^\infty}\|(\Id-\PP)b\|_{L^2}.
$$
Finally, the spatial average for the vertical component is bounded by 
\begin{align*}
\left |\int_{\TT^2} \Pi_v(\cS B)(x,y) \dd x \dd y\right | \leq |\Pi_v\overline B|+\left |\int_{\TT^2}h(T^{-1}(x,y)) \cdot (\Pi_h\overline B+b(T^{-1}(x,y)) )\dd x \dd y \right |.
\end{align*}
Writing $b=\PP b+(\Id-\PP)b$, we estimate each term separately. Firstly, since $h$ is mean-free, the contribution from the spatial average vanishes.
Next, the gradient part may be bounded by 
$$
\|h\|_{L^\infty} \|(\Id-\PP)b\|_{L^2}.
$$
The divergence-free component is dealt with via 
\begin{align*}
\left | \int_{\TT^2} h(T^{-1}(x,y))\cdot (\PP b)(T^{-1}(x,y)) \dd x \dd y \right |\leq \|h\|_{H^{r+\delta}}\|(\PP b)\|_{H^{-r-\delta}}&\leq C \|h\|_{C^1}\|\PP b\|_{H^{-r-\delta}}.
\end{align*}
Hence, we now move on to the non-average components.\\

\textbf{Step 2: The divergence-free part.}
We begin by studying 
$$
\|\PP\Pi_h\cS(b,0)\|_{H^{-r}}^2.
$$

First, let us consider initial data of the form $\PP b$. Then, $b=\nabla^\perp \psi$, for some smooth scalar field $\psi$. As in the Zeldovich anti-dynamo theorem, we have
\begin{equation}
\label{eq:zeldovich anisotropic}
\Pi_h\cS(\PP b,0)=\nabla^\perp (\psi(T^{-1}))=\PP (\nabla^\perp (\psi(T^{-1}))).
\end{equation}
Hence, we bound
\begin{equation}
\label{eq:div free part decay}
\|\nabla^\perp (\psi(T^{-1}))\|_{H^{-r}} \leq \|\psi(T^{-1})\|_{H^{1-r}} \leq C_r\|\psi\|_{H^{1-r}}\|DT\|_{L^\infty}^{1-r} \leq C_r\|\nabla^\perp \psi\|_{H^{-r}}\|DT\|_{L^\infty}^{1-r}.
\end{equation}
Here, we have employed Lemma \ref{lemma:volume preserving diffeo on H^r}. Next, we study the cross term arising from an initial datum of gradient form that generates a divergence-free component in the solution. For this, we simply bound
$$
\|\PP\Pi_h\cS((\Id-\PP)b,0)\|_{H^{-r}} \leq \|\Pi_h\cS((\Id-\PP)b,0)\|_{L^2} \leq \|DT\|_{L^\infty}\|(\Id-\PP)b\|_{L^2}.
$$
Finally, note that a constant input can create a divergence-free output. Indeed, we have 
$$
\|\PP\Pi_h(\cS \overline B)\|_{H^{-r}} \leq \|\Pi_h(\cS \overline B)\|_{L^2} \leq \|DT\|_{L^\infty}|\overline B|.
$$
Next, we move on to controlling the gradient part of the output.\\

\textbf{Step 3: The gradient part.} 
We abuse notation and write $(\Id-\PP)\Pi_h\cS((\Id-\PP)b,0)=\nabla \Delta^{-1}\nabla \cdot\Pi_h\cS((\Id-\PP)b,0)$, noting that this is the expression one would obtain upon subtracting the mean part, and then applying $(\Id-\PP)$. The key insight for this block is that the potentially problematic cross-term $(\Id-\PP)\Pi_h\cS(\PP b,0)$ vanishes. Indeed, write $\PP b=\nabla^\perp \psi$. Then, from \eqref{eq:zeldovich anisotropic}, we have
$$
\Pi_h\cS(\PP b,0)=\nabla^\perp (\psi(T^{-1})),
$$
so that $(\Id-\PP)=\nabla \Delta^{-1} \nabla \cdot$ annihilates $\Pi_h\cS(\PP b,0)$.
Now, a computation reveals that 
$$
\nabla \cdot\Pi_h\cS(b,0)=(\nabla \cdot b)\circ T^{-1}.
$$
Hence, we have 
\begin{align*}
    \|(\Id-\PP)\Pi_h\cS((\Id-\PP)b,0)\|_{H^{r}}& =\|\nabla \Delta^{-1}((\nabla \cdot (\Id-\PP)b)\circ T^{-1})\|_{H^r} 
    \\
    &\leq \|(\nabla \cdot (\Id-\PP)b)\circ T^{-1}\|_{H^{r-1}}.
\end{align*}
Applying Lemma \ref{lemma:volume preserving diffeo on H^r} once more, we bound this by
$$
C_r \|DT\|_{L^\infty}^{1-r}\|\nabla \cdot (\Id-\PP)b\|_{H^{r-1}} \leq C_r\|DT\|_{L^\infty}^{1-r}\|(\Id-\PP)b\|_{H^{r}}.
$$
Next, it holds that $(\Id-\PP)\Pi_h(\cS \overline B)=0$. Indeed, $(\Id-\PP)=\nabla \Delta^{-1}\nabla \cdot$, and $\Pi_h(\cS \overline B)$ is divergence-free.
Finally, we move on to the third component of the SFS map.\\

\textbf{Step 4: the vertical component.} Firstly, we note that 
$$
\|b_v(T^{-1})\|_{H^{-r-\delta}} \leq C_{r,\delta}\|DT\|_{L^\infty}^{r+\delta} \|b_v\|_{H^{-r-\delta}},
$$
Next, applying Lemmas \ref{lemma:volume preserving diffeo on H^r} and \ref{lemma:smooth sobolev multiplier} we bound 
$$
\|h(T^{-1})\cdot \PP b(T^{-1})\|_{H^{-r-\delta}} \leq C_{r,\delta}\|h\|_{C^1}\|DT\|_{L^\infty}^{r+\delta} \|\PP b\|_{H^{-r-\delta}}
$$
Similarly, we have 
$$
\|h(T^{-1})\cdot ((\Id-\PP) b)(T^{-1})\|_{H^{-r-\delta}} \leq \|h (\Id-\PP)b\|_{L^2}  \leq \|h\|_{L^\infty}\|(\Id-\PP) b\|_{H^{r-\delta}}
$$
Finally, the mean-free contribution arising from a constant input is
$$
\|h\circ T^{-1} \cdot \Pi_h\overline B\|_{H^{-r-\delta}} \leq |\overline B|\|h\|_{L^\infty}.
$$
Hence, we may now complete the proof:\\

\textbf{Step 5: Assembling the estimates.} 
We now write up the full estimates that we have established, and assemble them into a Lasota--Yorke inequality. Firstly, we have that 
$$
|\overline{\cS B}|\leq |\overline B|+(\|h\|_{L^\infty}+\|DT\|_{L^\infty})\|(\Id-\PP)b\|_{L^2}+C\|h\|_{C^1}\|\PP b\|_{H^{-r-\delta}}.
$$
All in all, we deduce the existence of a constant $C_{r,\delta}>0$ independent of $T, h$ so that 
\begin{equation}
\label{eq:LY-average}
|\overline{\cS B}|\leq C_{r,\delta}(\|h\|_{C^1}+\|DT\|_{L^\infty}) \|B\|_{\mathcal Y}.
\end{equation}
Next up, the gradient-part estimates yield 
$$
\|(\Id-\PP)\Pi_h(\cS B)\|_{H^r} \leq C_r \|DT \|_{L^\infty}^{1-r}\|(\Id-\PP)b \|_{H^r}.
$$
As such, we deduce the existence of a constant $C_{r,\delta}>0$ so that 
\begin{equation}
\label{eq:gradient part}
\|(\Id-\PP)\Pi_h(\cS B)\|_{H^{r}} \leq C_{r,\delta} \|DT\|_{L^\infty}^{1-r}\|B\|_{\mathcal X}, \quad \|(\Id-\PP)\Pi_h(\cS B)\|_{H^{r-\delta}} \leq C_{r,\delta} \|DT\|_{L^\infty}^{1-r+\delta}\|B\|_{\mathcal Y}.
\end{equation}
Next, we assemble the divergence-free estimates. Here, we have 
$$
\|\PP\Pi_h(\cS B)\|_{H^{-r}} \leq C_r \|DT\|_{L^\infty}^{1-r} \|\PP b\|_{H^{-r}}+C_r\|DT\|_{L^\infty}\|(\Id-\PP)b\|_{L^2}+\|DT\|_{L^\infty}|\overline B|.
$$
Hence, we deduce the existence of $C_r, C_{r,\delta}>0$ so that 
\begin{equation}
\label{eq:LY-divergence free}
\|\PP\Pi_h(\cS B)\|_{H^{-r}} \leq C_r \|DT\|_{L^\infty}^{1-r}\|B\|_{\mathcal X}+C_r \|DT\|_{L^\infty}\|B\|_{\mathcal Y}
\end{equation}
\begin{equation}
\label{eq:LY-divergence free Y}
\|\PP\Pi_h(\cS B)\|_{H^{-r-\delta}} \leq C_{r,\delta}\|DT \|_{L^\infty}\|B\|_{\mathcal Y}. 
\end{equation}
Finally, we consider the vertical component. Here, there exist constants $C_r>0$ and $C_{r,\delta}>0$ such that
$$
\|\Pi_v(\cS B)\|_{H^{-r-\delta}} \leq C_{r,\delta}\|DT\|_{L^\infty}^{r+\delta}\|b_v\|_{H^{-r-\delta}}+C_{r,\delta}\|DT\|_{L^\infty}^{r+\delta}\|h\|_{C^1}\|b\|_{Y} + C_{r,\delta}(1+\|h\|_{C^1})|\overline B|
$$
and hence 
\begin{equation}
\label{eq:LY vertical}
\|\Pi_v(\cS B)\|_{H^{-r-\delta}} \leq C_{r,\delta} \|DT\|_{L^\infty}^{r+\delta} \|B\|_{\mathcal X}+C_{r,\delta}\|h \|_{C^1}\|DT\|^{r+\delta}_{L^\infty}\|B\|_{\mathcal Y},
\end{equation}
\begin{equation}
\label{eq:LY vertical Y}
\|\Pi_v(\cS B)\|_{H^{-r-2\delta}} \leq C_{r,\delta} \|DT\|^{r+2\delta}_{L^\infty}(1+\|h\|_{C^1})\|B\|_{\mathcal Y}.
\end{equation}
Hence, the proof is complete.
\end{proof}
The proof of Proposition \ref{prop:LY} now follows from the following lemma.
\begin{lemma}
\label{lemma:multipliers anisotropic}

Let $a \in C^\infty(\TT^2)$. Then, there exist $C_r, C_{r,\delta}>0$ so that 
$$
\|a B\|_{\mathcal X} \leq C_r \|a \|_{L^\infty} \|B\|_{\mathcal X}+C_{r,\delta}\|\nabla a\|_{L^\infty}\|B\|_{\mathcal Y}.
$$
Furthermore, there holds the bound 
$$
\|aB\|_{\mathcal Y} \leq C_{r,\delta}\|a\|_{C^1}\|B\|_{\mathcal Y}.
$$
\end{lemma}
\begin{proof}
We begin with the mean component. Write $B=\overline B+(\PP b+(\Id-\PP)b,b_v)$, and estimate each component separately. Doing so, we obtain firstly
$$
\left | \int_{\TT^2} a(x,y) \overline B \dd x \dd y\right | \leq \|a\|_{L^\infty}|\overline B|, \quad \left | \int_{\TT^2} a(x,y) (\Id-\PP)b \dd x \dd y\right | \leq \|a\|_{L^\infty}\|(\Id-\PP)b\|_{L^2}.
$$
Finally, by duality, we can estimate 
$$
\left | \int_{\TT^2} a(x,y) \PP b (x,y) \dd x \dd y\right | \leq \|a\|_{\dot H^{r+\delta}}\|\PP b\|_{H^{-r-\delta}} \leq \|\nabla a\|_{L^\infty} \|\PP b\|_{H^{-r-\delta}},
$$
and an identical estimate works for the average induced in the vertical component.

Next, we deal with the vertical component. Here, we have by Lemma \ref{lemma:smooth sobolev multiplier}
$$
\|a b_v\|_{H^{-r-\delta}} \leq C_r \|a\|_{L^\infty}\|b_v\|_{H^{-r-\delta}}+C_{r,\delta}\|\nabla a\|_{L^\infty}\|b_v\|_{H^{-r-2\delta}}.
$$
We now move on to the horizontal component. Here, we have to be somewhat careful, since a potential leakage of the distribution-valued divergence-free part into the function-valued gradient part could in principle yield an unbounded operator. Writing this two-component field as $b$, we have 
$$
\|(\Id-\PP) a \PP b\|_{H^{r}} =\|[\Id-\PP,a] \PP b\|_{H^r}.
$$

But now, note that by \cite{CoifmanMeyer1978}*{Theorem 2, and the paragraph following (4), p.~180}, if $T$ is a pseudodifferential operator of order zero (see Appendix \ref{sec:pseudodifferential intro}), and $a$ is Lipschitz, then 
$$
\|[T,a] b\|_{H^1} \leq C_T \|\nabla a\|_{L^\infty}\|b\|_{L^2}.
$$
In particular, we can take $T=(\Id-\PP)$, and use the fact that it is self-adjoint on $L^2$ to also bound 
$$
\left | \langle [T,a] b,\widetilde b \rangle_{L^2}\right |=\left |\langle b, [\overline a,T]\widetilde b\rangle \right | \leq C_T\|b\|_{H^{-1}} \|\nabla a\|_{L^\infty} \|\widetilde b\|_{L^2},
$$
so that 
$$
\|[T,a] b\|_{L^2} \leq C_T \|\nabla a\|_{L^\infty}\|b\|_{H^{-1}}.
$$
Hence, interpolation yields 
$$
\|[(\Id-\PP),a] b\|_{H^{1/2}} \leq C\|\nabla a\|_{L^\infty}\|b\|_{H^{-1/2}}.
$$
Since $r<1/2$, we thus obtain 
$$
\|[(\Id-\PP),a] \PP b\|_{H^{r}} \leq C\|\nabla a\|_{L^\infty}\|\PP b\|_{H^{-1/2}}\leq C\|\nabla a\|_{L^\infty}\|\PP b\|_{H^{-r-\delta}}.
$$
Similarly, we may bound 
$$
\|[\PP,a](\Id-\PP)b\|_{H^{-r}} \leq \|[\PP,a](\Id-\PP)b\|_{H^{1}} \leq C\|\nabla a\|_{L^\infty}\|(\Id-\PP)b\|_{L^2}.
$$
Identical arguments yield the desired bounds for the diagonal blocks $\|\PP a \PP b\|_{H^{-r}}$, and $\|(\Id-\PP) a (\Id-\PP) b\|_{H^{r}}$.
Finally, consider the contribution arising from a constant input, i.e.\ $a\overline B$. For the divergence-free part, we note that $L^2$ embeds continuously into $H^{-r}$, and simply bound 
$$
\|\PP (a-\overline a)\Pi_h\overline B\|_{H^{-r}}\leq \|a\|_{L^{\infty}}|\Pi_h\overline B|.
$$
The same argument also controls the vertical component.
For the gradient part, we bound 
$$
\|(\Id-\PP) (a-\overline a)\Pi_h\overline B\|_{H^{r}} \leq \|(a-\overline a)\Pi_h\overline B\|_{H^1} \leq C\|a\|_{L^\infty} |\Pi_h\overline B|+C\|\nabla a\|_{L^\infty}|\Pi_h\overline B|.
$$
Collecting all the estimates completes the proof of the strong norm estimate. The weak norm estimate follows immediately by repeating the proof with $r$ replaced by $r-\delta$ and $-r$ replaced by $-r-\delta$.
\end{proof}
\subsection{Unstable eigenvalues in a class of SFS maps}
A particularly appealing consequence of the Lasota--Yorke inequality of Proposition \ref{prop:LY} is the following: If one can find a family of maps $T_N$ such that $\|D T_N\|_{L^\infty} \sim N^2$, then, upon normalising the operator by dividing by $N^2$, for fixed $g, h$, the normalised strong Lasota--Yorke constant $C_{LY,S}(T_N,h,g)$ behaves like $N^{-2r}+N^{2r+2\delta-2}$, while the weak constant is of order $1$. Thus, the essential spectral radius contracts towards zero, possibly leaving behind isolated, discrete eigenvalues. Indeed, this mechanism was exhibited in the recent work \cite{CZSV26}. We hence consider the family of maps $T_N$ from \eqref{eq:TN}.
\begin{proposition}
\label{prop:first growing mode}
Let $T_N$ be the time-$1$ flow map in \eqref{eq:TN}, and let $g \in C^\infty(\TT^2)$ be any real-valued, smooth function so that 
$$
\int_{\TT^2} \e^{2 \pi i g(x,y)}f_1'(x)f_2'(y) \dd x \dd y \neq 0.
$$
Then, for all $N$ large enough, the SFS operator 
$$
\mathcal{M}_{T_N,g} b=\e^{2 \pi i g}(DT_N b)\circ T_N^{-1}
$$
admits an isolated, algebraically simple eigenvalue $\Lambda_N$ on $X$, with $|\Lambda_N|\geq c_0N^2$, for some $c_0>0$.
\end{proposition}

The proof of the proposition will be divided into multiple smaller lemmas. The first is the following elementary consequence of Proposition \ref{prop:LY}.
\begin{lemma}
\label{lemma:uniform LY N}
For any $g \in C^\infty(\TT^2)$, there exist $C_{r,g}, C_{r,\delta,g}>0$ so that 
$$
\|N^{-2}\mathcal{M}_{T_N,g} b\|_{X} \leq C_{r,g}N^{-2r}\|b\|_{X}+C_{r,\delta,g}\|b\|_{Y}, 
\quad \|N^{-2}\mathcal{M}_{T_N,g} b\|_{Y} \leq C_{r,\delta,g} \|b\|_{Y}.
$$

\end{lemma}
\begin{proof}
We note that $\mathcal{M}_{T_N,g}$ is simply the first two components of an SFS operator with $h=0$. Hence, this operator acts on the two-component spaces $X$ and $Y$. Now, a computation reveals that 
$$
T_N(x,y)=\begin{pmatrix}
x+Nf_2(y+Nf_1(x))\\
y+Nf_1(x)
\end{pmatrix}.
$$
In particular, $\|DT_N\|_{L^\infty} \sim N^2$. Hence, applying Proposition \ref{prop:LY}, we deduce that 
$$
\|\mathcal{M}_{T_N,g} b\|_{X} \leq C_{r,g}N^{2-2r}\|b\|_{X}+C_{r,\delta,g}N^2\|b\|_{Y}, 
\quad \|\mathcal{M}_{T_N,g} b\|_{Y} \leq C_{r,\delta,g}N^2 \|b\|_{Y},
$$
completing the proof.
\end{proof}
Next, we identify a suitable limiting object which $N^{-2}\mathcal{M}_{T_N,g}$ approaches as $N \to \infty$. In fact, it turns out that it ``converges'' to a rank-one operator! However, while this may be good news for establishing the spectral properties of the limiting object, it makes the convergence result somewhat tricky to establish. Thus, we shall content ourselves with the weaker notion of convergence due to Keller--Liverani \cite{Keller_Liverani}, which will turn out to be rather important throughout this document.
\subsection{The rank-one limiting operators $\mathcal L_\infty$ and $\mathcal M_\infty$}

We define the linear functional
\begin{equation}\label{eq:ell_def}
\ell(b)
:=
\int_{\TT^2} f_1'(x)b_1(x,y)\,\dd x\,\dd y,
\end{equation}
and the rank-one operator
\begin{equation}\label{eq:Linf_def}
(\mathcal L_\infty b)(x,y)
:=
\begin{pmatrix}
f_2'(y)\ell(b)\\
0
\end{pmatrix}=
\ell(b)\,
\begin{pmatrix}
f_2'(y)\\
0
\end{pmatrix}.
\end{equation}
Since $\int_{\TT}f_1'(x)\,\dd x=0$, one has $\mathcal L_\infty^2=0$.
The corresponding limiting full operator is
\begin{equation}\label{eq:Minf_def}
\mathcal M_\infty b
:=
\e^{2\pi i g}\mathcal L_\infty b
=
\ell(b)h_g,
\qquad
h_g(x,y)
:=
\begin{pmatrix}
\e^{2\pi i g(x,y)}f_2'(y)\\
0
\end{pmatrix}.
\end{equation}
The operator $\mathcal M_\infty$ is rank one. Its only possible nonzero
eigenvalue is
\begin{equation}\label{eq:lambda_g_def}
\lambda_g
:=
\ell(h_g)
=
\int_{\TT^2}
\e^{2\pi i g(x,y)}f_1'(x)f_2'(y)
\,\dd x\,\dd y.
\end{equation}
Indeed,
$$
\mathcal M_\infty h_g
=
\ell(h_g)h_g
=
\lambda_g h_g.
$$
Moreover,
$
\mathcal M_\infty^2=\lambda_g\mathcal M_\infty
$.
Thus, as an operator on either $X$ or $Y$, if $\lambda_g\neq0$, then
$$
\sigma(\mathcal M_\infty)=\{0,\lambda_g\},
$$
and $\lambda_g$ is an algebraically simple isolated eigenvalue of
$\mathcal M_\infty$. Indeed, the associated Riesz projector is the rank-one operator
$$
\Pi_g=\lambda_g^{-1}\mathcal M_\infty.
$$

\subsection{Weak--strong convergence}

We next prove the convergence of the finite-$N$ operators to the
rank-one limiting operator. The relevant convergence is from the strong
space $X$ to the weak space $Y$.

\begin{lemma}\label{lem:MN_to_Minf}
There exists a sequence $\omega_N\to0$ such that
\begin{equation}\label{eq:MN0_to_Minf}
\|N^{-2}\mathcal M_{T_N,g}b-\mathcal M_\infty b\|_Y
\leq
\omega_N\|b\|_X.
\end{equation}
\end{lemma}
Before proving this result, we first need the following preparatory lemma.
\begin{lemma}\label{lem:constant_modes_inviscid}
For every constant vector $c\in\mathbb C^2$ and every
$\sigma\in[0,1]$, one has
\begin{equation}\label{eq:constant_modes_inviscid}
\|N^{-2} \mathcal{M}_{T_N,0} c-N^{-2}c\|_{H^{-\sigma}}
\lesssim
N^{-\sigma}|c|.
\end{equation}
Consequently, if $\sigma=r+\delta$, then
\begin{equation}\label{eq:constant_modes_Y}
\|N^{-2} \mathcal{M}_{T_N,0} c\|_Y
\lesssim
N^{-(r+\delta)}|c|.
\end{equation}
\end{lemma}

\begin{proof}
Since $c$ is divergence-free, the mean-preservation identity gives
$$
\overline{N^{-2} \mathcal{M}_{T_N,0} c}=N^{-2}c.
$$
Moreover, $N^{-2} \mathcal{M}_{T_N,0} c$ is divergence-free, and hence
$$
(\Id-\PP)(N^{-2} \mathcal{M}_{T_N,0} c-N^{-2}c)=0.
$$
It remains to estimate the mean-free divergence-free part. Let
$$
\ell_c(x,y):=c_2x-c_1y,
$$
so that, on the lift to $\mathbb R^2$,
$
c=\nabla_{\rm h}^\perp \ell_c
$.
Although $\ell_c$ is not periodic, the difference
$$
q_c:=\ell_c\circ T_N^{-1}-\ell_c
$$
is periodic. 
Indeed, using
$$
T_N^{-1}(x,y)=
\bigl(x-Nf_2(y),\,y-Nf_1(x-Nf_2(y))\bigr),
$$
we obtain
$$
q_c(x,y)
=
Nc_1 f_1(x-Nf_2(y))
-
Nc_2 f_2(y).
$$
Therefore
$$
\|q_c\|_{L^2}\lesssim N|c|,
\qquad
\|q_c\|_{H^1}\lesssim N^2|c|.
$$
Interpolating between these two bounds gives, for $\sigma\in[0,1]$,
$$
\|q_c\|_{H^{1-\sigma}}
\lesssim
N^{2-\sigma}|c|.
$$
A direct computation gives
$$ \mathcal{M}_{T_N,0} c-c=\nabla_{\rm h}^\perp q_c.
$$
Using this identity,
we deduce
$$
\|N^{-2} \mathcal{M}_{T_N,0} c-N^{-2}c\|_{H^{-\sigma}}
\leq
N^{-2}\|\nabla_{\rm h}^\perp q_c\|_{H^{-\sigma}}
\lesssim
N^{-2}\|q_c\|_{H^{1-\sigma}}
\lesssim
N^{-\sigma}|c|.
$$
This proves \eqref{eq:constant_modes_inviscid}.
Finally, taking $\sigma=r+\delta$, and using
$$
\overline{N^{-2} \mathcal{M}_{T_N,0} c}=N^{-2}c,
\qquad
(\Id-\PP)(N^{-2} \mathcal{M}_{T_N,0} c-N^{-2}c)=0,
$$
we get
\begin{align*}
\|N^{-2} \mathcal{M}_{T_N,0} c\|_Y
&\leq
|\overline{N^{-2} \mathcal{M}_{T_N,0} c}|
+
\|\PP(N^{-2} \mathcal{M}_{T_N,0} c-\overline{N^{-2} \mathcal{M}_{T_N,0} c})\|_{H^{-r-\delta}}\\
&\quad+
\|(\Id-\PP)(N^{-2} \mathcal{M}_{T_N,0} c-\overline{N^{-2} \mathcal{M}_{T_N,0} c})\|_{H^{r-\delta}}
\\
&=
N^{-2}|c|
+
\|N^{-2} \mathcal{M}_{T_N,0} c-N^{-2}c\|_{H^{-r-\delta}}
\\
&\lesssim
N^{-2}|c|
+
N^{-(r+\delta)}|c|
\\
&\lesssim
N^{-(r+\delta)}|c|,
\end{align*}
since $N\geq1$ and $r+\delta<1$. This proves
\eqref{eq:constant_modes_Y}.
\end{proof}
Hence, we may now prove Lemma \ref{lem:MN_to_Minf}.
\begin{proof}[Proof of Lemma \ref{lem:MN_to_Minf}]
Since
$$
\mathcal M_{T_N,g}=\e^{2\pi i g} \mathcal{M}_{T_N,0},
\qquad
\mathcal M_\infty=\e^{2\pi i g}\mathcal L_\infty,
$$
and multiplication by $\e^{2\pi i g}$ is bounded on $Y$ by Lemma~\ref{lemma:multipliers anisotropic}, it suffices to
prove
$$
\|N^{-2}\mathcal M_{T_N,0}b-\mathcal L_\infty b\|_Y
\leq
\omega_N\|b\|_X.
$$
We first isolate the scalar oscillatory operator appearing in the leading
entry of $N^{-2} \mathcal{M}_{T_N,0}$. Using
$$
T_N^{-1}(x,y)
=
\bigl(x-Nf_2(y),\,y-Nf_1(x-Nf_2(y))\bigr),
$$
we define
$$
\mathcal A_Nh(x,y)
:=
f_2'(y)f_1'(x-Nf_2(y))
h\bigl(x-Nf_2(y),\,y-Nf_1(x-Nf_2(y))\bigr),
$$
and
$$
\mathcal A_\infty h(x,y)
:=
f_2'(y)\int_{\TT^2}f_1'(x')h(x',y')\,\dd x'\,\dd y'.
$$
The key claim is
\begin{equation}\label{eq:SN_compact_convergence}
\alpha_N:=
\sup_{\|h\|_{H^r}\leq1}
\|\mathcal A_Nh-\mathcal A_\infty h\|_{H^{-r-\delta}}
\to0.
\end{equation}
We prove the claim. First, $\mathcal A_N$ is uniformly bounded on $L^2$. Indeed,
the map
$$
(x,y)\mapsto
\bigl(x-Nf_2(y),\,y-Nf_1(x-Nf_2(y))\bigr)
$$
is measure-preserving, and therefore
$$
\|\mathcal A_Nh\|_{L^2}\lesssim \|h\|_{L^2}.
$$
The operator $\mathcal A_\infty:L^2\to L^2$ is also bounded.

We next show that $\mathcal A_Nh\rightharpoonup \mathcal A_\infty h$ weakly in $L^2$ for
every $h\in L^2$. By density and the uniform $L^2$-bounds, it is enough
to check this for trigonometric polynomials. Let $h,\varphi$ be
trigonometric polynomials. A change of variables gives
\begin{equation}\label{eq:SN_pairing}
\int_{\TT^2}\mathcal A_Nh\,\varphi
=
\int_{\TT^2}
f_1'(x)f_2'(y)
h(x,y-Nf_1(x))
\varphi(x+Nf_2(y),y)
\,\dd x\,\dd y.
\end{equation}
We write
$$
h(x,y)=\sum_{k,\ell}h_{k,\ell}\e^{2\pi i(kx+\ell y)},
\qquad
\varphi(x,y)=\sum_{m,n}\varphi_{m,n}\e^{2\pi i(mx+ny)}.
$$
The right-hand side of \eqref{eq:SN_pairing} is a finite sum of terms of
the form
$$
h_{k,\ell}\varphi_{m,n}
I_{k+m,\ell}(N)J_{\ell+n,m}(N),
$$
where
$$
I_{a,\ell}(N)
:=
\int_{\TT}f_1'(x)\e^{2\pi i(ax-\ell Nf_1(x))}\,\dd x,
\qquad
J_{q,m}(N)
:=
\int_{\TT}f_2'(y)\e^{2\pi i(qy+mNf_2(y))}\,\dd y.
$$
If $\ell\neq0$, then
$$
f_1'(x)\e^{-2\pi i\ell Nf_1(x)}
=
-\frac{1}{2\pi i\ell N}
\partial_x\left(\e^{-2\pi i\ell Nf_1(x)}\right),
$$
and integration by parts gives
$$
I_{a,\ell}(N)\to0, \qquad\text{as } N\to\infty. 
$$
Similarly, if $m\neq0$, then
$$
J_{q,m}(N)\to0, \qquad\text{as } N\to\infty. 
$$
Thus only the terms with $\ell=0$ and $m=0$ survive in the limit.
Consequently,
$$
\int_{\TT^2}\mathcal A_Nh\,\varphi
\to
\left(
\int_{\TT^2}f_1'(x)h(x,y)\,\dd x\,\dd y
\right)
\left(
\int_{\TT^2}f_2'(y)\varphi(x,y)\,\dd x\,\dd y
\right)=
\int_{\TT^2}\mathcal A_\infty h\,\varphi.
$$
Therefore $\mathcal A_Nh\rightharpoonup \mathcal A_\infty h$ weakly in $L^2$ for every
$h\in L^2$.
Since the embedding $L^2\hookrightarrow H^{-r-\delta}$ is compact, this weak $L^2$ convergence implies
$$
\mathcal A_Nh\to \mathcal A_\infty h
\qquad
\text{strongly in }H^{-r-\delta}
$$
for every fixed $h\in L^2$.

We now upgrade this pointwise convergence to convergence uniformly on the
unit ball of $H^r$. Suppose \eqref{eq:SN_compact_convergence} failed.
Then there would exist $\eps_0>0$, a subsequence $N_j\to\infty$,
and $h_j$ satisfying
$$
\|h_j\|_{H^r}\leq1,
\qquad
\|\mathcal A_{N_j}h_j-\mathcal A_\infty h_j\|_{H^{-r-\delta}}\geq\eps_0.
$$
Since $H^r\hookrightarrow L^2$ is compact, after passing to a subsequence
we may assume
$$
h_j\to h
\qquad
\text{strongly in }L^2.
$$
Using the uniform boundedness of $\mathcal A_N-\mathcal A_\infty:L^2\to H^{-r-\delta}$, we
get
$$
\|(\mathcal A_{N_j}-\mathcal A_\infty)(h_j-h)\|_{H^{-r-\delta}}
\lesssim
\|h_j-h\|_{L^2}
\to0.
$$
On the other hand,
$$
(\mathcal A_{N_j}-\mathcal A_\infty)h\to0
\qquad
\text{in }H^{-r-\delta}.
$$
This contradicts the lower bound above and proves
\eqref{eq:SN_compact_convergence}.

We now return to the vector-valued operator. Decompose
$$
b=\overline b+v+w,
\qquad
v:=\PP(b-\overline b),
\qquad
w:=(\Id-\PP)(b-\overline b).
$$
First, since $\overline b$ is constant,
$\mathcal L_\infty\overline b=0$,
because $\int_{\TT}f_1'(x)\,\dd x=0$. 
Moreover, by Lemma~\ref{lem:constant_modes_inviscid} with
$\sigma=r+\delta$,
$$
\|N^{-2} \mathcal{M}_{T_N,0}\overline b\|_Y
\lesssim
N^{-r-\delta}|\overline b|.
$$
Hence
\begin{equation}\label{eq:constant_convergence}
\|(N^{-2} \mathcal{M}_{T_N,0}-\mathcal L_\infty)\overline b\|_Y
\lesssim
N^{-r-\delta}|\overline b|.
\end{equation}
Second, $v$ is divergence-free and mean-free. Hence
$
\mathcal L_\infty v=0
$.
Indeed, writing $v=\nabla_{\rm h}^\perp\psi$, we have
$$
\ell(v)
=
\int_{\TT^2}f_1'(x)v_1(x,y)\,\dd x\,\dd y
=
-\int_{\TT^2}f_1'(x)\partial_y\psi(x,y)\,\dd x\,\dd y
=
0.
$$
Furthermore, by \eqref{eq:div free part decay}, we have
$$
\|N^{-2} \mathcal{M}_{T_N,0}v\|_Y
\lesssim
N^{-2(r+\delta)}\|v\|_{H^{-r}}.
$$
Thus
\begin{equation}\label{eq:divfree_convergence}
\|(N^{-2} \mathcal{M}_{T_N,0}-\mathcal L_\infty)v\|_Y
\lesssim
N^{-2(r+\delta)}\|v\|_{H^{-r}}.
\end{equation}
It remains to treat the curl-free component
$
w:=(\Id-\PP)(b-\overline b)
$.
We first control the mean and the divergence-free component of $(N^{-2} \mathcal{M}_{T_N,0}-\mathcal L_\infty)w$.
The explicit expression
for $N^{-2} \mathcal{M}_{T_N,0}$ gives
$$
N^{-2} \mathcal{M}_{T_N,0} w
=
\begin{pmatrix}
\mathcal A_Nw_1\\
0
\end{pmatrix}
+
\operatorname{Rem}_Nw,
$$
where
$$
\|\operatorname{Rem}_Nw\|_{L^2}
\lesssim
N^{-1}\|w\|_{L^2}
\lesssim
N^{-1}\|w\|_{H^r}.
$$
Since
$$
\mathcal L_\infty w
=
\begin{pmatrix}
\mathcal A_\infty w_1\\
0
\end{pmatrix},
$$
we obtain, using \eqref{eq:SN_compact_convergence},
$$
\|(N^{-2} \mathcal{M}_{T_N,0}-\mathcal L_\infty)w\|_{H^{-r-\delta}}
\lesssim
(\alpha_N+N^{-1})\|w\|_{H^r}.
$$
Indeed, the $\operatorname{Rem}_Nw$ term is estimated in $H^{-r-\delta}$ by its $L^2$
norm.

This estimate controls both the mean and the divergence-free component in
the $Y$-norm. Namely,
$$
|\overline{(N^{-2} \mathcal{M}_{T_N,0}-\mathcal L_\infty)w}|
\lesssim
\|(N^{-2} \mathcal{M}_{T_N,0}-\mathcal L_\infty)w\|_{H^{-r-\delta}},
$$
and, since $\PP$ is bounded on $H^{-r-\delta}$,
\begin{align*}
&\|\PP((N^{-2} \mathcal{M}_{T_N,0}-\mathcal L_\infty)w-\overline{(N^{-2} \mathcal{M}_{T_N,0}-\mathcal L_\infty)w})\|_{H^{-r-\delta}}\\
&
\lesssim
\|(N^{-2} \mathcal{M}_{T_N,0}-\mathcal L_\infty)w-\overline{(N^{-2} \mathcal{M}_{T_N,0}-\mathcal L_\infty)w}\|_{H^{-r-\delta}}\\
&
\lesssim
\|(N^{-2} \mathcal{M}_{T_N,0}-\mathcal L_\infty)w\|_{H^{-r-\delta}}.
\end{align*}
Therefore
\begin{equation}\label{eq:curlfree_mean_P_convergence}
|\overline{(N^{-2} \mathcal{M}_{T_N,0}-\mathcal L_\infty)w}|
+
\|\PP((N^{-2} \mathcal{M}_{T_N,0}-\mathcal L_\infty)w-\overline{(N^{-2} \mathcal{M}_{T_N,0}-\mathcal L_\infty)w})\|_{H^{-r-\delta}}
\lesssim
(\alpha_N+N^{-1})\|w\|_{H^r}.
\end{equation}
It remains to estimate the curl-free component of $(N^{-2} \mathcal{M}_{T_N,0}-\mathcal L_\infty)w$. Since
$$
\mathcal L_\infty w
=
\begin{pmatrix}
\ell(w)f_2'(y)\\
0
\end{pmatrix},
$$
it is divergence-free. Hence
$
(\Id-\PP)\mathcal L_\infty w=0
$.
Moreover, $\Id-\PP$ annihilates constants. Consequently,
$$
(\Id-\PP)((N^{-2} \mathcal{M}_{T_N,0}-\mathcal L_\infty)w-\overline{(N^{-2} \mathcal{M}_{T_N,0}-\mathcal L_\infty)w})
=
(\Id-\PP)N^{-2} \mathcal{M}_{T_N,0} w.
$$
By \eqref{eq:gradient part}, we have
$$
\|(\Id-\PP)N^{-2} \mathcal{M}_{T_N,0} w\|_{H^r}
\lesssim
N^{-2r}\|w\|_{H^r}.
$$
Thus
\begin{equation}\label{eq:curlfree_Q_convergence}
\|(\Id-\PP)((N^{-2} \mathcal{M}_{T_N,0}-\mathcal L_\infty)w-\overline{(N^{-2} \mathcal{M}_{T_N,0}-\mathcal L_\infty)w})\|_{H^r}
\lesssim
N^{-2r}\|w\|_{H^r}.
\end{equation}
Putting together \eqref{eq:curlfree_mean_P_convergence} and
\eqref{eq:curlfree_Q_convergence}, we obtain
\begin{equation}\label{eq:curlfree_convergence}
\|(N^{-2} \mathcal{M}_{T_N,0}-\mathcal L_\infty)w\|_Y
\lesssim
(\alpha_N+N^{-1}+N^{-2r})\|w\|_{H^r}.
\end{equation}
Combining
\eqref{eq:constant_convergence},
\eqref{eq:divfree_convergence}, and
\eqref{eq:curlfree_convergence}, we get
$$
\|N^{-2} \mathcal{M}_{T_N,0}b-\mathcal L_\infty b\|_Y
\leq
\omega_N
\left(
|\overline b|+\|v\|_{H^{-r}}+\|w\|_{H^r}
\right)
\leq C
\omega_N\|b\|_X,
$$
where
$$
\omega_N
=
C\left(
N^{-r-\delta}
+
N^{-2(r+\delta)}
+
\alpha_N
+
N^{-1}
+
N^{-2r}
\right)
\to0.
$$
This proves the lemma.
\end{proof}
Finally, we may conclude the proof of Proposition \ref{prop:first growing mode}.
\begin{proof}[Proof of Proposition \ref{prop:first growing mode}]
In view of Lemma \ref{lemma:uniform LY N}, the weak--strong convergence result of Lemma \ref{lem:MN_to_Minf} ensures we may apply the Keller--Liverani perturbation theorem \ref{thm:keller liverani}. Thus, if $\mathcal M_\infty$ has the non-zero eigenvalue $\lambda_g$, then for all $N$ large enough, $N^{-2}\mathcal M_{T_N,g}$ has an eigenvalue $\lambda_N\in\sigma(N^{-2}\mathcal M_{T_N,g})$ with $\lambda_N \to \lambda_g$. Setting $\Lambda_N:=N^2\lambda_N\in\sigma(\mathcal M_{T_N,g})$, we obtain an algebraically simple eigenvalue of $\mathcal M_{T_N,g}$, and for $N$ large enough, $|\Lambda_N|\geq \frac{N^2}{2}|\lambda_g|$, completing the proof.
\end{proof}
\subsection{The proof of time-periodic fast dynamo action}
\label{sec:time periodic proof}
We now have all the ingredients to prove Theorem \ref{thm:time-periodic}. Indeed, fix $f_1, f_2$ non-constant and real-valued, and pick $g \in C^\infty(\TT^2;\RR)$ so that 
$$
\int_{\TT^2}\e^{2 \pi i g(x,y)}f'_1(x)f'_2(y) \dd x \dd y \neq 0.
$$
Let $\chi:[0,\frac{1}{2}] \to [0,3]$ be a compactly supported, smooth function, with integral equal to $1$. Extend it by zero outside $[0,\frac{1}{2}]$. Then, let 
$$
u_N(t,x,y)=\chi(t)\begin{pmatrix}
0\\
Nf_1(x)\\
0
\end{pmatrix}+\chi(t-1/2)\begin{pmatrix}
N f_2(y)\\
0\\
0
\end{pmatrix}
+\chi(t-1)\begin{pmatrix}
0\\
0\\
-g(x,y)
\end{pmatrix},
$$
continued periodically in time, with period $2$. The time-$2$ flow map of $u_N$ is then exactly
$$
\begin{pmatrix}
T_N(x,y)\\
z-g(T_N(x,y))
\end{pmatrix}.
$$
Furthermore, we now introduce the noisy flow map $\Phi_{u_N,\eps}^t(\omega,\cdot)$, given by the solution to the SDE
\begin{equation}
\label{eq:noisy flow map}
\dd \Phi_{u_N,\eps}^t(\omega,x)=u_N\bigl(t,\Phi_{u_N,\eps}^t(\omega,x)\bigr)\dd t+\sqrt{2 \eps}\dd W_t(\omega), \quad \Phi_{u_N,\eps}^0(\omega,x)=x,
\end{equation}
where $(W_t)_{t \geq 0}$ is a standard $3$-dimensional Brownian motion. We now compute this flow map explicitly. Define first $F_{1,\eps}^\omega,F_{2,\eps}^\omega,g_{\eps}^\omega$ as follows:
\begin{align*}
&F_{1,\eps}^\omega(x)=\int_0^{1/2}\chi(s)f_1(x+\sqrt{2\eps}W_s^1) \dd s,\\
&F_{2,\eps}^\omega(y)=\int_{1/2}^1 \chi(s-\frac{1}{2})f_2(y+\sqrt{2\eps}(W_s^{2}-W_{1/2}^2)) \dd s,\\
&g_{\eps}^\omega(x,y)=\int_{1}^{3/2}\chi(s-1)g(x+\sqrt{2\eps}(W_s^1-W_1^1),y+\sqrt{2 \eps}(W_s^2-W_1^2)) \dd s.
\end{align*}
Then, we have that $\Phi_{u_N,\eps}^2(\omega,x,y,z)$ is equal to
$$
\begin{pmatrix}
x+\sqrt{2 \eps}W_2^1+NF_{2,\eps}^\omega(y+\sqrt{2 \eps}W_{\frac{1}{2}}^2+NF_{1,\eps}^\omega(x))\\
y+\sqrt{2\eps}W_{2}^2+NF_{1,\eps}^\omega(x)\\
z+\sqrt{2 \eps}W_{2}^3-g_\eps^\omega(x+\sqrt{2 \eps}W_{1}^1 +NF_{2,\eps}^\omega(y+\sqrt{2 \eps}W_{\frac{1}{2}}^2+NF_{1,\eps}^\omega(x)),y+\sqrt{2 \eps}W_{1}^2+NF_{1,\eps}^\omega(x))
\end{pmatrix}.
$$
Write this in block form as 
$$
\Phi_{u_N,\eps}^2(\omega,x,y,z)=\begin{pmatrix}
T_{N,\eps}^\omega(x,y)\\
z+\sqrt{2 \eps}W_2^3-\gamma_{\eps}^\omega(x,y)
\end{pmatrix}.
$$
As such, for an initial datum of the form $B(x,y,z)=\e^{2 \pi i z}B_0(x,y)$, the time-$2$ dynamo operator becomes 
\begin{equation}
\label{eq:time periodic noisy dynamo operator}
\e^{2 \pi i z}\e^{-8 \pi^2 \eps}\mathbb{E}_{W^1,W^2}\left [\left (\e^{2 \pi i \gamma_\eps^\omega(x,y)}\begin{pmatrix}
DT_{N,\eps}^\omega & 0\\
-\dd \gamma_\eps^\omega & 1
\end{pmatrix} B_0 \right )\circ ((T_{N,\eps}^\omega)^{-1}(x,y)) \right ].
\end{equation}
Notably, since $f_1,f_2$ are smooth, from the explicit form of $T_{N,\eps}^\omega$, we deduce that there exists a deterministic constant $C_{f_1,f_2}>0$ so that $\|DT_{N,\eps}^{\omega}\|_{L^\infty} \leq C_{f_1,f_2}N^2$ almost surely. Furthermore, since the equation admits a block-triangular structure, we shall focus only on the top $2 \times 2$ block, namely the operator 
$$
\mathcal{M}_{T_N,g}^\eps b_0:=\e^{-8\pi^2\eps}\mathbb{E}_{W^1,W^2}\left [\left (\e^{2 \pi i \gamma_\eps^\omega(x,y)}DT_{N,\eps}^\omega b_0   \right )\circ (T_{N,\eps}^\omega)^{-1}\right ].
$$
We then have the following result.
\begin{lemma}
\label{lemma:time-periodic perturbation}
Let $N$ be large enough. Then, for all $\eps>0$ small enough, $\mathcal{M}_{T_N,g}^\eps$ admits an unstable, isolated eigenvalue of modulus at least $\frac{c_0}{2}N^2$, where $c_0$ is the same as in Proposition \ref{prop:first growing mode}.
\end{lemma}
Applying Proposition \ref{prop:LY}, we obtain the following estimate with probability $1$:
$$
\left \|\left (\e^{2 \pi i \gamma_\eps^\omega(x,y)}DT_{N,\eps}^\omega b_0   \right )\circ (T_{N,\eps}^\omega)^{-1} \right \|_{X} \leq C_{f_1,f_2}N^{2-2r}\|b_0\|_{X}+C_W \|b_0\|_{Y},
$$
where 
$$
C_W=C_{f_1,f_2,r,\delta}(1+\|\nabla(\gamma_\eps^\omega\circ (T_{N,\eps}^{\omega})^{-1})\|_{L^\infty})N^2.
$$
From the explicit expression of $\gamma_{\eps}^\omega$, it follows that 
$$
\|\nabla(\gamma_\eps^\omega\circ (T_{N,\eps}^{\omega})^{-1})\|_{L^\infty} \leq \|\nabla g_{\eps}^\omega\|_{L^\infty} \leq C(g).
$$
Thus, the operator $N^{-2}\mathcal{M}_{T_N,g}^\eps$ satisfies the Lasota--Yorke inequality
$$
\|N^{-2}\mathcal{M}_{T_N,g}^\eps b_0\|_{X} \leq C N^{-2r}\|b_0\|_{X}+C\|b_0\|_{Y},
$$
uniformly for $\eps \in [0,1)$. Similarly, Proposition \ref{prop:LY} implies that it is uniformly bounded on $Y$ for $\eps \in [0,1)$. Finally, it remains to estimate
\begin{align*}
\|\mathcal{M}_{T_N,g}^\eps b-\mathcal{M}_{T_N,g}b\|_{Y}.
\end{align*}
The contribution of the scalar factor is harmless because $|\e^{-8\pi^2\eps}-1|\to0$ and the expectation in the definition of $\mathcal M_{T_N,g}^\eps$ is uniformly bounded. Let first $b_0 \in C^\infty(\TT^2;\CC^2)$. Then, via the triangle inequality it remains to bound
\begin{align*}
&\mathbb{E}\|\e^{2 \pi i \gamma_\eps^\omega \circ (T_{N,\eps}^\omega)^{-1}}\left ((DT_{N,\eps}^\omega b_0)\circ (T_{N,\eps}^\omega)^{-1}- (DT_N b_0)\circ T_N^{-1} \right )\|_{H^{1}}\\
&\qquad\qquad+\mathbb{E} \|(\e^{2 \pi i \gamma_\eps^\omega \circ (T_{N,\eps}^\omega)^{-1}}-\e^{2 \pi i g})((DT_N b_0)\circ T_N^{-1})\|_{H^1}.
\end{align*}
We proceed term by term. The first of these terms picks up a factor uniformly bounded by a deterministic constant depending only on $g$, since $\|\nabla(\gamma_\eps^\omega\circ (T_{N,\eps}^{\omega})^{-1})\|_{L^\infty}$ is uniformly bounded by $C(g)$. Hence, we need only bound 
\begin{align*}
\|(DT_{N,\eps}^\omega b_0)\circ (T_{N,\eps}^\omega)^{-1}- (DT_N b_0)\circ T_N^{-1} \|_{H^1}& \leq \|(DT_{N,\eps}^\omega b_0-DT_N b_0)\circ T_N^{-1}\|_{H^1}\\
&\quad +\|(DT_{N,\eps}^\omega b_0)\circ (T_{N,\eps}^\omega)^{-1}-(DT_{N,\eps}^\omega b_0)\circ T_N^{-1}\|_{H^1}.
\end{align*}
Applying Lemma \ref{lemma:smooth sobolev multiplier}, we bound the first of these terms by 
$$
C_{T_N}\|b_0\|_{C^1}\|DT_{N,\eps}^\omega-DT_{N}\|_{C^1(\TT^2)}.
$$
Similarly, the second is bounded by 
\begin{align*}
    \|(DT_{N,\eps}^\omega b_0) &\circ (T_{N,\eps}^\omega)^{-1}- (DT_{N,\eps}^\omega b_0)\circ T_N^{-1}\|_{H^1}  
    \\
    & \leq C_{f_1,f_2} \left \|\int_{0}^1 \nabla (DT_{N,\eps}^\omega b_0)((1-s)T_N^{-1}+s(T_{N,\eps}^\omega)^{-1}) \cdot ((T_{N,\eps}^\omega)^{-1}-T_N^{-1}) \dd s \right \|_{H^1} 
    \\
    & \leq C_{f_1,f_2,N} \|b_0\|_{C^2}\|(T_{N,\eps}^{\omega})^{-1}-T_N^{-1}\|_{C^1}.
\end{align*}
Finally, we bound
$$
\|(\e^{2 \pi i \gamma_\eps^\omega \circ (T_{N,\eps}^\omega)^{-1}}-\e^{2 \pi i g})((DT_N b_0)\circ T_N^{-1})\|_{H^1}\leq C_{T_N}\|b_0\|_{H^1}\|\e^{2 \pi i \gamma_\eps^\omega \circ (T_{N,\eps}^\omega)^{-1}}-\e^{2 \pi i g}\|_{C^1}.
$$
If we can show that all of these quantities tend to zero in expectation as $\eps \to 0$, we then deduce the desired weak-strong convergence: indeed, since $\|\mathcal{M}_{T_N,g}^\eps b-\mathcal{M}_{T_N,g}b\|_{Y}$ is uniformly bounded for $b \in Y$, converges to zero for all $b$ in a dense subset of $Y$, and $X$ embeds compactly into $Y$, it follows as in the proof of Lemma \ref{lem:MN_to_Minf} that 
$$
\|N^{-2}\mathcal{M}_{T_N,g}^\eps -N^{-2} \mathcal{M}_{T_N,g}\|_{X \to Y} \to 0,
$$
as $\eps \to 0$. From here, the Keller--Liverani theorem \ref{thm:keller-liverani 1} combined with Proposition \ref{prop:first growing mode} implies that for all $N>0$ large enough, and all $\eps>0$ small enough, $\mathcal{M}_{T_N,g}^\eps$ has an unstable eigenvalue of modulus at least $\frac{c_0}{2}N^2$. The factor $\e^{-8 \pi^2 \eps}$ is already included in the definition of $\mathcal M_{T_N,g}^\eps$ and tends to $1$ as $\eps \to 0$. Hence, it remains to prove the following lemma.
\begin{lemma}
There holds 
$$
\mathbb{E}\|(T_{N,\eps}^\omega)^{-1}-T_N^{-1}\|_{C^2} \to 0,
$$
as $\eps \to 0$. Similarly, 
$$
\mathbb{E}\|\e^{2 \pi i \gamma_\eps^\omega \circ (T_{N,\eps}^\omega)^{-1}}-\e^{2 \pi i g}\|_{C^1} \to 0,
$$ 
as $\eps \to 0$.
\end{lemma}
\begin{proof}
We compute explicitly
\begin{align*}
\left((T_{N,\eps}^{\omega})^{-1}(x,y)\right)_1
&=
x-\sqrt{2\eps}\,W_2^1
-NF_{2,\eps}^{\omega}\left(
y-\sqrt{2\eps}\left(W_2^2-W_{1/2}^2\right)
\right),\\
\left((T_{N,\eps}^{\omega})^{-1}(x,y)\right)_2
&=
y-\sqrt{2\eps}\,W_2^2\\
&\quad
-NF_{1,\eps}^{\omega}\left(
x-\sqrt{2\eps}\,W_2^1
-NF_{2,\eps}^{\omega}\left(
y-\sqrt{2\eps}\left(W_2^2-W_{1/2}^2\right)
\right)
\right).
\end{align*}
We now work on the event $A_\eps=\{\sup_{t \in [0,2]}|W_t|\leq \eps^{-1/4}\}$. For any $k\in\NN$, the mean value theorem gives, on this event,
\begin{equation}
\label{eq:good event}
\|F_{1,\eps}^\omega-f_1\|_{C^k}+\|F_{2,\eps}^\omega-f_2\|_{C^k}+\|g_{\eps}^{\omega}-g\|_{C^k}\leq C_{k,f_1,f_2,g}\eps^{1/4}.
\end{equation}
Indeed, this follows immediately from the explicit representations of $F_{1,\eps}^\omega, F_{2,\eps}^\omega,g_{\eps}^{\omega}$. Hence, we now compute the first component of $(T_{N,\eps}^{\omega})^{-1}-T_N^{-1}$, which is given by
$$
-\sqrt{2 \eps}W_{2}^1-NF_{2,\eps}^\omega(y-\sqrt{2\eps}(W_2^2-W_{1/2}^2))+Nf_2(y).
$$
Taking $C^2$ norms and employing \eqref{eq:good event}, we bound this by $C\eps^{1/4}$ on $A_\eps$. Similarly, the same holds for the second component of $(T_{N,\eps}^{\omega})^{-1}-T_N^{-1}$. Finally, since 
$$
\gamma_\eps^\omega \circ (T_{N,\eps}^\omega)^{-1}(x,y)=g_{\eps}^{\omega}(x-\sqrt{2\eps}(W_2^1-W_1^1),y-\sqrt{2 \eps}(W_2^2-W_1^2)),
$$
another application of the mean-value theorem, combined with \eqref{eq:good event}, shows that, on $A_\eps$, all quantities we care about are bounded by $C\eps^{1/4}$.
The explicit inverse-map formulas also give the global bound
$$
\|(T_{N,\eps}^\omega)^{-1}-T_N^{-1}\|_{C^2}
\leq C_{N,f_1,f_2}(1+\sqrt{\eps}|W_2|),
$$
since the shear terms and their first two derivatives are uniformly
bounded. Thus, splitting the expectation over $A_\eps$ and
$A_\eps^c$, and applying Cauchy--Schwarz, we obtain
\begin{align*}
\mathbb{E}\|(T_{N,\eps}^\omega)^{-1}-T_N^{-1}\|_{C^2}
&\leq C\eps^{1/4}\mathbb{P}(A_\eps)
+C\mathbb{P}(A_\eps^c)
+C\sqrt{\eps}\,
\mathbb{E}\bigl[|W_2|\mathbf{1}_{A_\eps^c}\bigr]\\
&\leq C\left(
\eps^{1/4}+\mathbb{P}(A_\eps^c)
+\sqrt{\eps}\,\mathbb{P}(A_\eps^c)^{1/2}
\right),
\end{align*}
where we used $\mathbb{E}|W_2|^2<\infty$.
Here the constants may depend on the fixed $N,f_1,f_2,g$,
but are independent of $\eps$.
Since $\mathbb{P}(A_\eps^c)\leq\eps^{1/4}$ for all sufficiently
small $\eps>0$,\footnote{Of course, this probability is actually
exponentially small in $\eps^{-1/2}$, but we do not need this
stronger statement.}
the right-hand side tends to zero.
Finally, an identical argument controls
$$
\mathbb{E}\|\e^{2 \pi i \gamma_\eps^\omega \circ (T_{N,\eps}^\omega)^{-1}}-\e^{2 \pi i g}\|_{C^1},$$
and so we complete the proof.
\end{proof}
Hence, we may now provide a proof of Theorem \ref{thm:time-periodic}.
\begin{proof}[Proof of Theorem \ref{thm:time-periodic}]
In view of Lemma \ref{lemma:time-periodic perturbation}, it remains to construct the third component of \eqref{eq:time periodic noisy dynamo operator}. Write $B_0=(b_0,b_{0,v})$. Here, we note that the map
$$
b_{0,v} \mapsto \e^{-8\pi^2\eps}\mathbb{E}\left[\left(\e^{2 \pi i \gamma_\eps^\omega}b_{0,v}\right)\circ (T_{N,\eps}^\omega)^{-1}\right]
$$
has operator norm bounded by $1$ on $L^2(\TT^2)$. Let $\Lambda_{N,\eps}$ be the unstable eigenvalue of $\mathcal{M}_{T_N,g}^\eps$. Note that by standard parabolic smoothing, $b_{0}$ is $C^\infty$, and thus, since
$$
b_{0,v}\mapsto \Lambda_{N,\eps}b_{0,v}-\e^{-8\pi^2\eps}\mathbb{E}\left[\left(\e^{2 \pi i \gamma_\eps^\omega}b_{0,v}\right)\circ (T_{N,\eps}^\omega)^{-1}\right]
$$
is invertible on $L^2$, we may find $b_{0,v} \in L^2(\TT^2)$ solving
$$
\Lambda_{N,\eps}b_{0,v}-\e^{-8\pi^2\eps}\mathbb{E}\left[\left(\e^{2 \pi i \gamma_\eps^\omega}b_{0,v}\right)\circ (T_{N,\eps}^\omega)^{-1}\right]
=-\e^{-8\pi^2\eps}\mathbb{E}\left[\left(\e^{2 \pi i \gamma_\eps^\omega}\dd \gamma_\eps^\omega b_0\right)\circ (T_{N,\eps}^\omega)^{-1}\right],
$$
hence finding an unstable eigenmode for the full dynamo operator. Finally, taking the divergence of \eqref{passive-vector}, we see that $\nabla \cdot B$ solves
$$
\partial_t (\nabla \cdot B)+u \cdot \nabla (\nabla \cdot B)=\eps \Delta (\nabla \cdot B).
$$
Hence, the $L^2$ norm of the divergence is non-increasing in time. Finally, by the eigenfunction condition, we have $(\nabla \cdot B)(2)=\Lambda_{N,\eps} (\nabla \cdot B)(0)$. Since $|\Lambda_{N,\eps}|>1$, this implies $(\nabla \cdot B)(0)=0$, and so the same is true for all positive times.
\end{proof}

\section{The Poincar\'e section framework}
\label{sec:framework autonomous}
In this section, we detail the framework necessary to prove autonomous, smooth fast dynamo action. The framework is motivated and introduced  in Subsection \ref{subsec:autonomous-framework} and it is based on the idea of the so-called Poincar\'e return map; see for instance \cite{KH95} for a classical discussion of this concept.
We treat the $z$-coordinate as an evolution coordinate $s$ and study the Poincaré return time and return map on the section $\{z=0\}$.
More precisely, we extend a vector field periodically from $\TT^3$ to $\TT^2\times[0,\infty)$, using coordinates $(x,y,z)$ on the former and $(a,s)$ on the latter, and study the eigenvalue problem \eqref{eq:elliptic evolution}
$$ \eps \Delta B - u\cdot\nabla B+B\cdot\nabla u = \lambda B\,.$$
Since $u_3>0$, dividing by $u_3$ allows us to view this PDE as an evolution equation in $s\in[0,\infty)$, perturbed by the $\eps \partial_s^2 B$ term
$$ 
\partial_s B +\frac{u_{1,2}}{u_3}\cdot\nabla_a B - \frac{1}{u_3}  (\nabla u - \lambda \Id) B =  \frac{\varepsilon}{u_3}\Delta_a B +\frac{\varepsilon}{u_3}\partial_s^2 B \,.
$$
Due to the presence of a second derivative in $s$, we require a suitable growth bound as $s\to \infty$ to ensure existence and uniqueness of solutions to this PDE when we impose an initial condition $B_0 : \TT^2 \to \CC^3$ at $s=0$\footnote{We remark that this corresponds to the Poincaré section $\{z=0\}$.}, see Lemma \ref{lemma:notion of solution}.  We denote by $\cK_{\lambda, \eps} B_0$  this unique  solution evaluated at the section $\{s=1\}$, namely
\begin{align} \label{d:K-lambda-eps}
    \cK_{\lambda, \eps } B_0 = {\rm tr} (B)|_{s=1} \,.
\end{align}
The aim of this section is to construct a map $\mathcal{E}_u$ such that conjugating $\cK_{\lambda,0}$ by $\mathcal{E}_u$ yields an operator
$$ \mathcal{D}_{u,\lambda}= \mathcal{E}_u^{-1}\circ\cK_{\lambda,0}\circ\mathcal{E}_u 
$$
which is of the ``SFS'' form studied in Section \ref{sec:periodic-fast-dynamo}. We then derive the corresponding conjugation formula for $\eps>0$, introducing the operator 
$\cU_{\lambda, \eps}$ so that 
$$ \cK_{\lambda, \eps} = \mathcal{E}_u \circ (\mathcal{D}_{u,\lambda} \mathcal{U}_{\lambda,\eps}) \circ \mathcal{E}_u^{-1}\,. $$

\subsection{The decomposition of $\cK_{\lambda, 0}$}
We consider vector fields $u : \TT^3 \to \RR^3$ such that one component, which we take to be  $u_3$ for convenience, is \emph{strictly} positive. The key property this implies is that the flow generated by $u$, which we denote by $\Phi_u^t$, intersects the surface $\{z=1\}$ transversally, so that we may study the Poincar\'e return map to the section $\{z=1\}$ of such vector fields. To do so, we first extend $u$ to the domain $(a,s)\in   \TT^2 \times [0,\infty)$ by periodicity and study the flow map $\Phi_u^t$ on this domain. 
We define the $m^{\mathrm{th}}$ return time $\overline\tau^m_u(a)$ of $\Phi_u^t$ to the section $\{s=m\}$ as 
$$
\overline\tau_u^m(a)=\inf\{t \geq 0:(\Phi_u^t(a,0))^{(3)}=m\}\,.
$$
We note that, since $\Phi_u^t$ intersects the sections $\{s=m\}$ transversally, the return time $\overline\tau^m_u$ is well-defined and smooth (see \cite{LermanYakovlev2019}).

We now move to a systematic study of $\Phi_u^t$ and $\overline \tau^m_u$.
We begin by defining the two-dimensional velocity field $\w: \TT^2 \times [0,\infty) \to \RR^2$ as 
$$ \w (a,s)=\frac{u_{1,2}}{u_3} (a,s) \,,$$
where we use $s\in [0,\infty)$ as the time variable and $a \in \TT^2$ as the base point.
We denote by $\Phi_{\tilde u_{1,2}}^s : \TT^2 \to \TT^2$ its flow map
$$\partial_s \Phi_{\w}^s=\w (\Phi_\w^s,s)\,, \qquad \Phi_{\w}^0(a)=a\,.$$
For any $s \in [0,\infty)$, the flow map is a smooth diffeomorphism thanks to the smoothness of $u$ and the fact that $u_3$ is strictly positive.
We then have the following elementary lemma.

\begin{lemma}
\label{lemma:properties autonomous 1}
Let $u : \TT^2 \times [0,\infty) \to \RR^3$ be a divergence-free, smooth vector field, $1$-periodic in $s\in[0,\infty)$, such that $u_3 >0$. Let  $\w: \TT^2 \times [0, \infty) \to \RR^2$ be defined as $\w(a,s)= \frac{u_{1,2}}{u_3} (a,s)$. Then for any $m \in \NN$ we have 
\begin{align} \label{eq:identity-return-mt}
    \overline\tau_u^m(a) = \int_0^m \frac{1}{u_3(\Phi_{\w}^s(a),s)}\,\dd s \,, \qquad \Phi_u^{\overline\tau_u^m(a)}(a,0) = (\Phi_{\w}^m(a),m )\,,
\end{align}
and the pushforward identity $(\Phi_\w^m)_\# (u_3 (\cdot, 0) \dd a)  = u_3 (\cdot, 0) \dd a$.
\end{lemma}

\begin{proof}
    For fixed $a\in\TT^2$, we define
$$
Y_a(s)=\bigl(\Phi_{\w}^s(a),s\bigr) \,, \qquad \theta_a(s) = \int_0^s \frac{1}{u_3(\Phi_{\w}^\sigma(a),\sigma)}\,\dd\sigma.
$$
Since $u_3>0$, the map $\theta_a$ is strictly increasing and we can define $s_a= \theta_a^{-1}$ and  $X_a(t)=Y_a(s_a(t))$. Then, differentiating the identity $\theta_a(s_a(t))=t$, we have 
$$ s_a'(t) = u_3(\Phi_{\w}^{s_a(t)}(a),s_a(t)), $$
and therefore
\begin{align*}
\frac{\dd X_a}{\dd t} &=
\begin{pmatrix}
\w(\Phi_{\w}^{s_a(t)}(a),s_a(t))
\\
1
\end{pmatrix}
s_a'(t) = u(X_a(t))\,.
\end{align*}
Since $X_a(0)=(a,0)$, uniqueness for the flow of $u$ gives
$ \Phi_u^{\theta_a(s)}(a,0) = \bigl(\Phi_{\w}^s(a),s\bigr)\,, $
which directly implies \eqref{eq:identity-return-mt} by taking $s=m$. 
It remains to prove the pushforward identity. We denote the Jacobian by $ J_s(a)=\det D_a\Phi_{\w}^s(a)$ and by  Liouville's formula we have 
$$
\partial_s J_s(a) = (\nabla_a\cdot\w)(\Phi_{\w}^s(a),s)J_s(a).
$$
Therefore, we deduce
\begin{align*}
\frac{\dd}{\dd s}\left ( u_3(\Phi_{\w}^s(a),s)J_s(a) \right)
&= \left ( \partial_s u_3 +\w\cdot\nabla_a u_3 +u_3\nabla_a\cdot\w \right )(\Phi_{\w}^s(a),s)J_s(a)\\
&= \left ( \partial_s u_3+\nabla_a\cdot u_{1,2} \right ) (\Phi_{\w}^s(a),s)J_s(a) =0 \,,
\end{align*}
where we have used $\nabla\cdot u=0$. Hence, we have $u_3(\Phi_{\w}^s(a),s) \det D_a\Phi_{\w}^s (a) = u_3(a,0)$ which gives the pushforward identity $(\Phi_{\w}^s)_\# (u_3(\cdot,0)\,\dd a) = u_3(\cdot,s)\,\dd a$. Taking $s=m\in\NN$ and using the periodicity of $u_3$ in $s$ gives the desired identity.
\end{proof}
We observe that in general the return map $\Phi^1_\w : \TT^2\to \TT^2$ is not Lebesgue-measure-preserving: indeed by Lemma \ref{lemma:properties autonomous 1} it preserves the weighted area measure $ u_3(a,0)\,\dd a.$ Consequently, the Lasota--Yorke estimates developed in the time-periodic setting for Lebesgue-area-preserving maps cannot be applied directly to $\Phi_\w^1$. We therefore conjugate $\Phi_\w^1$ to an area-preserving map. We denote the average by 
$$ \overline u_3=\int_{\TT^2}u_3(a,0)\,\dd a \,,$$
and since  the densities $u_3(\cdot,0) $ and $\overline u_3$ are smooth, strictly positive, and have the same total mass we can apply Moser's lemma \ref{lemma:quantitative-moser} to find  a smooth diffeomorphism $\kappa_u\colon\TT^2\to\TT^2$ such that
\begin{equation} \label{eq:moser relation}
(\kappa_u)_\#
\left (\overline u_3 \dd a\right ) = u_3(\cdot,0) \dd a\,.
\end{equation}
We now define
\begin{align} \label{d:P-u-tau-u}
    P_u = \kappa_u^{-1}\circ \Phi_\w^1 \circ\kappa_u, \qquad \tau_u^m = \overline\tau_u^m\circ\kappa_u\,.
\end{align}
Using Lemma \ref{lemma:properties autonomous 1} and the pushforward rule we obtain
\begin{align*}
(P_u)_\# \left (\overline u_3\,\dd a \right) = (\kappa_u^{-1})_\# (\Phi_\w^1)_\# (\kappa_u)_\# \left (\overline u_3 \dd a\right ) = \overline u_3\,\dd a\,.
\end{align*}
Since $\overline u_3$ is a positive constant, $P_u: \TT^2 \to \TT^2$ preserves the Lebesgue measure. Moreover, since $P_u$ is conjugate to the orientation-preserving map $\Phi_\w^1$, it is itself orientation-preserving.
We observe that for any $m \in \NN$ it holds that 
\begin{align} \label{eq:tau-property}
    \tau_u^m(a)=\sum_{j=0}^{m-1} \tau_u \circ P_u^{j}(a) \,.
\end{align} 
We now consider $u, v \in C^\infty_\sigma$ with $u_3,v_3>0$; hence, we can find $c_0>0$ such that $\min\{ v_3, u_3\} \geq c_0>0$. Furthermore, we assume that $\| u \|_{C^{k+2}}+ \| v \|_{C^{k+2}} \leq M$.
Then, from Lemma \ref{lemma:quantitative-moser} for any $k , m \in \NN$ there exists $C_{k,c_0,M, m }>0$ such that the following holds true
\begin{equation}
\label{eq:continuity in u}
\|P_{u}^m -P_{v}^m\|_{C^k}+\|\tau^m_{u}-\tau^m_{v}\|_{C^k}+\|\kappa_{u}-\kappa_{v}\|_{C^k} \leq C_{k,c_0,M,m}\|u- v\|_{C^{k+2}}.
\end{equation}
Finally, we also introduce a change of frame with the matrix  $E_u (a): \CC^3 \to \CC^3$ defined as
$$ E_u (a) = \begin{pmatrix}
     D\kappa_u(a) & u_{1,2} (\kappa_u (a))
     \\
     0 & u_3(\kappa_u (a))
 \end{pmatrix}$$
 that induces the operator $\mathcal{E}_u$ on vector fields $B \in C^\infty (\TT^2; \CC^3)$ defined as
 $$\mathcal{E}_u : B (\cdot ) \mapsto E_u(\kappa_u^{-1}(\cdot )) B(\kappa^{-1}_u(\cdot)) \,. $$

 \begin{lemma} \label{lemma:flow conjugation}
 Let $u \in C^\infty_\sigma( \TT^3;  \RR^3)$ with $u_3 >0$ on $\TT^3$. We recall the operator $\cK_{\lambda, 0}$ defined in \eqref{d:K-lambda-eps}. Then, the following  factorization holds true for any $m \in \NN$
      $$\cK_{\lambda,0}^m = \mathcal{E}_u \circ \mathcal{D}_{u,\lambda}^m \circ \mathcal{E}_u^{-1} \,,$$
      where the operator $\mathcal{D}_{u, \lambda}$ acts on vector fields $B_0 : \TT^2 \to \CC^3$ and takes   the form 
\begin{equation} \label{eq:inviscid section map}
    \mathcal{D}^m_{u,\lambda} B_0= \e^{-\lambda \tau_u^m \circ P_u^{-m}} 
    \begin{pmatrix} DP_u^m \circ P_u^{-m} & 0 
    \\
    -D \tau_u^m \circ P_u^{-m} & 1
    \end{pmatrix}
    B_0 \circ P_u^{-m},
\end{equation}
 \end{lemma}

 \begin{proof}
    Recalling that $\Phi_u^{\overline\tau_u^m(a)}(a,0) = (\Phi_{\w}^m(a),m )$ from \eqref{eq:identity-return-mt}, we deduce the identity
$$
\Phi_u^{\tau^m_u(a)}(\kappa_u(a),0) = (\kappa_u,m)\circ P^m_u(a).
$$
Differentiating with respect to $a\in\TT^2$ yields
$$
\partial_t\Phi_u^{\tau_u^m(a)}(\kappa_u(a),0)
\otimes D\tau^m_u(a) + D_x\Phi_u^{\tau^m_u(a)}(\kappa_u(a),0)
\begin{pmatrix}
D\kappa_u(a)
\\
0
\end{pmatrix}
= \begin{pmatrix}
D\kappa_u(P^m_u(a))
\\
0
\end{pmatrix}
DP^m_u(a).
$$
Using the definition of the flow and the periodicity of $u$ in its third variable gives
$$
u(\kappa_u(P^m_u(a)),0)\otimes D\tau^m_u(a)
+
D_x\Phi_u^{\tau^m_u(a)}(\kappa_u(a),0)
\begin{pmatrix}
D\kappa_u(a)\\
0
\end{pmatrix} = \begin{pmatrix}
D\kappa_u(P^m_u(a))\\
0
\end{pmatrix}
DP^m_u(a)\,.
$$
Rearranging, we obtain
$$
E_u(P^m_u(a)) \begin{pmatrix}
DP^m_u(a)\\
-D\tau^m_u(a)
\end{pmatrix}
= D_x \Phi_u^{\tau^m_u(a)}(\kappa_u(a),0)
\begin{pmatrix}
D\kappa_u(a)\\
0
\end{pmatrix}\,.
$$
To identify the remaining column, note that the autonomous nature
of $u$ gives the flow identity $ \Phi_u^{t+s}=\Phi_u^t\circ\Phi_u^s$ and 
differentiating with respect to $s$ and evaluating at $s=0$ yields
$$
u(\Phi_u^t(x))=D_x \Phi_u^t(x)u(x).
$$
Evaluating this identity at the space point $(\kappa_u(a),0)$ and at
$t=\tau_u^m(a)$, and again using periodicity, gives
$$
D_x\Phi_u^{\tau_u^m(a)}(\kappa_u(a),0)u(\kappa_u(a),0) = u(\kappa_u(P_u^m(a)),0).
$$
Combining these identities and recalling the definition of $E_u$, we obtain
$$
E_u(P_u^m(a))^{-1}
D\Phi_u^{\tau_u^m(a)}(\kappa_u(a),0)E_u(a) =
\begin{pmatrix}
DP_u^m(a) & 0\\
-D\tau_u^m(a) & 1
\end{pmatrix}\,,
$$
concluding the proof.
 \end{proof}

\subsection{The decomposition of $\cK_{\lambda, \eps}$}
We now aim to get the decomposition for $\eps >0$. We may define $\cU_{\lambda, \eps} $ simply by $\cU_{\lambda, \eps}= \mathcal{D}_{u, \lambda}^{-1} \mathcal{E}_u^{-1} \cK_{\lambda, \eps} \mathcal{E}_u$, using the well-posedness result given in Lemma \ref{lemma:notion of solution}. We may now give a strong form of the definition of the operator $\cU_{\lambda, \eps}$ as the solution operator of an elliptic PDE. We start by  considering the eigenvalue problem \eqref{eq:elliptic evolution}. We use the change of variables $\Psi_u(a,s)=(\Phi_\w^s(\kappa_u(a)),s)$, which is a smooth diffeomorphism of $\TT^2\times[0,\infty)$ onto itself. Then $\tilde B=B\circ\Psi_u$ satisfies 
$$
\partial_s \tilde B-\frac{1}{u_3\circ \Psi_u}(\nabla u \circ \Psi_u-\lambda )\tilde B-\frac{\eps}{\overline u_3}\partial_{i} \left (\det (D_a ( \Phi_\w^s \circ \kappa_u))((D\Psi_u)^{-1} (D\Psi)^{-T}_u)^{i,j} \partial_j \tilde B\right )=0\,.
$$
Here, $\bar u_3=\int_{\TT^2}u_3(x,y,s) \dd x \dd y$ is constant by incompressibility, and the indices $i,j$ range over $a_1,a_2,s$.
Furthermore, we may undertake another change of variables to get rid of the principal (in $\eps$) non-scalar part of this equation. Indeed, we define the matrix cocycle $C_{u,s,\lambda}(a)$ as the solution to the family of ODEs
\begin{equation}
\label{eq:matrix cocycle}
\partial_s C_{u,s,\lambda}=\frac{1}{u_3\circ \Psi_u}(\nabla u \circ \Psi_u-\lambda ) C_{u,s,\lambda}, \quad C_{u,0,\lambda}(a)=E_u(a).
\end{equation}
Thus, setting $\tilde B (a,s)=C_{u,s,\lambda}(a) v(a,s)$, we finally deduce the following equation for $v$:
\begin{equation}
\label{eq:full conjugated elliptic equation}
\partial_s v-\frac{\eps}{\overline u_3}C^{-1}_{u,s,\lambda}\partial_i \left (\det (D_a (\Phi_\w^s(\kappa_u) ))((D\Psi_u)^{-1} (D\Psi)^{-T}_u)^{i,j} \partial_j (C_{u,s,\lambda}v) \right )=0.
\end{equation}
Expanding this expression, we obtain an operator of the form $\cL_{\lambda,\eps}v=0$,  where 
\begin{equation}
\label{eq:abstract elliptic equation}
\cL_{\lambda,\eps}= \partial_s-\eps\left (\alpha^{0,0}\partial_s^2+2\alpha^{0,j}\partial_s \partial_j +\alpha^{i,j}\partial_i \partial_j+H^0 \partial_s +H^j \partial_j +H \right ),
\end{equation}
with coefficients $\alpha^{i,j}(a,s)$ that are real, smooth, symmetric, scalar functions and form a uniformly elliptic operator on any compact subset of $\TT^2 \times [0,\infty)$. Furthermore, $H^0,H^j, H$ are matrix-valued functions depending smoothly on $a,s$, and holomorphically on $\lambda$. We now recall the operator $\mathcal{E}_u B(a)=E_u(\kappa_u^{-1}(a)) B(\kappa^{-1}_u(a))$. We impose an initial condition $B_0 : \TT^2 \to \CC^3$ at $s=0$, which, after the change of variables and frame, corresponds to 
\begin{align} \label{d:initial-elliptic-pseudo}
    v(a,0)= \mathcal{E}_u^{-1} B_0 (a)=  E^{-1}_u(a)B_{0}(\kappa_u(a))\,.
\end{align}
Similarly, we observe that 
\begin{align*}
B(a,m)&=\tilde B (P_u^{-m}(\kappa^{-1}_u(a)),m)
=C_{u,m,\lambda}(P_u^{-m}(\kappa^{-1}_u(a)))v(P_u^{-m}(\kappa^{-1}_u(a)),m)\\
&=\left(C_{u,m,\lambda}\mathcal{U}^{[m]}_{\lambda,\eps}(\mathcal{E}_u^{-1}B_{0})\right)\left(P_u^{-m}(\kappa_u^{-1}(a))\right)\,.
\end{align*}
Finally, an explicit computation shows that the solution to \eqref{eq:matrix cocycle} is given by 
$$
C_{u,s,\lambda}(a)=\e^{-\lambda \int_0^s \frac{1}{u_3(\Psi_u(a,r))} \dd r} D_x \Phi_u^{\int_0^s \frac{1}{u_3(\Psi_u(a,r))} \dd r}(\kappa_u(a),0)E_u(a).
$$
In particular, at $s=m$, we observe 
$$
C_{u,m,\lambda}(a)=\e^{-\lambda \tau^m_u(a)}D_x \Phi_u^{\tau^m_u(a)}(\kappa_u(a),0)E_u(a)=\e^{-\lambda \tau^m_u(a)}E_u(P^m_u(a)) \begin{pmatrix}
DP^m_u(a) & 0 \\
-D \tau^m_u(a) & 1
\end{pmatrix},
$$
by Lemma \ref{lemma:flow conjugation}. Hence, we are now ready to deduce the factorisation we need for $\cK_{\lambda, \eps}$. 

We consider the solution to the equation \eqref{eq:elliptic evolution}, subject to the boundary condition ${\rm tr} (B)|_{s=0} =B_{0}$, and the growth condition of Lemma \ref{lemma:notion of solution}. Let $\mathcal{K}_{\lambda,\eps}B_{0}$ denote the trace of the solution at $s=1$ defined in \eqref{d:K-lambda-eps}. 
We define $\mathcal{U}^{[m]}_{\lambda,\eps}v_0 = {\rm tr (v)|_{s=m}}$, where $v$ is the unique solution (see Lemma \ref{lemma:full decomposition}) to
\begin{align} \label{d:PDE-U-lambda-eps}
    \begin{cases}
        \cL_{\lambda, \eps } v=0\,,
        \\
        {\rm tr} (v)|_{s=0} = v_0 \in \mathcal{X}\,,
        \\
        {\rm tr} (v)|_{s=T} = C_{u,T,\lambda}^{-1}(a) B( \Phi_\w^T (\kappa_u(a)),T) \,,
    \end{cases}
\end{align}
where $B: \TT^2 \times [0,\infty) \to \CC^3$ is the unique  solution to \eqref{eq:elliptic evolution} with ${\rm tr} (B)|_{s=0} =B_{0}=\mathcal{E}_u v_0$\, satisfying the growth condition of Lemma \ref{lemma:notion of solution}, and $T>m$. 

From the previous computations and Proposition \ref{prop:LY} we deduce the following result.

\begin{corollary} \label{cor:factorisation1}
    Let $\cK_{\lambda, \eps}$ be defined in \eqref{d:K-lambda-eps}. Then, for any $m \in \NN$ it admits the following decomposition 
    $$\mathcal{K}^m_{\lambda,\eps}=\mathcal{E}_u \circ \left ( \mathcal{D}^m_{u,\lambda}\mathcal{U}^{[m]}_{\lambda,\eps} \right ) \circ\mathcal{E}_u^{-1}\,.$$
    Furthermore, by Proposition \ref{prop:LY}  the operator $\mathcal{D}_{u,\lambda}$ satisfies the following quantitative \emph{Lasota--Yorke} inequality. More precisely, there exists $C_0 \geq 1$  independent of $u, \lambda$ such that if we define
    $$
    C_{LY}=C_0\e^{-\Re(\lambda)\inf_{a \in \TT^2}\tau^m_u(a)}\left (\|DP^m_u\|_{L^\infty}^{1-r}+\|DP_u^{m}\|_{L^\infty}^{r+\delta} \right ) \,,
    $$
    then,  there exists a constant  $C=C(\|DP^m_u\|_{L^\infty},\|\tau^m_u\|_{C^2} , \lambda) >0$ which depends continuously on its arguments, so that if $\Re(\lambda) \geq 0$  we have
\begin{align*}
\|\mathcal{D}^m_{u,\lambda} B\|_{\mathcal X} &\leq C_{LY}\|B\|_{\mathcal X} +C  \|B\|_{\mathcal Y} \,, \qquad \|\mathcal{D}^m_{u,\lambda} B\|_{\mathcal Y} \leq  C\|B\|_{\mathcal Y} \,.
\end{align*}
\end{corollary}

\section{A family of ideal dynamo vector fields}

Having constructed the framework in which we analyse autonomous vector fields, we now move on to defining a set of vector fields $\mathcal{D}$ for which we aim to establish fast dynamo action.

The first condition on $\mathcal{D}$ we require is that  $1 \in \sigma(\mathcal{D}_{u,\lambda})$ with $\Re (\lambda) >0$. 
This is analogous to the Birman–Schwinger principle: we consider a family of operators depending on the spectral parameter $\lambda$ and seek values of $\lambda$ for which $\mathcal{D}_{u,\lambda}\colon \mathcal X \to \mathcal X$ defined in \eqref{eq:inviscid section map} has $1$ as an eigenvalue. This condition is natural in view of the decomposition $\cK_{\lambda, 0} = \mathcal{E}_u \circ  \mathcal{D}_{u,\lambda}  \circ \mathcal{E}_u^{-1}$ given in Corollary \ref{cor:factorisation1}. Indeed, we aim to prove that solutions to the eigenvalue problem \eqref{eq:elliptic evolution} on the extended domain $(a,s)\in \TT^2 \times [0, \infty)$ with initial datum at the Poincaré  section $\{s=0\}$ given by $B_0 : \TT^2 \to \CC^3$ are $1$-periodic in $s$. Indeed, this is a necessary condition so that the solution on $\TT^2 \times [0, \infty)$ is well defined on $\TT^3$. More precisely, we impose
$$
\mathcal{K}_{\lambda,0}B_{0}=B_{0}\,.
$$
The second condition on $\mathcal{D}$ we impose is needed to treat the case of positive diffusivities $0< \eps \ll 1$.  Indeed, we shall prove a Lasota--Yorke inequality that is uniform in $\eps$, so that  we can apply the Keller--Liverani theorem \ref{thm:keller liverani} in combination with Rouché's theorem to establish fast dynamo action. The key technical result needed in order to undertake this program is Proposition \ref{prop:pseudodifferential abstract}, proved in Section \ref{sec:pseudodifferential result}.

Therefore, we define  the set of vector fields $\mathcal{D}$ as follows.

\begin{definition}
\label{def:D}
We define $\mathcal D$ as the set of all vector fields $u\in C^\infty_\sigma(\TT^3)$ with $u_3>0$ so that there exists
$\lambda\in\CC$ with $\Re(\lambda)>0$ for which the following hold:
\begin{enumerate}
    \item $1$ is an algebraically simple eigenvalue of $\mathcal{D}_{u,\lambda}:\mathcal{X}\to \mathcal{X}$ 
\item $C_0 \inf_{m\geq1}
\exp\left(
    -\Re(\lambda)\inf_{a\in\TT^2}\tau_u^m(a)
\right)
\left\|DP_u^m\right\|_{L^\infty}^{\,1-r}
<1,$ with $C_0$ defined in Corollary \ref{cor:factorisation1}.
\end{enumerate}
\end{definition}
Theorem \ref{thm:autonomous} then follows from the following proposition.
\begin{proposition}
\label{prop:abstract proposition fast dynamo}
The set $\mathcal{D}$ satisfies the following properties:
\begin{enumerate}
    \item $\mathcal{D}$ is non-empty.
    \item There exists $k\in\NN$ such that the set $\mathcal{D}$ is open in the $C^k$ topology on $C^\infty_\sigma(\TT^3)$.
    \item Any $u \in \mathcal{D}$ generates a fast dynamo on $\TT^3$.
\end{enumerate}
\end{proposition}
The remainder of this paper will be devoted to proving Proposition \ref{prop:abstract proposition fast dynamo}. In this section we prove the first two points of Proposition \ref{prop:abstract proposition fast dynamo}.  We begin by noting a simple consequence of Definition \ref{def:D} and Corollary \ref{cor:factorisation1}.
\begin{lemma}
\label{lemma:D implies lasota yorke}
Let $\bar u \in \mathcal{D}$ with $\bar \lambda \in \CC$ satisfying the assumption of Definition \ref{def:D}. Then, there exists a neighbourhood $\mathcal{V}$ of $(\bar u , \bar \lambda) \in C^\infty(\TT^3) \times \CC$, where $C^\infty_\sigma(\TT^3)$ is equipped with the $C^k$ topology, for $k$ large enough, such that the following holds true.  There exist constants $C, M, \eta >0$, such that the following Lasota--Yorke inequality holds true
$$
\|\mathcal{D}_{ u, \lambda}^m B\|_{\mathcal{X}} \leq C (1-\eta)^m \|B\|_{\mathcal{X}}+CM^m \|B\|_{\mathcal{Y}},
$$
for all $m \in \NN$ and for all $(u, \lambda)\in \mathcal{V}$. Consequently, by Proposition \ref{prop:hennion} it holds that    $r_{\mathrm{ess}}(\mathcal{D}_{u,\lambda})\leq 1-\eta<1$ for all $(u,\lambda) \in \mathcal{V}$.
\end{lemma}

\begin{proof}
We note that by \eqref{eq:tau-property} and the chain rule we have that the quantity
$$
a_m(u,\lambda)
:= \e^{-\Re(\lambda)\inf_a\tau_u^m(a)}
   \|DP_u^m\|_{L^\infty}^{1-r}
$$
is submultiplicative when $\Re (\lambda)>0$, namely $a_{m+n} (u, \lambda)\leq a_m (u, \lambda)a_n (u, \lambda)$. 
Thus, by the assumption, there exist $N\geq1$ and $\theta\in(0,1)$
such that $2C_0 a_N(\bar u, \bar \lambda)<\theta$.

Using \eqref{eq:continuity in u}, Corollary \ref{cor:factorisation1}, and the inequality $r+\delta<1-r$, we restrict to a sufficiently small $C^k\times\CC$ neighbourhood $\mathcal V$ of $(\bar u,\bar\lambda)$. Then
there exist constants $K,M>1$ such that, uniformly for
$(u,\lambda)\in\mathcal V$, the operator $\mathcal{D}_{u,\lambda}$ satisfies
$$
\|\mathcal{D}_{u, \lambda}^N B\|_{\mathcal X} \leq \theta\|B\|_{\mathcal X}+K\|B\|_{\mathcal Y},
\qquad
\|\mathcal{D}_{u, \lambda}^N\|_{\mathcal Y\to\mathcal Y} \leq M,
\qquad 
\max_{0\leq\ell<N}\|\mathcal{D}_{u, \lambda}^\ell\|_{\mathcal X\to\mathcal X}\leq K.
$$
Iterating the first estimate, we obtain, for all $j\in\NN$,
$$
\|\mathcal{D}_{u, \lambda}^{jN}B\|_{\mathcal X}
\leq \theta^j\|B\|_{\mathcal X}
 +  \frac{K}{M-1}M^j \|B\|_{\mathcal Y}.
$$
Writing $m=jN+\ell$, with $0\leq\ell<N$, and using the bound on $\mathcal{D}_{u, \lambda}^\ell$ to reabsorb finitely many terms in a constant $C$, we obtain the claimed Lasota--Yorke inequality.
Finally, Proposition \ref{prop:hennion} yields, uniformly on $\mathcal V$, $ r_{\mathrm{ess}}(\mathcal{D}_{u,\lambda}) \leq 1-\eta <1. $
\end{proof}
\subsection{$\mathcal{D}$ is non-empty.}
We begin by constructing an explicit family of vector fields that live in the set $\mathcal{D}$.
We first recall that by Proposition \ref{prop:first growing mode}, upon picking a suitable mean-free $g$, for any $N \in \NN$ sufficiently large there exists a map $T_N:\TT^2\to\TT^2$ isotopic to the identity such that the operator
$$
\mathcal{M}_{T_N,g}:b \mapsto \e^{2 \pi i g} (D T_N b)\circ T_N^{-1}(x,y)
$$
admits an algebraically simple, isolated eigenvalue $\Lambda_N \in \CC$ with $|\Lambda_N|\geq c_0 N^2$ for some $c_0 >0$ independent of $N$.
We now aim to construct a vector field $u \in C^\infty(\TT^3)$ whose Poincaré return map is $P_u= T_N$ and whose return time $\tau_u$ is tuned to generate the phase $\e^{2\pi ig }$. Fix now $N$ large enough so that $c_0N^2 \gg N^{2-2r}$. Then, we have the following result. 

\begin{lemma}
\label{lemma:autonomous mimicking time-periodic}
Let $T: \TT^2 \to \TT^2$ be a volume-preserving diffeomorphism that is isotopic to the identity, and let $g \in C^\infty(\TT^2;\RR)$ be mean-free. Then there exists $n_0 \in \NN$ such that for any   $n \geq n_0$, there exists $u^{(n)} \in C_\sigma^\infty(\TT^3)$ with $u^{(n)}_3>0$, so that 
\begin{enumerate}
    \item $P_{u^{(n)}}=T$,
    \item $\tau_{u^{(n)}}(P_{u^{(n)}}^{-1}(x,y))=1-\frac{g(x,y)}{n}$.
\end{enumerate}
\end{lemma}
\begin{proof}
Let $w(\cdot, t)$ be a smooth divergence-free velocity field on $\TT^2$ whose flow $\Phi^t_w$ satisfies $\Phi_w^1=T$. Choose a smooth nondecreasing function $\theta:[0,1]\to[0,1]$, equal to $0$ near $0$ and to $1$
near $1$, and reparametrise the time variable so that 
$$
\tilde w(t,x):=\theta'(t)w(\theta(t),x), \qquad \Phi_{\tilde w}^t:=\Phi^{\theta(t)}_w\,.
$$
Thus $\tilde w$ is divergence-free, vanishes near $t=0,1$,
and has flow $\Phi_{\tilde w}$ with $\Phi_{\tilde w}^0 =\Id$ and $\Phi_{\tilde w}^1=T$.

Choose $\eta\in C_c^\infty(0,1)$ with $\eta\geq0$ and
$\int_0^1\eta(z)\,\dd z=1$. For $n$ sufficiently large, we define $\rho_n : \TT^3 \to \RR$ and $v^{(n)}: \TT^3 \to \RR^3$ as follows
\begin{align*}
    \rho_n(x,y, z) & :=1-\frac{\eta(z)}{n}g \left (T((\Phi^z_{\tilde w})^{-1}(x,y))\right )>0 \,,
    \\
    v^{(n)}(x,y, z)& := \frac{1}{\rho_n(x,y, z)} \begin{pmatrix}
        \tilde w_1 (x,y, z)
        \\
        \tilde w_2 (x,y, z)
        \\
        1
    \end{pmatrix}\,.
\end{align*}
Following the computations given in Section \ref{subsec:autonomous-framework} we deduce that the Poincaré return time of $v^{(n)}$ is 
$$ \tau_{v^{(n)}} (x,y)= \int_0^1 \rho_n(\Phi^z_{\tilde w}(x,y),z)\,\dd z
=1-\frac{g(T(x,y))}{n},
$$
and its return map satisfies $\Phi_{v^{(n)}}^{\tau_{v^{(n)}}}(x,y,0)=(T(x,y),1)$.

Since $\tilde w$ is divergence-free, $\Phi^z_{\tilde w}$ and $T$ preserve Lebesgue measure, and since $g$ is mean-free, we have
$$
\nabla\cdot(\rho_n v^{(n)})=0, \qquad \int_{\TT^2}\rho_n(x,y, z)\,\dd x \, \dd y=1 \,.
$$
Thus $v^{(n)}$ preserves $\rho_n\,\dd x \, \dd y\,\dd z$, and
Corollary \ref{cor:foliated-moser} provides a smooth diffeomorphism $R_n(x,z)=(\kappa_{n,z}(x),z)$, equal to the identity in a neighbourhood of  $z=0$ and $z=1$,
such that
$$
(R_n)_\star(\rho_n\,\dd x\,\dd y\, \dd z)=\dd x\,\dd y\, \dd z.
$$
Consequently, the pushforward
$$
u^{(n)}:=(R_n)_\star v^{(n)}=(DR_n\,v^{(n)})\circ R_n^{-1}
$$
is smooth and divergence-free, with $u^{(n)}_3>0$. We now notice that the flows satisfy the following conjugation  property
$$ \Phi_{u^{(n)}}^t = R_n\circ \Phi_{v^{(n)}}^t \circ R_n^{-1} \,.
$$
Since $R_n$  is the identity in a neighbourhood of the section $\{z=0\}$, the conjugacy preserves both the section return map and the return time, concluding the proof. 
\end{proof}

With this construction established, we prove that the set $\mathcal{D}$ from Definition \ref{def:D} is non-empty.
\begin{lemma}
$\mathcal{D}$ is non-empty.
\end{lemma}
\begin{proof}
We choose $u=u^{(n)}$ as in Lemma \ref{lemma:autonomous mimicking time-periodic}, with $T=T_N$ and $g \in C^\infty(\TT^2;\RR)$ mean-free given by Proposition \ref{prop:first growing mode}. Let $\Lambda_N=\e^{\theta_N}$ denote the eigenvalue of the operator
$$
\mathcal{M}_{T_N,g}:b \mapsto \e^{2 \pi i g} (DT_N b)\circ T_N^{-1}
$$
provided by Proposition \ref{prop:first growing mode}. Writing $\lambda=2 \pi i n+\theta_N+\widetilde\theta$, the exponential factor in \eqref{eq:inviscid section map} becomes
$$
\e^{-\theta_N} \e^{-\widetilde\theta(1-\frac{g}{n})+\theta_N\frac{g}{n}} \e^{2 \pi i g}.
$$
Thus, for all sufficiently large $n$, Corollary \ref{cor:factorisation1} yields the Lasota--Yorke inequalities
\begin{equation}
\label{eq:m=1 lasota yorke inequality}
\|\mathcal{D}_{u,\lambda} B\|_{\mathcal X} \leq C (N^{-2r}+N^{2r+2\delta-2})\|B\|_{\mathcal X}+C\|B\|_{\mathcal Y}, \quad \|\mathcal{D}_{u,\lambda} B\|_{\mathcal Y} \leq C \|B\|_{\mathcal Y},
\end{equation}
uniformly for $\widetilde\theta$ in a sufficiently small neighbourhood of $0$, where we have used the bound $|\Lambda_N^{-1}|=|\e^{-\theta_N}|\leq \frac{1}{c_0}N^{-2}$.
We next consider the first two components of $\mathcal{D}_{u^{(n)},\lambda}$, which yield the reduced operator
\begin{equation}
\label{eq:reduced map D}
b \mapsto \e^{-\theta_N}\e^{-\widetilde\theta(1-\frac{g}{n})+\theta_N\frac{g}{n}}\mathcal{M}_{T_N,g}b.
\end{equation}
Our aim is to show that, for every sufficiently large $n$, there exists $\widetilde\theta$ close to zero such that this operator has $1$ as an eigenvalue. To this end, we apply the implicit function theorem, which we recall below in its Banach-space formulation for convenience.

\begin{theorem}[\cite{Deimling1985}*{Theorem~15.1 and Corollary~15.1}]
\label{thm:inverse function}
Let $W,Y,Z$ be Banach spaces, 
and let $ F:W\times Y\to Z $
be a $C^1$ map. Suppose that $(w_0,y_0) \in W\times Y$ satisfies $ F(w_0,y_0)=0, $
and that the Fr\'echet  derivative with respect to $w$,
$ D_w F(w_0,y_0):W\to Z \,, $ is a bounded linear isomorphism.
Then there exist open neighbourhoods $U \subset W$ of $w_0$ and
$V \subset Y$ of $y_0$, and a unique $C^1$ map
$$
\varphi:V \to U
$$
such that $ \varphi (y_0)=  w_0$
and
$$
F(\varphi(y), y)=0
\qquad
\forall y\in V \,.
$$
\end{theorem}
Let $b_0 \in X$ be the eigenfunction with eigenvalue $1$ of $\e^{-\theta_N}\cM_{T_N ,g}$ guaranteed by Proposition \ref{prop:first growing mode}. Let furthermore $\Pi_N$ be the Riesz projector onto the eigenspace corresponding to $1$ of $\e^{-\theta_N}\cM_{T_N ,g}$. Then, by standard functional calculus (see e.g. \cite{K76}), it follows that $\Id-\e^{-\theta_N}\cM_{T_N ,g}:(\Id-\Pi_N)X \to (\Id-\Pi_N)X$ is a bounded linear isomorphism. Furthermore by Hahn--Banach there exists $\ell_0 \in X^\star$ so that $\ell_0(b_0)=1$.

We now set, in the notation of the implicit function theorem, $Z=W=X \oplus \CC$, $Y=\RR$, and let 
$$
F(b,\widetilde\theta, s)=\begin{pmatrix}
\e^{-\theta_N}\e^{-\widetilde\theta(1-sg(x,y))+\theta_Nsg(x,y)}\cM_{T_N,g}b-b 
\\
\ell_0(b)-1
\end{pmatrix}.
$$
By construction, $F(b_0,0,0)=0$. Furthermore, it is clear that $F$ is $C^1$ and its Fr\'echet derivative in $w$ at $(b_0,0,0)$ is given by 
$$
D_{b,\widetilde\theta} F(b_0,0,0) \begin{pmatrix}
b\\
z
\end{pmatrix}
=\begin{pmatrix}
\e^{-\theta_N} \cM_{T_N,g} b-b-zb_0
\\
\ell_0(b)
\end{pmatrix}\,.
$$
Let us verify this defines a bounded linear isomorphism. Boundedness is immediate, and so by the open mapping theorem it suffices to check bijectivity. We first prove injectivity. Suppose that $D_{b,\widetilde\theta} F(b_0,0,0) \begin{pmatrix}
b\\
z
\end{pmatrix}=(0,0)$. Then, write $b=\Pi_N b+(\Id-\Pi_N)b=ab_0+(\Id-\Pi_N)b$. Hence, we see that 
$$
(\e^{-\theta_N}\cM_{T_N,g}-\Id)(\Id-\Pi_N)b-zb_0=0.
$$
Since $(\e^{-\theta_N}\cM_{T_N,g}-\Id)$ maps $(\Id-\Pi_N)X$ injectively into itself, it follows that $(\Id-\Pi_N)b=0$, $z=0$. Hence, $b=ab_0$. But then, since furthermore $\ell_0(b)=0$, it follows that $a=0$, proving injectivity. For surjectivity, let $(\widetilde b,\omega) \in X \oplus \CC$, and write $\widetilde b=(\Id-\Pi_N)\widetilde b+a_{\widetilde b}b_0$. Then, we see that for $b=(\Id-\Pi_N)b+a_b b_0$
$$
(\e^{-\theta_N}\cM_{T_N,g}-\Id)(\Id-\Pi_N)b-zb_0=(\Id-\Pi_N)\widetilde b+a_{\widetilde b} b_0.
$$
Hence, set $(\Id-\Pi_N)b=(\e^{-\theta_N}\cM_{T_N,g}-\Id)^{-1}(\Id-\Pi_N)\widetilde b$, and $z=-a_{\widetilde b}$. Finally, to determine $a_b$, we rearrange 
$$
\ell_0((\Id-\Pi_N)b)+a_b=\omega,
$$
to get $a_b=\omega-\ell_0((\Id-\Pi_N)b)$. 
Thus, the conditions of the implicit function theorem are satisfied. For all sufficiently large $n$, there exist $\widetilde\theta_n$ and $b_n\in X$, with $\widetilde\theta_n\to0$ and $b_n\to b_0$ as $n\to\infty$, such that 
$$
\e^{-\theta_N} \e^{-\widetilde\theta_n(1-\frac{g}{n})+\theta_N\frac{g}{n}}\cM_{T_N,g} b_n=b_n.
$$ 
Furthermore, since $\e^{-\theta_N} \e^{-\widetilde\theta_n(1-\frac{g}{n})+\theta_N\frac{g}{n}}\cM_{T_N,g} \to \e^{-\theta_N}\cM_{T_N,g}$ in operator norm as $n \to \infty$, it follows that the respective Riesz projectors corresponding to a contour surrounding the point $1 \in \CC$ are isomorphic for $n$ large enough, implying that the eigenvalue $1$ is simple for the operator $\e^{-\theta_N} \e^{-\widetilde\theta_n(1-\frac{g}{n})+\theta_N\frac{g}{n}}\cM_{T_N,g}$, for all $n$ sufficiently large.
Hence, we have proven the existence of the eigenvalue $1$ for the first two components of $\mathcal{D}_{u,\lambda}$. Thus, it remains to deal with the third component. Writing the full eigenvector as $B_n=(b_n,b_{n,v})$, we simply need to solve 
$$
\e^{-\theta_N} \e^{-\widetilde\theta_n(1-\frac{g}{n})+\theta_N\frac{g}{n}} \e^{2\pi i g } \left(\frac{1}{n}\nabla g \cdot DT_N(T_N^{-1})b_n\circ T_N^{-1}+b_{n,v}\circ T_N^{-1}\right)=b_{n,v}.
$$
But noting that for $n$ large, the proof of Proposition \ref{prop:LY} implies that the operator norm of $b_{n,v} \mapsto \e^{-\theta_N} \e^{-\widetilde\theta_n(1-\frac{g}{n})+\theta_N\frac{g}{n}} \e^{2\pi i g}(b_{n,v} \circ T_N^{-1})$ on $H^{-r-\delta}(\TT^2)$ is strictly less than $1$, we may recover $b_{n,v}$ via a Neumann series. Invertibility of the vertical block also shows that the algebraic simplicity of the eigenvalue is retained after introducing the vertical component. Together with the Lasota--Yorke inequality \eqref{eq:m=1 lasota yorke inequality}, this completes the proof.
\end{proof}

\subsection{$\mathcal{D}$ is open}
\label{sec:openness}
We next show that $\mathcal{D} \subset C^\infty_\sigma (\TT^3)$ is an open set where $C^\infty$ is equipped with the $C^k$ topology for $k$ large enough. This property turns out to be  delicate to establish, since the map $u \to \mathcal{D}_{u,\lambda}$ is in general not continuous in the operator norm topology on $\mathcal X$, but only weak-strong continuous. We start by proving the following key lemma.

\begin{lemma}
\label{lemma:perturbing velocity}
Let $u_1 \in \mathcal{D}$ and $u_2 \in C^\infty_\sigma$, and let $K \subset \CC$ be a compact set. Then, there exist $\eta >0$,  $k \in \NN$, and $C>0$ depending only on $u_1$ and $K$, such that, if $\|u_1-u_2\|_{C^k}<\eta$, the following estimates hold uniformly for $\lambda\in K$:
$$
\|(\mathcal{D}_{u_1,\lambda}-\mathcal{D}_{u_2,\lambda})B \|_{\mathcal Y} \leq C\|u_1-u_2\|_{C^{k}}^{\delta}\|B\|_{\mathcal X}.
$$
\end{lemma}

\begin{proof}
We first notice that 
$$
\mathcal{D}_{u_1,\lambda}-\mathcal{D}_{u_2,\lambda}=\e^{-\lambda \tau_{u_1}\circ P_{u_1}^{-1}}(\mathcal{D}_{u_1,0}-\mathcal{D}_{u_2,0})+(\e^{-\lambda \tau_{u_1}\circ P_{u_1}^{-1}}-\e^{-\lambda \tau_{u_2}\circ P_{u_2}^{-1}})\mathcal{D}_{u_2,0} \,.
$$
From Lemma \ref{lemma:multipliers anisotropic}, followed by \eqref{eq:continuity in u} it suffices to estimate $\mathcal{D}_{u_1, 0} - \mathcal{D}_{u_2, 0}$.

For a volume-preserving diffeomorphism $Q :\TT^2 \to \TT^2$, we define $\mathcal{T}_Q f=f\circ Q$. Furthermore, we write $\dd_\tau (\tau_1, \tau_2) = \|\tau_{u_1}-\tau_{u_2}\|_{C^k}$ and $\dd_P (Q_1, Q_2)= \|Q_1 -Q_2\|_{C^k}+ \|Q_1^{-1}-Q_2^{-1}\|_{C^k}$. We define $Q_s= Q_1 + s (Q_2 - Q_1)$ for $s\in [0,1]$ and if $\eta >0$ is sufficiently small we have that $\det (D Q_s)\geq 1/2$ for any $s\in [0,1]$. Therefore, we  can compute 
$$ f(Q_2 (x)) - f(Q_1 (x)) = \int_0^1 \nabla f (Q_s (x)) \cdot (Q_2(x) - Q_1(x)) ds\,.$$
Hence, if $\eta >0$ is sufficiently small we deduce by a change of variables that  
$$ \| \mathcal{T}_{Q_1} f - \mathcal{T}_{Q_2} f\|_{L^2} \leq 2 \| f\|_{L^2}\,, \qquad  \| \mathcal{T}_{Q_1} f - \mathcal{T}_{Q_2} f\|_{L^2} \leq 2 \dd_P (Q_1, Q_2) \| f\|_{H^1} \,.$$
Interpolating in the domain yields
$\|\mathcal{T}_{Q_1} f - \mathcal{T}_{Q_2} f\|_{L^2}\leq C \dd_P (Q_1,Q_2)^q\|f\|_{H^q}$ for $0\leq q\leq1$.
Furthermore, interpolating this with the uniform bound
$\|\mathcal{T}_{Q_1} f - \mathcal{T}_{Q_2} f\|_{H^q}\leq C\|f\|_{H^q}$ gives
\begin{equation}
\label{eq:perturbation-positive-composition}
\|\mathcal{T}_{Q_1} f - \mathcal{T}_{Q_2} f\|_{H^{q-\delta}}
\leq C \dd_P(Q_1,Q_2)^\delta\|f\|_{H^q},
\qquad \delta\leq q\leq1.
\end{equation}
For negative regularities, volume preservation yields
$$
\langle \mathcal{T}_{Q_1} f - \mathcal{T}_{Q_2} f,\varphi\rangle =
\langle f,\varphi\circ Q_1^{-1}-\varphi\circ Q_2^{-1}\rangle.
$$
Applying \eqref{eq:perturbation-positive-composition} with
$q=s+\delta$ and arguing by duality, we obtain
\begin{equation}
\label{eq:perturbation-negative-composition}
\|\mathcal{T}_{Q_1} f - \mathcal{T}_{Q_2} f\|_{H^{-s - \delta}} \leq C d_P^\delta (Q_1, Q_2)\|f\|_{H^{-s}},
\qquad s\in\{r,r+\delta\}.
\end{equation}
For a diffeomorphism $Q:\TT^2\to\TT^2$, we define $\cA_{Q,\tau}B=\begin{pmatrix}DQ&0\\-D\tau&1\end{pmatrix}B$
and $\cA_Qb=DQb$.
We set $P_i=P_{u_i}$ and $\tau_i=\tau_{u_i}$.
Then we have the uniform Sobolev multiplier estimates for any $s \in [-r- \delta, r+\delta]$
\begin{equation}\label{eq:estimate-A-0}
     \| \cA_{P_1} b \|_{H^s} \leq C \| b \|_{H^s}\,, \qquad \| (\cA_{P_1}- \cA_{P_2})  b \|_{H^s} \leq C \dd_P (P_1, P_2)   \| b \|_{H^s} \,,
\end{equation}
and 
\begin{equation} \label{eq:estimate-A}
    \| \cA_{P_1, \tau_1} B \|_{H^s} \leq C \| B \|_{H^s}\,, \qquad \| (\cA_{P_1, \tau_1}- \cA_{P_2, \tau_2})  B \|_{H^s} \leq C (\dd_P (P_1, P_2) + \dd_{\tau} (\tau_1, \tau_2) ) \| B \|_{H^s} \,.
\end{equation}
Therefore,  we have 
\begin{align*}
    \mathcal{D}_{u_1, 0} - \mathcal{D}_{u_2, 0} & = \mathcal{T}_{P_1^{-1}}\circ  \cA_{P_1, \tau_1} - \mathcal{T}_{P_2^{-1}}\circ  \cA_{P_2, \tau_2}
    \\
    &= (\mathcal{T}_{P_1^{-1}} - \mathcal{T}_{P_2^{-1}})\circ  \cA_{P_1, \tau_1}   + \mathcal{T}_{P_2^{-1}}\circ  ( \cA_{P_1 , \tau_1} - \cA_{P_2, \tau_2}) \,.
\end{align*}
Let $\Pi_h:\CC^3\to\CC^2$ denote the projection onto the first two components, and write $B=(b,B_v)\in\CC^3$, with $b\in\CC^2$.
We start by analysing the first two components. We notice that $\Pi_h\cA_{P,\tau}=\cA_P\Pi_h$. Hence, we combine \eqref{eq:perturbation-negative-composition} and \eqref{eq:estimate-A-0} to deduce that 
$$\|  \Pi_h (\mathcal{D}_{u_1,0} - \mathcal{D}_{u_2, 0}) B\|_{H^{-r - \delta} } \leq C \dd_P^\delta (P_1, P_2)   \| b \|_{H^{-r}} \,.$$
For the positive-regularity component, we use the crucial identity $(\Id-\PP)\mathcal T_{P_i^{-1}}\circ\cA_{P_i}\PP=0$,
since pushforward by a volume-preserving diffeomorphism preserves
divergence-free fields. Consequently, for any $i =1,2$ we have
$$(\Id-\PP)\mathcal T_{P_i^{-1}}\circ\cA_{P_i}b
=(\Id-\PP)\mathcal T_{P_i^{-1}}\circ\cA_{P_i}(\Id-\PP)b \,.$$
Hence, applying \eqref{eq:perturbation-positive-composition} and  \eqref{eq:estimate-A-0} we obtain
$$\|(\Id-\PP)\Pi_h(\mathcal D_{u_1,0}-\mathcal D_{u_2,0})B\|_{H^{r-\delta}}
\leq C d_P^\delta(P_1,P_2)\|(\Id-\PP)b\|_{H^r}\,.$$

For the vertical projection $\Pi_v (\mathcal{D}_{u_1, 0} - \mathcal{D}_{u_2,0}) B$ we simply use \eqref{eq:estimate-A} instead of \eqref{eq:estimate-A-0}
and we have that 
$$\| \Pi_v (\mathcal{D}_{u_1, 0} - \mathcal{D}_{u_2,0}) B\|_{H^{-r-2\delta}}
\leq C\left(
d_P^\delta (P_1, P_2)\|b_v\|_{H^{-r-\delta}}
+(d_P^\delta (P_1, P_2)+d_\tau (\tau_1, \tau_2))\|b\|_{H^{-r}}
\right)\,,$$
which concludes the proof.
\end{proof}

The final ingredient needed for proving openness of $\mathcal{D}$ is the following statement about the spectral properties of $\lambda \mapsto \mathcal{D}_{u,\lambda}$.

\begin{lemma}
\label{lemma:holomorphic 1}
Let $u_0 \in \mathcal{D}$. Then the following hold:
\begin{enumerate}
    \item the map $\lambda\mapsto\mathcal{D}_{u_0,\lambda}$
is holomorphic with values in the space of bounded linear
operators on $\mathcal{X}$\,,
    \item there exists some $\lambda_0 \in \CC$ with $\Re (\lambda_0) >0$ for which $\mathcal{D}_{u_0,\lambda_0}$ admits $1 \in \CC$ as an isolated, algebraically simple eigenvalue\,,
    \item let $\zeta(\lambda)$ denote the isolated, algebraically simple eigenvalue of $\mathcal{D}_{u_0,\lambda}$. The map $\lambda\mapsto\zeta(\lambda)$ is well defined and holomorphic in a neighbourhood of $\lambda_0$.
\end{enumerate}
In fact, the first assertion holds more generally for every
$u_0\in C^\infty_\sigma(\TT^3;\RR^3)$ with $(u_0)_3>0$.
\end{lemma}
\begin{proof}
We start by proving the first point. We notice that
$$
\mathcal{D}_{u_0,\lambda} B=\sum_{n \geq 0}\frac{(-1)^n\lambda^n(\tau_{u_0}\circ P_{u_0}^{-1})^n}{n!}  \begin{pmatrix}
DP_{u_0} \circ P_{u_0}^{-1} & 0 \\
-D \tau_{u_0} \circ P_{u_0}^{-1} & 1
\end{pmatrix} B\circ P_{u_0}^{-1}\,,
$$
and using the fact that multiplication by a smooth function is a bounded operator on our anisotropic space, with operator norm bounded by the function's $C^1$ norm (see Lemma \ref{lemma:multipliers anisotropic}), we deduce the desired holomorphicity immediately from the power series representation.
Hence, since the eigenvalue $1$ is algebraically simple, the existence of the holomorphic function $\zeta(\lambda)$ follows immediately by standard perturbation theory \cite{K76}*{Chapter VII, Theorem 1.8}. The second point follows immediately by Lemma \ref{lemma:D implies lasota yorke}. It remains to show the third point. Let $u_0 \in \mathcal{D}$, and pick  $m \in \NN$ so that $C_0 \e^{-\Re(\lambda_0)\inf_{a \in \TT^2} \tau_{u_0}^m(a)}\|DP_{u_0}^m\|^{1-r}_{L^\infty}<1$ by Definition \ref{def:D}. Then, there exists a neighbourhood $\mathcal{V}$ of $\lambda_0$ such that for all $\lambda \in \mathcal{V}$ we have  $\Re(\lambda) \geq \frac{\Re(\lambda_0)}{2} >0$ and  $C_0\e^{-\Re(\lambda)\inf_{a \in \TT^2} \tau_{u_0}^m(a)}\|DP_{u_0}^m\|^{1-r}_{L^\infty}<1$, where $C_0$ is the constant in Definition \ref{def:D} related to $\lambda_0$. Therefore, by Lemma \ref{lemma:D implies lasota yorke} and Proposition \ref{prop:hennion}, $\Id-\mathcal{D}^m_{u_0,\lambda}$ is  Fredholm for all $\lambda \in \mathcal{V}$. Furthermore, we observe that by Corollary \ref{cor:factorisation1} we can take a connected open set $\mathcal{V}$ containing $\{\lambda \in \CC: \Re (\lambda )\geq \Re (\lambda_0)\,, \Im (\lambda) = \Im (\lambda_0) \}$ and satisfying $C_0\e^{-\Re(\lambda)\inf_{a \in \TT^2} \tau_{u_0}^m(a)}\|DP_{u_0}^m\|^{1-r}_{L^\infty}<1 $ for all $\lambda \in \mathcal{V}$. It follows that there exists $\lambda_\star \in \mathcal{V}$ with $\Re (\lambda_\star)$ sufficiently large such that $\|\mathcal{D}^m_{u_0,\lambda_\star}\|_{\mathcal X \to \mathcal X} <1$.
Therefore, $\Id-\mathcal D_{u_0,\lambda}^m$ is a holomorphic family of Fredholm operators on a connected open set containing $\lambda_0$, and also containing a point $\lambda_\star$ where it is bijective. Its Fredholm index is zero throughout $\mathcal V$, by local constancy of the index. As such, by the analytic Fredholm theorem, it follows that the set of values $\lambda\in\mathcal V$ where $\Id-\mathcal D_{u_0,\lambda}^m$ is non-invertible is discrete, see \cite{K76}*{Chapter VII, proof of Theorem 1.9}. In particular, in the neighbourhood of $\lambda_0$ where
$\zeta$ is defined, $\zeta(\lambda)=1$ implies that
$\Id-\mathcal D_{u_0,\lambda}^m$ is non-invertible.
Consequently, $\lambda_0$ is an isolated zero of
$\zeta(\lambda)-1$, proving the third point.
\end{proof}
We now have all the ingredients needed to complete the proof of openness of $\mathcal{D}$.
\begin{proposition}
There exists $k \in \NN$ so that $\mathcal{D}$ is open in the $C^k$ topology on $C^\infty_\sigma(\TT^3)$.
\end{proposition}
\begin{proof}
We begin by noting that the condition of having $u_3>0$ is evidently open in the $C^k$ topology. Pick now $ u_0 \in \mathcal{D}$, and $\lambda_0 \in \CC$ so that $\mathcal{D}_{ u_0,  \lambda_0}$ has an algebraically simple eigenvalue $1$. By Lemma \ref{lemma:holomorphic 1}, we may find a holomorphic function $\zeta(\lambda)$, so that $\zeta(\lambda_0)=1$ and such that there exists a radius $R>0$ so that $\lambda_0$ is the only zero of $\zeta(\lambda)-1$ in $\overline{D(\lambda_0;2R)}$. Upon further shrinking $R$, we can ensure by Lemma \ref{lemma:D implies lasota yorke} that uniformly for $\lambda \in \overline{D(\lambda_0;2R)}$, and $u$ in a $C^k$ ball around $ u_0 $, there exist $\eta,C,M>0$ so that for all $m \in \NN$,
$$
\|\mathcal{D}^m_{u,\lambda} B\|_{\mathcal{X}} \leq C(1-\eta)^m\|B\|_{\mathcal{X}}+CM^m\|B\|_{\mathcal{Y}}.
$$
Combining this with Lemma \ref{lemma:perturbing velocity}, we obtain, for all $u\in C^\infty_\sigma$ sufficiently close to $u_0$ in $C^k$, the following estimates \emph{uniformly for $\lambda \in \overline{D(\lambda_0;2R)}$}
\begin{align*}
    \|\mathcal{D}^m_{u,\lambda} B\|_{\mathcal X} & \leq C(1-\eta)^m\|B\|_{\mathcal X}+CM^m\|B\|_{\mathcal Y}, 
    \\
    \|\mathcal{D}^m_{u,\lambda} B\|_{\mathcal Y} & \leq CM^m\|B\|_{\mathcal Y},
    \\
    \|\mathcal{D}_{u,\lambda} B-\mathcal{D}_{u_0,\lambda}B\|_{\mathcal Y} & \leq C \|u-u_0\|_{C^k}^{\delta/2}\|B\|_{\mathcal X}\,.
\end{align*}
Furthermore, in view of \eqref{eq:continuity in u} and Corollary \ref{cor:factorisation1}, uniformly for $u$ sufficiently close  to $u_0$ in $C^k$ and $\lambda \in \overline{D(\lambda_0;2R)}$, there holds $ C_0 \inf_{m \in \NN}\e^{-\Re(\lambda)\inf_{a \in \TT^2}\tau_{u}^m(a)}\|DP_{u}^m\|^{1-r}_{L^\infty}<1$.
In fact, by Theorem \ref{thm:keller liverani} there exists a holomorphic function $\zeta^{u}(\lambda)$, tracking an algebraically simple, isolated eigenvalue of $\mathcal{D}_{u,\lambda}$ in a neighbourhood of $\lambda_0$ and satisfying $\zeta^{u}(\lambda) \to \zeta(\lambda)$ as $u \to u_0$ in $C^k$, uniformly for $\lambda \in \overline{D(\lambda_0;2R)}$. To show that $\zeta^{u}(\lambda)$ is holomorphic, fix $\lambda_\star \in \overline{D(\lambda_0;2R)}$, and let $\Sigma \subset \rho(\mathcal{D}_{u,\lambda_\star})$ be a simple contour containing $\zeta^{u}(\lambda_\star)$ in its interior, and otherwise containing no spectral points of $\mathcal{D}_{u,\lambda_\star}$ in its interior. Then, by the continuity of $\lambda \mapsto \mathcal{D}_{u,\lambda}$, it follows that 
$$
\zeta^{u}(\lambda)=\mathrm{tr} \left  (\frac{1}{2 \pi i}\int_{\Sigma}z(z-\mathcal{D}_{u,\lambda})^{-1} \dd z \right),
$$
for all $\lambda$ close to $\lambda_\star$. Since the map $\lambda \mapsto \mathcal{D}_{u,\lambda}$ is holomorphic by Lemma \ref{lemma:holomorphic 1}, it follows that $\zeta^{u}(\lambda)$ is holomorphic.
We finally note that we can write 
$$\zeta^{u}(\lambda)-1=\zeta(\lambda)-1+(\zeta^{u}(\lambda)-\zeta(\lambda))$$
and the latter of these terms converges to zero uniformly on $\Gamma=\partial D(\lambda_0;R)$ as $u \to u_0$ in $C^k$, whereas the first term satisfies $\inf_{\lambda \in \Gamma}|\zeta(\lambda)-1|>0$.
 Hence, Rouch\'e's theorem implies that, for all $u$ sufficiently close to $u_0$ in $C^k$, $\zeta(\lambda)-1$ and $\zeta^{u}(\lambda)-1$ have the same number of zeros inside $\Gamma$. Since $\zeta(\lambda)-1$ has at least one zero, this implies that there exists $\lambda_u$ close to $\lambda_0$, so that $1 \in \sigma_{\mathcal X}(\mathcal{D}_{u,\lambda_u})$. Furthermore, by the Keller--Liverani theorem \ref{thm:keller liverani}, it follows that the eigenvalue must be algebraically simple. Hence, the proof of openness is complete.
\end{proof}

\section{The proof of fast dynamo action}
Having established non-emptiness and openness of $\mathcal{D}$, we now prove the third point of Proposition \ref{prop:abstract proposition fast dynamo}, namely that any autonomous vector field $u \in \mathcal{D}$ generates a fast dynamo. We first make the framework in Section \ref{sec:framework autonomous} precise.

\subsection{The trace operator at $s=1$}

In this section, we shall study the eigenvalue problem \eqref{eq:elliptic evolution} for the kinematic dynamo on the half cylinder  $s \in [0, \infty)$. After dividing by $u_3>0$ and using the variables $(a,s)\in \TT^2\times [0,\infty)$ we obtain the following PDE
\begin{align} \label{eq:elliptic evolution 2}
\begin{cases}
    \partial_s B +\frac{u_{1,2}}{u_3}\cdot\nabla_a B - \frac{1}{u_3}  (\nabla u - \lambda \Id) B =  \frac{\varepsilon}{u_3}\Delta_a B +\frac{\varepsilon}{u_3}\partial_s^2 B \,, \qquad (a,s) \in \TT^2 \times [0,\infty), 
    \\
    B(\cdot, 0)=B_0 
    \\
    \e^{-\beta s} B \in H^1(\TT^2 \times [0,\infty)) \,.
\end{cases}
\end{align}
 In particular, we aim to define the operator $\mathcal{K}_{\lambda,\eps}B_0 = \mathrm{tr} (B)|_{s=1} $, where $\mathrm{tr} (\cdot)|_{s=s_0}$ is the trace at $\{s=s_0\} \subset \TT^2 \times [0,\infty)$ and $B$ is the unique solution to \eqref{eq:elliptic evolution 2} with $\mathrm{tr} (B)|_{s=0}=B_0$.
However, doing so first entails specifying a suitable notion of solution to these equations.  
A standard approach, used extensively in Section \ref{sec:pseudodifferential result}, is to freeze the coefficients of the equation and solve the resulting constant-coefficient PDE. Taking the Fourier transform in the horizontal variables yields
$$
\left(-\eps\partial_s^2 + u_3\partial_s + \lambda + \eps|\xi|^2
+ i u_{1,2}\cdot\xi - \nabla u\right)\hat B(\xi,s)=0.
$$
We seek solutions of the form $\hat B(\xi,s)=\e^{qs}v$, with $v\in\CC^3$. Neglecting the lower-order matrix term $\nabla u$, we obtain
a quadratic polynomial in $q$ whose two roots correspond to
a fast wave $\e^{q_+s}=\e^{(u_3/\eps+O(1))s}$, and a slow
wave $\e^{q_-s}=\e^{O(1)s}$ as $\eps\to0$ at fixed Fourier
frequency. Thus, requiring solutions to grow at most like $\e^{\beta s}$, with $\beta\gg1$ \emph{independent} of $\eps\ll1$, should allow us to recover uniqueness. Indeed, we have the following result.

\begin{lemma}
\label{lemma:notion of solution}
We consider the PDE  \eqref{eq:elliptic evolution 2} with a vector field  $u\in C^\infty(\TT^3)$  such that $u_3 \geq c_0>0$, and $\nabla \cdot u=0$.
Then, for any compact set $K \subset \CC$, there exist $\beta>0$  independent of $\eps$ and $\eps_0>0$ so that for all $\eps \in (0, \eps_0)$, $\lambda \in K$, and  $B_0 \in H^{1/2}(\TT^2)$, there exists a unique solution $B$ to \eqref{eq:elliptic evolution 2}. Furthermore, the solution operator $\mathcal{C}_{\lambda,\eps}:B_0 \to B$ is continuous from $H^{1/2}(\TT^2)$ to $\e^{\beta s} H^1(\TT^2 \times [0,\infty))$, and extends to a continuous linear map 
$$
\mathcal{C}_{\lambda,\eps}:H^{-1/2}(\TT^2) \to \e^{\beta s}L^2(\TT^2 \times [0,\infty)),
$$
satisfying 
$$
\|\e^{-\beta s}\mathcal{C}_{\lambda,\eps}\|_{H^{-1/2}(\TT^2) \to L^2(\TT^2 \times [0,\infty))} \leq  C_K\eps^{-1/2},
$$
uniformly for all $\lambda \in K$, and all $\eps$ sufficiently small.
\end{lemma}

\begin{proof}[Proof of Lemma \ref{lemma:notion of solution}]
We first prove existence and uniqueness of  solutions to \eqref{eq:elliptic evolution 2}. Let $B=\e^{\beta s}v$. Then, inserting this into \eqref{eq:elliptic evolution 2}, there holds $\cA_{\lambda, \eps, \beta} v=0$, where 
\begin{align} \label{d:op-beta}
    \cA_{\lambda, \eps, \beta} v=\eps \Delta v-(u-2 \eps \beta \mathbf{e}_3) \cdot \nabla v+v\cdot \nabla u-\lambda v-\beta u_3 v+\eps \beta^2 v.
\end{align}
Taking the $L^2$ inner product with $v$ and integrating by parts thus yields, upon noting that $v \mapsto (u-2 \eps \beta \mathbf{e}_3)\cdot \nabla v$ is skew-symmetric for any real $\beta$,
$$
\Re \langle -\cA_{\lambda, \eps, \beta} v, v\rangle \geq \eps \|\nabla v\|_{L^2}^2+ (\beta c_0-\eps \beta^2 -\|\nabla u\|_{L^\infty}-|\lambda|)\|v\|_{L^2}^2\,,
$$
for any $v \in H^1_0(\TT^2 \times [0,\infty))$.
Hence, there exists $\beta>0$ large enough such that for all $\eps \in (0, \eps_0)$  the following coercivity holds true
$$
\Re \langle -\cA_{\lambda, \eps, \beta} v, v\rangle \geq c\|v\|_{L^2}^2+ \eps \|\nabla v\|_{L^2}^2\,.
$$
Fix $B_0\in H^{1/2}(\TT^2)$. By the trace extension theorem, there exists $\tilde B\in H^1(\TT^2\times[0,\infty))$ such that
$$
\mathrm{tr}(\tilde B)|_{s=0}=B_0, \qquad \|\tilde B\|_{H^1(\TT^2\times[0,\infty))} \leq C\|B_0\|_{H^{1/2}(\TT^2)}.
$$
We seek a solution $v_{B_0}\in H^1(\TT^2\times[0,\infty))$ satisfying $\mathrm{tr}(v_{B_0})|_{s=0}=B_0$. We write $v_{B_0}=v_{0}+\tilde B$, where $v_{0}\in H_0^1(\TT^2\times[0,\infty))$. The equation $\cA_{\lambda, \eps, \beta} v_{B_0}=0$ is then equivalent to finding $v_{0}\in H_0^1(\TT^2\times[0,\infty))$ such that
$$ \cA_{\lambda, \eps, \beta} v_{0}= - \cA_{\lambda, \eps, \beta} \tilde B.
$$
Since
$ \|\cA_{\lambda, \eps, \beta} \tilde B\|_{H^{-1}} \leq C\|\tilde B\|_{H^1}$\footnote{Here we use $H^{-1}$ as the dual of $ H^1_0$.},
the  coercivity estimate and the Lax--Milgram theorem yield a unique solution $v_0\in H_0^1(\TT^2\times[0,\infty))$.  Furthermore, testing the equation against $-v_0$ and using coercivity gives
$$
\|v_0\|_{L^2}+\sqrt{\eps}\|\nabla v_0\|_{L^2} \leq C\eps^{-1/2} \|\cA_{\lambda, \eps, \beta} \tilde B\|_{H^{-1}} \leq C\eps^{-1/2}\|B_0\|_{H^{1/2}},
$$
uniformly for $\lambda\in K$ and all sufficiently small $\eps$.  

We now prove the estimate  $ \|\e^{-\beta s}\mathcal{C}_{\lambda,\eps}\|_{H^{-1/2}(\TT^2) \to L^2(\TT^2 \times [0,\infty))} \leq  C_K\eps^{-1/2}.$
We first observe that the $L^2$-adjoint operator 
$$
\cA_{\lambda, \eps, \beta}^\star w=\eps \Delta w+(u-2\eps \beta \mathbf{e}_3)\cdot \nabla w+\nabla u^T w-\overline{\lambda} w-\beta u_3 w+\eps \beta^2 w
$$
satisfies the same coercivity estimate as $\cA^\beta_{\lambda, \eps}$; hence for any given  $F \in L^2(\TT^2 \times [0,\infty))$, there exists a unique solution $w_F\in H_0^1(\TT^2\times[0,\infty))$ to 
$$
\cA_{\lambda, \eps, \beta}^\star w_F=F,
$$
satisfying 
$$
\|w_F\|_{L^2}+\sqrt{\eps}\|\nabla w_F\|_{L^2(\TT^2 \times [0,\infty))} \leq C\|F\|_{L^2(\TT^2 \times [0,\infty))}\,.
$$
Taking the $L^2$ norm of the equation $\cA_{\lambda, \eps, \beta}^\star w_F=F$ and using the above inequality, we in fact deduce 
$$
\|w_F\|_{H^2(\TT^2 \times [0,\infty))} \leq C \eps^{-3/2}\|F\|_{L^2(\TT^2 \times [0,\infty))}\,,
$$
so that  $\partial_s w_F\in H^1(\TT^2 \times [0,\infty))$. As such, the trace theorem implies that 
\begin{align} \label{eq:trace-pzw}
    \|\mathrm{tr} (\partial_s w_F)|_{s=0}\|_{H^{1/2}(\TT^2)} \leq C\|\partial_s w_F\|_{H^1(\TT^2 \times [0,\infty))} \leq C\eps^{-3/2} \|F\|_{L^2(\TT^2\times [0,\infty))}.
\end{align}
Using the definition of $w_F$ and integrating by parts, with $\mathrm{tr}(w_F)|_{s=0}=0$ and $\mathrm{tr}(v_{B_0})|_{s=0}=B_0$, we obtain 
$$ \langle v_{B_0},F\rangle_{L^2}
=\langle v_{B_0},\cA_{\lambda, \eps, \beta}^\star w_F\rangle_{L^2} =-\eps \langle B_0, \mathrm{tr} (\partial_s w_F)|_{s=0}
 \rangle_{H^{-1/2},H^{1/2}}.$$
Consequently, the trace estimate \eqref{eq:trace-pzw} implies
$$
|\langle v_{B_0},F\rangle_{L^2 }| \leq C\eps^{-1/2} \|B_0\|_{H^{-1/2}}\|F\|_{L^2}.
$$
Taking the supremum over all $F$ with $\|F\|_{L^2(\TT^2 \times [0,\infty))}=1$, we obtain
$$ \|v_{B_0}\|_{L^2 (\TT^2 \times [0,\infty))} \leq C\eps^{-1/2}\|B_0\|_{H^{-1/2}(\TT^2)}\,.$$
Since $H^{1/2}(\TT^2)$ is dense in $H^{-1/2}(\TT^2)$, the map $B_0\mapsto v_{B_0}$ extends uniquely to a bounded linear map from $H^{-1/2}(\TT^2)$ to $L^2(\TT^2\times [0,\infty))$. 
\end{proof}
We finally prove the following quantitative interior regularity estimate.
\begin{lemma}
\label{lemma:elliptic regularity}
We consider the operator $\mathcal{C}_{\lambda,\eps}:H^{-1/2}(\TT^2) \to \e^{\beta s}L^2(\TT^2 \times [0,\infty))$ with $\beta >0$ and $\eps \in (0, \eps_0)$ defined in Lemma \ref{lemma:notion of solution}. Then, for any $B_0 \in H^{-1/2}(\TT^2)$, $\mathcal{C}_{\lambda,\eps} B_0 \in C^\infty(\TT^2 \times (0,\infty))$, and for any compact $ \tilde K \Subset \TT^2 \times (1/2,\infty)$, and any $M \in \NN$, there exists $C_{M, \tilde K} >0$ so that there holds the quantitative elliptic regularisation estimate for all $\eps \in (0,\eps_0)$
$$
\|\mathcal{C}_{\lambda,\eps} B_0\|_{H^M( \tilde K)} \leq C_{M, \tilde K} \eps^{-(2M+1)/2}\|B_0\|_{H^{-1/2}}.
$$
\end{lemma}

\begin{proof}
Since $\mathcal{C}_{\lambda,\eps} B_0$ defines a distributional solution of \eqref{eq:elliptic evolution 2}, and $u \in C^\infty(\TT^3)$, interior smoothness follows immediately by standard elliptic regularity theory.  We now choose a maximal $\eps/2$-separated family of points $\{X_j\}_{j=1}^{N_\eps}\subset\tilde K$. By maximality,
$ \tilde K\subset\bigcup_{j=1}^{N_\eps}B_{\eps/2}(X_j)$.
For each $j$, introduce the rescaled function
$ \tilde v_j(Y):=v(X_j+\eps Y),$ with $ Y\in B_1(0). $
In these variables, the equation $\cA_{\lambda, \eps, \beta} v=0$ becomes
$$ \Delta_Y\tilde v_j -\bigl(u(X_j+\eps Y)-2\eps\beta\mathbf e_3\bigr)
 \cdot\nabla_Y\tilde v_j
+\eps\,\tilde v_j\cdot\nabla u(X_j+\eps Y)-\eps\bigl(\lambda+\beta u_3(X_j+\eps Y)-\eps\beta^2\bigr)
 \tilde v_j=0 \,.$$
 The rescaled operators are uniformly elliptic, and their coefficients are bounded uniformly in $\eps$ and $\lambda\in K$, together with all derivatives. Standard interior elliptic estimates therefore give
$$
\|\tilde v_j\|_{H^M(B_{1/2}(0))} \leq C_M\|\tilde v_j\|_{L^2(B_1(0))},
$$
where $C_M$ is independent of $\eps$ and $j$; see, for example, \cite{GilbargTrudinger2001}*{Chapter~9, \S~9.5, Theorem~9.11}. Hence, returning to the original variables, we obtain, for every $0\leq m\leq M$,
$$
\|D^m v\|_{L^2(B_{\eps/2}(X_j))}
\leq C_M\eps^{-m}
\|v\|_{L^2(B_\eps(X_j))}.
$$
Summing over the covering and using the bounded overlap of the enlarged balls $\{B_\eps(X_j)\}_{j=1}^{N_\eps}$ yields
\begin{align*}
\|v\|_{H^M(\tilde K)}^2\leq C_M\eps^{-2M}
 \sum_{j=1}^{N_\eps}\|v\|_{L^2(B_\eps(X_j))}^2
\leq C_M\eps^{-2M}\|v\|_{L^2(\TT^2\times [0,\infty))}^2
\leq C_M\eps^{-2M-1}
 \|B_0\|_{H^{-1/2}(\TT^2)}^2\,,
\end{align*}
concluding the proof by observing that multiplication by $\e^{\beta s}$ is bounded on $H^M(\tilde K)$.
\end{proof}
With this result established, we now define the operator $\mathcal{K}_{\lambda,\eps}$ via 
$$
\mathcal{K}_{\lambda,\eps}B_0=\mathrm{tr}(\mathcal{C}_{\lambda,\eps} B_0)|_{s=1} \,, \qquad \mathcal{K}^m_{\lambda,\eps}B_0=\mathrm{tr}(\mathcal{C}_{\lambda,\eps} B_0)|_{s=m}\,.
$$
Indeed, note that this identity is justified in view of Lemma \ref{lemma:notion of solution}, and the periodicity of the coefficients of \eqref{eq:elliptic evolution 2}.
The following result relates the condition that $1$ be an eigenvalue of $\mathcal{K}_{\lambda,\eps}$ to the condition that $\lambda$ be an eigenvalue of the kinematic dynamo operator.
\begin{corollary}
\label{cor:section eigenvalue implies dynamo}
Let $K \subset \CC$ be compact, and let $\lambda \in K$. Then, for any $m \in \NN$, there exists $C_m>0$ independent of $\eps$ and of $\lambda\in K$, so that 
\begin{equation}
\label{eq:regularisation}
\|\mathcal{K}^m_{\lambda,\eps} B_0\|_{H^1(\TT^2)} \leq C_m\eps^{-5m/2}\|B_0\|_{H^{-1/2}(\TT^2)}.
\end{equation}
Finally, if there exist $\lambda_0\in\CC$ with $\Re(\lambda_0)>0$ and a non-trivial $B_0\in H^{-1/2}(\TT^2)$ such that $\mathcal{K}_{\lambda_0,\eps}B_0=B_0$, then $\lambda_0$ is an eigenvalue of the elliptic operator 
$$
B \mapsto \eps \Delta B+B \cdot \nabla u-u \cdot \nabla B
$$
on $L^2(\TT^3)$, with associated divergence-free eigenfunction $\mathcal{C}_{\lambda_0,\eps}B_0 \in C^\infty (\TT^3)$.
\end{corollary}
\begin{proof}
In view of Lemmas \ref{lemma:notion of solution} and \ref{lemma:elliptic regularity}, together with the trace theorem, estimate \eqref{eq:regularisation} follows immediately.
Suppose now that
$\mathcal K_{\lambda,\eps}B_0=B_0$,
and set $B=\mathcal C_{\lambda,\eps}B_0$. By \eqref{eq:regularisation} and the eigenfunction condition we have
$B_0\in H^1(\TT^2)\subset H^{1/2}(\TT^2)$,
and hence $B\in\e^{\beta s}H^1(\TT^2\times[0,\infty))$ by Lemma \ref{lemma:notion of solution}. Since $u$ is $1$-periodic in $s$, the shifted function $ \tilde B(a,s):=B(a,s+1) $ solves the same equation and belongs to the same weighted space. Moreover,
$$
\mathrm{tr}(\tilde B)|_{s=0} =\mathrm{tr}(B)|_{s=1} =\mathcal K_{\lambda,\eps}B_0 =B_0 \,.
$$
Then, by the uniqueness result of Lemma \ref{lemma:notion of solution} we have  that $\tilde B=B$. Thus $B$ is $1$-periodic in $s$ and defines an eigenfunction on $\TT^3$ with eigenvalue $\lambda_0$ that is smooth by elliptic regularity theory. For the divergence-free condition we take the divergence of the eigenvalue equation obtaining
$$
\eps\Delta (\nabla\cdot B)-u\cdot\nabla (\nabla\cdot B)=\lambda \nabla\cdot B\,.
$$
Testing with $\nabla\cdot B$ and then taking real parts yields $\nabla\cdot B =0$ whenever $\Re(\lambda_0)>0$.
\end{proof}

Finally, we need the following result.
\begin{lemma}
\label{lemma:holomorphic 2}
Let $K\subset \CC$ be a compact set. Then, there exists $\eps_0$ such that for all $\eps\in (0, \eps_0)$ the map 
$\lambda \mapsto \mathcal{K}_{\lambda,\eps}$ is holomorphic in an open neighbourhood of $K$ as a family of operators $\cK_{\lambda,\eps}:\mathcal{E}_u\mathcal X\to\mathcal{E}_u\mathcal X$.
\end{lemma}
\begin{proof}
Choose a compact set $K_1\subset\CC$ such that $ K\subset\operatorname{int}(K_1),$
and take $\beta>0$ and $\eps>0$ as in Lemma \ref{lemma:notion of solution}, applied to $K_1$. We recall the operator 
$$
\cA_{\lambda, \eps, \beta} : H^1_{0}(\TT^2 \times [0,\infty))  \cap H^2(\TT^2 \times [0,\infty)) \to L^2(\TT^2 \times [0,\infty))
$$ 
defined in \eqref{d:op-beta}, whose invertibility was established above. We denote by $R_{\lambda, \eps}^\beta$ its bounded inverse. By definition of the operator $\mathcal{C}_{\lambda, \eps}$ in Lemma \ref{lemma:notion of solution}, we have $\cA_{\lambda, \eps, \beta} (\e^{-\beta s}\mathcal{C}_{\lambda,\eps} B_0)=0$ and, by definition of the operator, we have $\cA_{\lambda, \eps, \beta} -\cA_{\mu, \eps, \beta}=(\mu-\lambda) \Id$. Therefore we deduce  that for $B_0 \in H^{1/2}(\TT^2)$ 
$$ \cA_{\lambda, \eps, \beta} (\e^{-\beta s}\mathcal{C}_{\lambda,\eps} B_0-\e^{-\beta s}\mathcal{C}_{\mu,\eps} B_0) = (\lambda-\mu) (\e^{-\beta s}\mathcal{C}_{\mu,\eps}B_0) \,.$$
Applying the inverse $R_{\lambda,\eps}^\beta$, noting that the trace at $s=0$ of $\e^{-\beta s}\mathcal{C}_{\lambda,\eps} B_0-\e^{-\beta s}\mathcal{C}_{\mu,\eps} B_0$ is zero, we obtain
\begin{equation}
\label{eq:continuity fractional map}
(\e^{-\beta s}\mathcal{C}_{\lambda,\eps}-\e^{-\beta s}\mathcal{C}_{\mu,\eps})B_0=(\lambda-\mu)R_{\lambda,\eps}^\beta (\e^{-\beta s}\mathcal{C}_{\mu,\eps} B_0)\,.
\end{equation}
Furthermore, since $\e^{-\beta s} \mathcal{C}_{\mu, \eps}$ maps $H^{-1/2} (\TT^2) $ to $L^2 (\TT^2 \times [0,\infty))$ boundedly, this identity extends to $B_0 \in H^{-1/2} (\TT^2)$ by continuity. We now take the trace at $s=1$ in \eqref{eq:continuity fractional map}. We recall that $\cK_{\lambda, \eps} = \mathrm{tr}(\mathcal{C}_{\lambda, \eps})|_{s=1}$. Then, we have
$$
\frac{\mathcal{K}_{\mu,\eps}-\mathcal{K}_{\lambda,\eps}}{\mu-\lambda}-\e^\beta\mathrm{tr} ( R_{\lambda,\eps}^\beta(\e^{-\beta s}\mathcal{C}_{\lambda,\eps}))|_{s=1}=\e^{\beta}\mathrm{tr} ( R_{\lambda,\eps}^\beta(\e^{-\beta s}\mathcal{C}_{\mu,\eps} -\e^{-\beta s}\mathcal{C}_{\lambda,\eps}))|_{s=1}.
$$
Using that $\mathrm{tr}(\cdot)|_{s=1} : H^1(\TT^2 \times [0,\infty)) \to H^{1/2}(\TT^2)$ is bounded, and $H^{1/2}(\TT^2) \hookrightarrow\mathcal E_u\mathcal X \hookrightarrow H^{-1/2}(\TT^2),$ we have 
\begin{align*}
    &\left \| \frac{\mathcal{K}_{\mu,\eps}-\mathcal{K}_{\lambda,\eps}}{\mu-\lambda}-\e^\beta\mathrm{tr}( R_{\lambda,\eps}^\beta(\e^{-\beta s}\mathcal{C}_{\lambda,\eps}))|_{s=1}\right \|_{\mathcal E_u\mathcal X \to \mathcal E_u\mathcal X} 
    \\
    &\qquad  \leq C\e^{\beta} \| \mathrm{tr}|_{s=1} \|_{H^1_0 \to H^{1/2}} \|R_{\lambda,\eps}^\beta \|_{L^2 \to H^{1}_0} \|(\e^{-\beta s}\mathcal{C}_{\mu,\eps} -\e^{-\beta s}\mathcal{C}_{\lambda,\eps})\|_{H^{-1/2} \to L^2}\,.
\end{align*}
Finally, taking norms in \eqref{eq:continuity fractional map}, we bound 
\begin{align*}
\| \e^{-\beta s}\mathcal{C}_{\mu,\eps}-\e^{-\beta s}\mathcal{C}_{\lambda,\eps} \|_{H^{-1/2} \to L^2} \leq |\lambda-\mu|\|R_{\lambda,\eps}^\beta \|_{L^2 \to L^2}\sup_{\mu \in K_1}\|\e^{-\beta s}\mathcal{C}_{\mu,\eps}\|_{H^{-1/2} \to L^2}\,,
\end{align*}
which converges to zero as $\mu \to \lambda$ by Lemma \ref{lemma:notion of solution}.  Thus $\lambda\mapsto\mathcal K_{\lambda,\eps}$ is complex differentiable in operator norm on $\operatorname{int}(K_1)$, and is therefore holomorphic.
\end{proof}

\subsection{The pseudodifferential estimates}

With $\mathcal{K}_{\lambda,\eps}$ rigorously defined as in the previous subsection, we now recall the factorisation $\mathcal{K}_{\lambda,\eps}^m =\mathcal{E}_u\mathcal{D}_{u,\lambda}^m\mathcal{U}_{\lambda,\eps}^{[m]}\mathcal{E}_u^{-1}:  \mathcal{E}_u \mathcal X \to \mathcal{E}_u \mathcal X$ given in Corollary \ref{cor:factorisation1}  and  
we aim to characterise its spectrum. This requires a careful analysis of the operator $\mathcal{U}_{\lambda,\eps}^{[m]} = (\mathcal{D}_{u, \lambda}^m)^{-1} \cE_u^{-1} \mathcal{K}_{\lambda,\eps}^m \cE_u$. This operator is the trace at $s=m$ of the unique solution given in Lemma \ref{lemma:notion of solution} up to the composition with the operators $\mathcal{D}_{u, \lambda}$ and $\cE_u$. We start by studying the PDE $\cL_{\lambda, \eps}B=0$ with an artificial final boundary condition $\mathrm{tr}( B)|_{s=T} = B_T$, where we recall that  $\cL_{\lambda, \eps}$ arises naturally  in \eqref{eq:abstract elliptic equation} for the definition of the operator $\cU_{\lambda, \eps} = \cU_{\lambda, \eps}^{[1]}$. 
More precisely, we study the  PDE 
\begin{align} \label{eq:pseudo-0-T}
    \cL_{\lambda, \eps} B=0\,, \qquad \mathrm{tr} (B)|_{s=0} = B_0 \in \mathcal{X} \,, \qquad \mathrm{tr} (B)|_{s=T} = B_T \in L^2
\end{align}
 where 
\begin{equation}
\label{eq:abstract operator statement}
\cL_{\lambda,\eps}=\partial_s-\eps (\alpha^{0,0}\partial_s^2+2\alpha^{0,j}\partial_s \partial_j +\alpha^{i,j}\partial_i \partial_j+H^0\partial_s+H^j \partial_j +H),
\end{equation}
Here, $i,j=1,2$, the derivatives $\partial_1,\partial_2$ act on $a\in\TT^2$, and $\partial_s$ denotes differentiation in $s\in[0,\infty)$.
The key ingredient is the following abstract result from pseudodifferential calculus that will be proved in Section~\ref{sec:pseudodifferential result}. We make the following assumptions on the coefficients of $\cL_{\lambda, \eps}$ that are satisfied by the operator obtained in  \eqref{eq:abstract elliptic equation}:
\begin{enumerate}
    \item The coefficients $\alpha^{i,j}$, $i,j=0,1,2$, are
real-valued scalar functions satisfying $\alpha^{i,j}=\alpha^{j,i}$.
They depend smoothly on $(a,s)\in\TT^2\times[0,T]$ and are
independent of $\eps$ and $\lambda$.
    \item There exists a constant $\Lambda_0>0$ independent of $a,s , \eps, \lambda$ so that the Hermitian matrix $A=(\alpha^{i,j})_{i,j}$ is bounded below by $\Lambda_0$.
    \item The coefficients $H^j$, $j=0,1,2$, and $H$ are matrix-valued and depend smoothly on $(a,s, \lambda) \in \TT^2 \times [0,T] \times \CC$.
\end{enumerate}

\begin{proposition}
\label{prop:pseudodifferential abstract}
We consider the differential operator $\cL_{\lambda,\eps}$ with coefficients satisfying the above assumptions. 
Then, for any compact $K\subset\CC$ and any fixed $0<m<T$,
there exist $\eps_0>0$ and $C_m>0$ such that, for every
$\lambda\in K$ and $0<\eps<\eps_0$, any solution $B$ of
\eqref{eq:pseudo-0-T} satisfies the following estimates:
\begin{align}
&\|\mathrm{tr} (B)|_{s=m}\|_{\mathcal X} \leq (1+C_m\eps^{\frac{1-2r-\delta}{2}})\|B_0\|_{\mathcal X}+C_m\|B_0\|_{\mathcal Y}+C\e^{-1/C\eps^{-1}}\|B_T\|_{L^2},\\
& \|\mathrm{tr} (B)|_{s=m} \|_{\mathcal Y} \leq C_m\|B_0\|_{\mathcal Y}+C_m\e^{-1/C\eps^{-1}}\|B_T\|_{L^2},\\
&\|\mathrm{tr} (B)|_{s=m} - B_0\|_{\mathcal Y} \leq C_m(\eps^{\frac{1-2r}{2}}+\eps^{\frac{\delta}{2}})\|B_0\|_{\mathcal X}+C_m\e^{-1/C\eps^{-1}}\|B_T\|_{L^2}.
\end{align}
\end{proposition}

As an immediate corollary, we have the following. 

\begin{corollary}
\label{cor:uniform lasota yorke}
Fix $u \in \mathcal{D} $, $\eta>0$, and a compact set $K \subset \CC$ so that uniformly for $\lambda \in K$, $\mathcal{D}_{u,\lambda}$  has a Lasota--Yorke constant  $1-{\eta} \in (0,1)$ by Lemma \ref{lemma:D implies lasota yorke}. Then, there exist $C,M >1 $ and $\eps_0>0$  so that uniformly for $\lambda \in K$, and all $\eps\in [0, \eps_0)$, the operator $\mathcal{K}_{\lambda,\eps}$ satisfies the Lasota--Yorke estimates for all $m \in \NN$
\begin{align}
&\|\mathcal{K}^{m}_{\lambda,\eps} B_0\|_{\mathcal{E}_u \mathcal X} \leq C\left (1- \frac{\eta}{2} \right )^{m} \|B_0\|_{\mathcal{E}_u \mathcal X}+C M^m\|B_0\|_{\mathcal{E}_u \mathcal Y} \label{eq:first-estimate-cor}
\\
&\|\mathcal{K}^{m}_{\lambda,\eps} B_0\|_{\mathcal{E}_u \mathcal Y} \leq  M^m \|B_0\|_{\mathcal{E}_u \mathcal Y}\\
& \|\mathcal{K}_{\lambda,\eps}B_0-\mathcal{E}_u \mathcal{D}_{u,\lambda} \mathcal{E}_u^{-1} B_0\|_{\mathcal{E}_u \mathcal Y} \leq C(\eps^{\frac{1-2r}{2}}+\eps^{\frac{\delta}{2}})\|B_0\|_{\mathcal{E}_u \mathcal X}. \label{eq:third-estimate-cor}
\end{align}
\end{corollary}
\begin{proof}
We aim to apply Proposition \ref{prop:pseudodifferential abstract} by  imposing in \eqref{eq:pseudo-0-T} the final boundary condition at $s=T=m+1$ as $B_T=C_{u,T,\lambda}^{-1}(a)(\mathcal{C}_{\lambda,\eps}B_0)(\Phi^T_{\w}(\kappa_u(a)),T)$ so that the resulting solution $B$ satisfies  ${\rm tr}(B)|_{s=m} = \cU_{\lambda,\eps}^{[m]}\mathcal E_u^{-1}B_0 $, recalling the PDE \eqref{d:PDE-U-lambda-eps} used to define $\cU_{\lambda,\eps}^{[m]}$.   
In  particular, we can bound this final boundary condition in $L^2 (\TT^2)$ by 
    $$
    C_{u,T,K}\|\mathrm{tr}(\mathcal{C}_{\lambda,\eps} B_0)|_{s=T}\|_{L^2(\TT^2)}\leq C_{u,T,K}\eps^{-M}\|B_0\|_{H^{-1/2}},
    $$
    where the last inequality follows by first applying Lemma \ref{lemma:elliptic regularity}, and then the trace theorem. Thus, it follows that $\|B_T\|_{L^2(\TT^2)} \leq C_{u,T,K}\eps^{-M}\|B_0\|_{\mathcal{E}_u \mathcal Y}$, thanks to $r+2\delta <1/2$, and so all terms in $B_T$ may be absorbed into the Lasota--Yorke constants, due to the presence of the exponential term. More precisely we deduce the following 
    \begin{align*}  &\|\mathcal{U}^{[m]}_{\lambda,\eps}\mathcal{E}_u^{-1}B_0\|_{\mathcal X} \leq (1+C_m\eps^{\frac{1-2r-\delta}{2}})\|B_0\|_{\mathcal{E}_u \mathcal X}+C_m\|B_0\|_{\mathcal{E}_u \mathcal Y},\\
    & \|\mathcal{U}^{[m]}_{\lambda,\eps}\mathcal{E}_u^{-1}B_0\|_{\mathcal Y} \leq C_m\|B_0\|_{\mathcal{E}_u \mathcal Y},\\
    &\|\mathcal{U}_{\lambda,\eps}\mathcal{E}_u^{-1} B_0-\mathcal{E}_u^{-1}B_0\|_{\mathcal Y} \leq C(\eps^{\frac{1-2r}{2}}+\eps^{\frac{\delta}{2}})\|B_0\|_{\mathcal{E}_u \mathcal X}.
\end{align*}
The last estimate directly implies \eqref{eq:third-estimate-cor} thanks to the decomposition of $\cK_{\lambda, \eps}$ given by Corollary \ref{cor:factorisation1}. The second estimate with $m=1$ implies that $ \| \cK_{\lambda, \eps }\|_{\mathcal{E}_u \mathcal Y \to \mathcal{E}_u \mathcal Y} \leq M$ by the decomposition given in Corollary \ref{cor:factorisation1}. Hence, it remains to prove \eqref{eq:first-estimate-cor}.
Using the Lasota-Yorke inequality for $\mathcal{D}_{u, \lambda}^{m}$ in Lemma \ref{lemma:D implies lasota yorke}, $\| \cdot \|_{\cE_u \mathcal{X}} = \| \mathcal{E}_u^{-1} \cdot \|_{\mathcal X }$ and the decomposition of $\mathcal{K}^m_{\lambda,\eps}=\mathcal{E}_u \circ  ( \mathcal{D}^m_{u,\lambda}\mathcal{U}^{[m]}_{\lambda,\eps}  ) \circ\mathcal{E}_u^{-1}$ by Corollary \ref{cor:factorisation1}, we deduce that 
$$ \|\mathcal K_{\lambda,\eps}^{m}B_0\|_{\mathcal{E}_u \mathcal{X}}
\leq
C (1- \eta)^m(1+C_m\eps^{\frac{1-2r-\delta}{2}})\|B_0\|_{\mathcal{E}_u \mathcal{X}}
+C_m\|B_0\|_{\mathcal{E}_u \mathcal{Y}}.$$
We choose an integer $L\geq1$ so that $2C (1- \eta)^L \leq (1- \frac{\eta}{2})^L$ and we choose $\eps_0\in (0,1)$ small enough so that $C_L\eps_0^{\frac{1-2r-\delta}{2}} \leq 1$. Therefore, it follows that 
$$
\|\mathcal K_{\lambda,\eps}^{L}B_0\|_{\mathcal{E}_u \mathcal{X}}
\leq
\left (1- \frac{\eta}{2} \right )^L \|B_0\|_{\mathcal{E}_u \mathcal{X}}+\bar C_L\|B_0\|_{\mathcal{E}_u \mathcal{Y}}\,,
$$
uniformly for $\lambda\in K$ and $0\leq\eps<\eps_0$. Iterating this estimate with $ \|\mathcal K_{\lambda,\eps}\|_{\mathcal{E}_u \mathcal{Y}\to \mathcal{E}_u \mathcal{Y}} \leq M$ gives
\begin{align*}
\|\mathcal K^{jL}B_0\|_{\mathcal{E}_u \mathcal{X}}
&\leq
\left (1- \frac{\eta}{2} \right )^{jL}\|B_0\|_{\mathcal{E}_u \mathcal{X}}
+\bar{C}_L\sum_{k=0}^{j-1}
\left (1- \frac{\eta}{2} \right )^{L(j-1-k)}M^{kL}\|B_0\|_{\mathcal{E}_u \mathcal{Y}}
\\
&\leq
\left (1- \frac{\eta}{2} \right )^{jL}\|B_0\|_{\mathcal{E}_u \mathcal{X}}
+\frac{\bar C_L}{M^L -1}M^{jL}\|B_0\|_{\mathcal{E}_u \mathcal{Y}}.
\end{align*}
For $m=jL+\ell$, with $0\leq\ell<L$, we use
$\max_{0 \leq \ell < L}\|\mathcal K^\ell\|_{\mathcal{E}_u \mathcal{X} \to \mathcal{E}_u \mathcal{X}}\leq M$, after increasing $M$ if necessary. Absorbing the finitely many factors corresponding to $0\leq\ell<L$ into a constant $C$, we conclude the proof.
\end{proof}

\subsection{Proof of fast dynamo action}
Finally, we may now prove the following.
\begin{proposition}
Let $u \in \mathcal{D}$. Then, $u $ generates a fast dynamo.
\end{proposition}
\begin{proof}
    Since $u\in\mathcal D$, Definition \ref{def:D} gives
$\lambda_0\in\CC$, with $\Re\lambda_0>0$, such that $1$ is an algebraically simple eigenvalue of $\mathcal{D}_{u,\lambda_0}$. By Lemma \ref{lemma:holomorphic 1}, there exists, in a neighbourhood
of $\lambda_0$, a nonconstant holomorphic branch of algebraically
simple eigenvalues $\zeta(\lambda)$ of $\mathcal{E}_u \mathcal{D}_{u,\lambda} \mathcal{E}_u^{-1}$ satisfying
$ \zeta(\lambda_0)=1.$ Consequently, the zeros of $\zeta-1$ are isolated.
Let $\mathcal{V}$ be a neighbourhood of $\lambda_0$ on which $\zeta$ is defined and the uniform ideal Lasota--Yorke estimates hold with rate $1-\eta$.
We choose $\Gamma \subset \{ \lambda \in \CC: \Re(\lambda)>\frac{\Re(\lambda_0)}{2}\}$ so that $\zeta(\lambda)-1$ does not vanish on $\Gamma$, and $\lambda_0 \in \mathrm{int}(\Gamma)$, where the latter denotes the bounded open set whose boundary is $\Gamma$. We fix a compact set $K \subset \mathcal{V}$ containing $\mathrm{int}(\Gamma)$, and apply Theorem \ref{thm:keller liverani} with strong and weak
spaces $
\mathcal E_u\mathcal X$ and $ \mathcal E_u\mathcal Y, $
respectively. Their compact embedding follows from the compact embedding $\mathcal X\hookrightarrow\mathcal Y$. The first two
estimates of Corollary \ref{cor:uniform lasota yorke} give the required uniform bounds on all iterates uniformly in $\eps \in [0, \eps_0)$, while its final estimate gives the weak-strong convergence 
$$\| \cK_{\lambda , \eps} - \mathcal{E}_u \mathcal{D}_{u, \lambda} \mathcal{E}_u^{-1} \|_{\mathcal{E}_u \mathcal X\to \mathcal{E}_u \mathcal Y} \to 0 \,, \qquad \text{as }  \eps \to 0\,,
$$
uniformly for $\lambda\in K$.
It follows that, for all sufficiently small $\eps$,
$\mathcal K_{\lambda,\eps}$ admits an algebraically simple eigenvalue $\zeta_\eps(\lambda)$ satisfying
$$
\sup_{\lambda\in K}
|\zeta_\eps(\lambda)-\zeta(\lambda)|
\to 0
\qquad\text{as }\eps\to0 \,.
$$
By Lemma \ref{lemma:holomorphic 2} and analytic perturbation theory (see \cite{K76}*{Chapter VII, Theorem 1.8}), $\zeta_\eps$ is holomorphic on a neighbourhood of
$K$. Since $\zeta-1$ does not vanish on $\Gamma$, for all sufficiently small $\eps$, we have
$$
|\zeta_\eps(\lambda)-\zeta(\lambda)| < |\zeta(\lambda)-1|, \qquad \forall \lambda\in\Gamma \,.
$$
Rouch\'e's theorem therefore implies that $\zeta_\eps-1$ and $\zeta-1$ have the same number of zeros in $\mathrm{int} (\Gamma)$. Thus, $1$ is an eigenvalue of
$\mathcal K_{\lambda_\eps,\eps}$ and
$ \Re\lambda_\eps>\frac{\Re\lambda_0}{2}>0. $ Finally,
Corollary  \ref{cor:section eigenvalue implies dynamo} implies the existence of a non-zero smooth, divergence-free eigenfunction of the dynamo operator with eigenvalue $\lambda_\eps$. Since the lower bound $\Re\lambda_\eps>\Re\lambda_0/2$ is uniform for all sufficiently small $\eps$, the velocity field $u$ generates a fast dynamo.
\end{proof}

Thus, the proof is complete, conditional on Proposition \ref{prop:pseudodifferential abstract}. The following section is therefore devoted to the proof of that result.

\section{The pseudodifferential proposition}
\label{sec:pseudodifferential result}
This section is devoted to the proof of Proposition \ref{prop:pseudodifferential abstract} using the notation and classical results in pseudodifferential calculus  from  Appendix \ref{sec:pseudodifferential intro}.
In this section $s \in [0,\infty)$ is the evolution coordinate of the abstract boundary problem, rather than the physical time of the velocity field.  Recall that the operators under consideration, defined in \eqref{eq:abstract operator statement}, have the form
\begin{equation}
\label{eq:abstract operator}
\mathcal L_{\lambda,\eps}=\partial_s-\eps (\alpha^{0,0}\partial_s^2+2\alpha^{0,j}\partial_s \partial_j +\alpha^{i,j}\partial_i \partial_j+H^0\partial_s+H^j \partial_j +H),
\end{equation}
where $\partial_j , j =1,2$ denotes differentiation in $a \in \TT^2$, and the coefficients satisfy the assumptions in Proposition \ref{prop:pseudodifferential abstract}. We fix $\lambda $ to lie in a compact subset  $K \subset \CC$ and we study the following PDE 
$$ \cL_{\lambda, \eps} B=0\,, \qquad \mathrm{tr} (B)|_{s=0}  = B_0 \in \mathcal{X} \,, \qquad \mathrm{tr} ( B)|_{s=T} = B_T \in L^2 \,.$$

\subsection{Factorising the operator}
The first key result we shall establish about the operator \eqref{eq:abstract operator} roughly says that it may be approximately factored as
$$
\mathcal L_{\lambda,\eps} \approx -\eps \alpha^{0,0}(\partial_s-F_J(s))(\partial_s-S_J(s)),
$$
where $F_J,S_J$ are pseudodifferential operators that will be referred to, respectively, as the \emph{fast and slow} generators, and the factorisation can be done so that the error is of arbitrarily high negative order depending on $J \in \NN$.
For the remainder of the section, we  define  the following  functions
$$
\rho_\eps(\xi)=(\langle \xi \rangle^2+\eps^{-2})^{1/2}, \quad \mu_\eps(\xi)=\frac{\eps \langle \xi \rangle^2}{1+\eps \langle \xi \rangle}, \quad \nu_\eps(\xi)= \xiangle^{-1} \mu_\eps (\xi) \,.
$$
These will appear frequently throughout the proof. We write $f(\xi)\approx g(\xi)$ if there are constants $C,c>0$ independent of $\eps$ such that $g(\xi )c \leq f(\xi)\leq C g(\xi)$. We can define symbol classes associated to these functions using the notions of admissible order and scale functions introduced in Appendix \ref{sec:pseudodifferential intro} as follows. The proof of the following result  is immediate from their definitions.
\begin{lemma} \label{lemma:computations}
The functions $\rho_\eps,\mu_\eps, \nu_\eps$ are  admissible order functions with respect to the admissible scale function $\ell_\eps(\xi)=\langle \xi \rangle$ and the function  $\rho_\eps$ is an admissible order function with respect to the admissible scale $\ell_\eps(\xi)=\rho_\eps(\xi)$. Furthermore, for any $\eta>0$, we have
\begin{align*}
     \nu_\eps(\xi) \leq \eta \mu_\eps(\xi)+\eta^{-1}\eps\,,
     \\
     \rho_\eps  \approx \eps^{-1}+ \xiangle\,.
\end{align*}
\end{lemma}
We now have the following result.
\begin{lemma}
\label{lemma:factorisation}
For any $J \in \NN^+$, there exist pseudodifferential operators $F_J(s), S_J(s), R_J(s)$ so that
$$
\mathcal L_{\lambda,\eps}=-\eps \alpha^{0,0}(\partial_s -F_J(s))(\partial_s -S_J(s))+\eps \alpha^{0,0}R_J(s)\,.
$$ 
These operators satisfy the following properties:
\begin{enumerate}
\item For all $J \in \NN^+$, we have $F_J \in \Psi_\eps(\rho_\eps) \cap \Psi_\eps(\rho_\eps;\rho_\eps)$, $S_J \in \Psi_\eps(\mu_\eps) \cap \Psi_\eps(\rho_\eps;\rho_\eps)$, $R_J \in \Psi_\eps(\langle\xi \rangle^{1-J}) \cap \Psi_\eps(\rho_\eps^{1-J};\rho_\eps)$.
\item There exist \emph{scalar} operators $F_0=\Op_a(q_f),S_0=\Op_a(q_s)$, so that for any $J \in \NN^+$, we have
$$
F_J-F_0 \in \Psi_\eps(1) \cap \Psi_\eps(1;\rho_\eps), \quad S_J-S_0 \in \Psi_\eps(\nu_\eps) \cap \Psi_\eps(1;\rho_\eps).
$$
\end{enumerate}
\end{lemma}
 We first split  
$$ \cL_{\lambda, \eps} = \cL_{\lambda, \eps, 1} + \cL_{\lambda, \eps ,2 }$$
where
$\cL_{\lambda, \eps, 1} = \partial_s-\eps (\alpha^{0,0}\partial_s^2+2\alpha^{0,j}\partial_s \partial_j +\alpha^{i,j}\partial_i \partial_j)$ is the scalar part of the operator and $\cL_{\lambda, \eps, 2} = -\eps (H^0\partial_s+H^j \partial_j +H)$ is the matrix part of the operator. As is typical in pseudodifferential calculus, we begin by freezing the coefficients in \eqref{eq:abstract operator}. Replacing $\partial_s$ by $q$ and $\partial_a$ by $i\xi$ in $\cL_{\lambda,\eps,1}$ gives the polynomial equation 
\begin{align} \label{eq:polynomial}
    q-\eps (\alpha^{0,0}(a,s)q^2+2i\alpha^{0,j}(a,s) \xi_j q-\alpha^{i,j}(a,s) \xi_i \xi_j)=0 
\end{align}
 whose roots will be particularly important. We denote them by $q_f(a,s,\xi)$ and $q_s(a,s,\xi)$; they are given, respectively, by  
\begin{equation}
\label{eq:fast root}
q_f(a,s,\xi)=\frac{1-2i\alpha^{0,j} \xi_j \eps+\sqrt{(1-2i \eps \alpha^{0,j}\xi_j)^2+4\eps^2 \alpha^{0,0}\alpha^{i,j}\xi_i \xi_j }}{2\eps \alpha^{0,0}}
\end{equation}
\begin{equation}
\label{eq:slow root}
q_s(a,s,\xi)=\frac{1-2i\alpha^{0,j} \xi_j \eps-\sqrt{(1-2i \eps \alpha^{0,j}\xi_j)^2+4\eps^2 \alpha^{0,0}\alpha^{i,j}\xi_i \xi_j }}{2\eps \alpha^{0,0}}.
\end{equation} 
Due to their importance, we shall first establish a number of useful properties of  $q_f, q_s$ before moving on to the proof of Lemma \ref{lemma:factorisation}. We recall that $S_\eps (\rho ; \ell)$ denotes the class of  pseudodifferential symbols with admissible order function $\rho$ and admissible scale function $\ell$ and $S_\eps (\rho) = S_\eps (\rho; \xiangle)$ .
\begin{lemma}\label{lemma:properties of the roots}
The following hold true
\begin{enumerate}
    \item $q_f$ and $q_s$ belong to the symbol classes $$
q_f\in S_\eps(\rho_\eps;\rho_\eps) \subset S_\eps (\rho_\eps),
\qquad
q_s\in S_\eps(\mu_\eps) \cap S_\eps(\rho_\eps; \rho_\eps) \subset S_\eps (\rho_\eps)\,.
$$
\item The two roots are uniformly separated and
\[
\frac{1}{q_f-q_s}
\in
S_\eps
\left(
\rho_\eps^{-1};
\rho_\eps
\right)
\subset
S_\eps
\left(
\rho_\eps^{-1}
\right).
\]
\item There exists a constant $c>0$ independent of
$0<\eps\leq1$, such that the following hold true for every $(a,s,\xi)$
\begin{equation}
\label{eq:qf-strict-accretivity}
\Re
\left(
q_f(a,s,\xi)
\right)
\geq
c\rho_\eps(\xi) \,, \qquad \Re(-q_s(a,s,\xi)) \geq  c\mu_\eps(\xi)-c \eps\,.
\end{equation}
\end{enumerate}
\end{lemma}

\begin{proof}
We begin by studying  
$$\Delta_\eps := (2i\alpha^{0,j}\xi_j-\eps^{-1})^2 +4\alpha^{0,0}\alpha^{i,j}\xi_i\xi_j\,,$$
so that 
$$ q_f = \frac{\eps^{-1} - 2i \alpha^{0,j} \xi_j + \sqrt{\Delta_\eps}}{2 \alpha^{0,0}} \,, \qquad q_s =\frac{\eps^{-1} - 2i \alpha^{0,j} \xi_j - \sqrt{\Delta_\eps}}{2 \alpha^{0,0}} \,. $$
The real part of $\Delta_\eps$ is  $
\eps^{-2}+4\alpha^{0,0}\alpha^{i,j}\xi_i \xi_j-4(\alpha^{0,j}\xi_j)^2$ and we first prove that there exists $c_0>0$ such that
$$
4\alpha^{0,0}\alpha^{i,j}\xi_i\xi_j -4(\alpha^{0,j}\xi_j)^2 \geq c_0|\xi|^2 \,.
$$
Indeed, since the matrix $(\alpha^{i,j})_{i,j=0}^2$ is uniformly elliptic we have
$$
\alpha^{0,0}q^2+2q\alpha^{0,j}\xi_j+\alpha^{i,j}\xi_i\xi_j
\geq c_0(q^2+|\xi|^2).
$$
Hence, choosing
$ q=-\frac{\alpha^{0,j}\xi_j}{\alpha^{0,0}},$
we obtain
$ -\frac{(\alpha^{0,j}\xi_j)^2}{\alpha^{0,0}} +\alpha^{i,j}\xi_i\xi_j \geq c_0|\xi|^2$ and therefore we obtain
\begin{align} \label{eq:uniformity-symbol}
    \Re\Delta_\eps \geq c\bigl(\eps^{-2}+\langle\xi\rangle^2\bigr) = c\rho_\eps^2 \,, \qquad |\Delta_\eps|\leq C\rho_\eps^2\,.
\end{align}
By direct differentiation, we further get 
$$
\Delta_\eps\in S_\eps(\rho_\eps^2;\rho_\eps).
$$
Therefore, by the Fa\`a di Bruno formula and \eqref{eq:uniformity-symbol} we obtain
$$
\sqrt{\Delta_\eps}\in S_\eps(\rho_\eps;\rho_\eps), \qquad \frac{1}{\sqrt{\Delta_\eps}} \in S_\eps(\rho_\eps^{-1};\rho_\eps)\,.
$$
Hence, using the formula for the fast root it is straightforward to conclude that 
$q_f\in S_\eps(\rho_\eps;\rho_\eps)
\subset S_\eps(\rho_\eps).$
Furthermore,
\begin{align} \label{eq:uniformity-q-f}
    \Re q_f = \frac{\eps^{-1}+\Re\sqrt{\Delta_\eps}}{2\alpha^{0,0}} \geq c\rho_\eps\,, \qquad |q_f|\leq C\rho_\eps \,.
\end{align}
Hence, again using the Fa\`a di Bruno formula and \eqref{eq:uniformity-q-f} we obtain 
$$
q_f^{-1} \in S_\eps(\rho_\eps^{-1};\rho_\eps) \subset S_\eps(\rho_\eps^{-1})\,.
$$
Using the factorization of the polynomial \eqref{eq:polynomial} as $-\eps \alpha^{0,0} (q- q_f) (q-q_s)$ we deduce that $q_fq_s = -\frac{\alpha^{i,j}\xi_i\xi_j}{\alpha^{0,0}},$ and hence by using again the Fa\`a di Bruno formula, \eqref{eq:uniformity-q-f} and Lemma \ref{lemma:computations} we have 
$$q_s = -q_f^{-1}\frac{\alpha^{i,j}\xi_i\xi_j}{\alpha^{0,0}} \in S_\eps\left( \frac{\langle\xi\rangle^2}{\rho_\eps} \right) = S_\eps(\mu_\eps)\,.
$$
On the other hand, from the factorization of the polynomial we also have
$ q_f+q_s = \frac{\eps^{-1}-2i\alpha^{0,j}\xi_j}{\alpha^{0,0}} \in S_\eps(\rho_\eps;\rho_\eps), $ so that $ q_s\in S_\eps(\rho_\eps;\rho_\eps).$

We next observe that
$ q_f-q_s = \frac{\sqrt{\Delta_\eps}}{\alpha^{0,0}}$ and then using again \eqref{eq:uniformity-symbol} and the Fa\`a di Bruno formula we have 
$$
\frac{1}{q_f-q_s} \in S_\eps(\rho_\eps^{-1};\rho_\eps) \subset S_\eps(\rho_\eps^{-1})\,.
$$
Finally, multiplying and dividing the expression for the slow root by $\eps^{-1}-2i\alpha^{0,j}\xi_j+\sqrt{\Delta_\eps}$, we obtain 
$$ q_s = \frac{(\eps^{-1} - 2 i \alpha^{0,j} \xi_j)^2 - \Delta_\eps}{2 \alpha^{0,0} (\eps^{-1} - 2 i \alpha^{0,j} \xi_j + \sqrt{\Delta_\eps} )} = - \frac{2\alpha^{i,j}\xi_i\xi_j}{\eps^{-1}-2i\alpha^{0,j}\xi_j+\sqrt{\Delta_\eps}}\,. $$
Since 
$ \Re ( \eps^{-1}-2i\alpha^{0,j}\xi_j+\sqrt{\Delta_\eps} ) \gtrsim\rho_\eps $ and $| \eps^{-1}-2i\alpha^{0,j}\xi_j+\sqrt{\Delta_\eps}| \lesssim \rho_\eps$,
uniform ellipticity of the matrix $\alpha^{i,j}$ gives
$$
\Re(-q_s) \gtrsim |\xi|^2 \frac{\Re (\eps^{-1}-2i\alpha^{0,j}\xi_j+\sqrt{\Delta_\eps})}{|\eps^{-1}-2i\alpha^{0,j}\xi_j+\sqrt{\Delta_\eps}|^2} \gtrsim \frac{|\xi|^2}{\rho_\eps}
\approx  \frac{\eps|\xi|^2}{1+\eps\langle\xi\rangle}. 
$$
Finally, using the definition of $\mu_\eps$ we have 
$ \frac{\eps|\xi|^2}{1+\eps\langle\xi\rangle} = \mu_\eps(\xi) - \frac{\eps}{1+\eps\langle\xi\rangle} \geq \mu_\eps(\xi)-\eps $ and then 
$$
\Re(-q_s) \geq c \frac{\eps|\xi|^2}{1+\eps\langle\xi\rangle} \geq c \mu_\eps(\xi)-c \eps $$
concluding the proof of the three properties.
\end{proof}

With this lemma established, we may now prove Lemma \ref{lemma:factorisation}.
\begin{proof}[Proof of Lemma \ref{lemma:factorisation}]

For convenience, we write $\alpha=\alpha^{0,0}$ and define
$$
F_0=\Op_a(q_f),\qquad S_0=\Op_a(q_s).
$$
We first compare the factorized operator with the scalar part
$\mathcal L_{\lambda,\eps,1}$. Expanding the product gives
$$
-\eps \alpha (\partial_s-F_0)(\partial_s-S_0) = -\eps \alpha \partial_s^2 +\eps \alpha(F_0+S_0)\partial_s +\eps \alpha \,[\partial_s , S_0] -\eps \alpha F_0S_0\,.
$$
Hence, we deduce that 
\begin{align}
    \mathcal L_{\lambda,\eps}
    & = -\eps \alpha (\partial_s-F_0)(\partial_s-S_0) +\eps \alpha \Op_a ( q_f\#q_s-q_fq_s-\partial_sq_s  )+ \cL_{\lambda, \eps, 2} 
    \\
    & 
    = -\eps \alpha (\partial_s-F_0)(\partial_s-S_0) +\eps \alpha R_0 -\eps H^0 \partial_s \,,
\end{align}
where  $ R_0=\Op_a(r_0), $\footnote{We highlight that we did not include $- \eps H^0 \partial_s$ in the error $R_0$ since it is not an operator given by a  pseudodifferential symbol due to the derivative in $s$.}  with
$$
r_0 = q_f\#q_s-q_fq_s-\partial_sq_s -\alpha^{-1}(iH^j\xi_j+H)\,.
$$
By Theorem~\ref{thm:composition} and Lemma~\ref{lemma:properties of the roots},
$$ R_0\in \Psi_\eps(\langle\xi\rangle) \cap \Psi_\eps(\rho_\eps;\rho_\eps). 
$$
We now argue inductively. Suppose that, for some $J\geq0$, we have constructed
$F_J,S_J,R_J, L_J$ such that
\begin{equation}\label{eq:factorisation-induction}
\mathcal L_{\lambda,\eps}
=-\eps \alpha (\partial_s-F_J)(\partial_s-S_J)
 +\eps \alpha R_J-\eps \alpha L_J\partial_s,
\end{equation}
where $L_J=0$ for $J\geq1$ and $L_0 = H^0/\alpha$, and
$$ R_J\in \Psi_\eps(\langle\xi\rangle^{1-J}) \cap \Psi_\eps(\rho_\eps^{1-J};\rho_\eps), $$
and, in addition,
$$
F_J-F_0\in\Psi_\eps(1)\cap\Psi_\eps(1;\rho_\eps),
\qquad
S_J-S_0\in\Psi_\eps(\nu_\eps)\cap\Psi_\eps(1;\rho_\eps)\,.
$$
We define 
$$ F_{J+1} = F_J + \Op_a (f_{J+1})\,, \qquad S_{J+1 }= S_J + \Op_a(s_{J+1})$$
for some symbols $s_{J+1}, f_{J+1}$ to be defined.
Expanding the product $(\partial_s -F_{J+1}) (\partial_s - S_{J+1})$ we get
\begin{align*}
    (\partial_s -F_{J+1}) (\partial_s - S_{J+1})&  = (\partial_s - F_J) (\partial_s - S_J) 
    \\
    & \quad + \Op_a (f_{J+1}) S_0 + F_0 \Op_a (s_{J+1})  - \Op_a (f_{J+1} +s_{J+1})
    \partial_s 
    \\
    & \quad - [\partial_s , \Op_a (s_{J+1})] + \Op_a (f_{J+1})\Op_a (s_{J+1})
    \\
    & \quad + \Op_a (f_{J+1}) (S_J - S_0) + (F_J - F_0) \Op_a (s_{J+1}) \,.
\end{align*}
We now use the second line to cancel the error terms $R_J - L_J \partial_s $. More precisely, we expand  the term
\begin{align*}
    \Op_a (f_{J+1}) S_0 + F_0 \Op_a (s_{J+1})&  = \Op_a (f_{J+1} q_s + q_f s_{J+1}) 
    \\
    & \quad + \Op_a (f_{J+1}\# q_s - f_{J+1} q_s)  + \Op_a ( q_f \# s_{J+1}  - q_f s_{J+1})\,, 
\end{align*}
and using the notation  $R_J= \Op_a (r_J)$ we define
$$
f_{J+1} =\frac{r_J-q_fL_J}{q_f-q_s}, \qquad s_{J+1} =-\frac{r_J-q_sL_J}{q_f-q_s}\,,
$$
so that 
$$
f_{J+1}+s_{J+1}=-L_J,
\qquad
r_J+q_fs_{J+1}+f_{J+1}q_s=0.
$$
Therefore, imposing \eqref{eq:factorisation-induction} for $J+1$ we deduce that
\begin{align*}
R_{J+1}
& =- [ \partial_s, \Op_a (s_{J+1}) ]+\Op_a(f_{J+1})\Op_a(s_{J+1})
\\
& \quad +(F_J-F_0)\Op_a(s_{J+1}) +\Op_a(f_{J+1})(S_J-S_0) \\
&\quad 
 +\Op_a(q_f\#s_{J+1}-q_fs_{J+1}) +\Op_a(f_{J+1} \# q_s-f_{J+1}q_s)   \,.
\end{align*}
For $J=0$, the separation estimate for the roots gives
$$
f_1\in S_\eps(1)\cap S_\eps(1;\rho_\eps), \qquad s_1\in S_\eps(\nu_\eps)\cap S_\eps(1;\rho_\eps).
$$
For $J\geq1$, since $L_J=0$, we have $f_{J+1}=-s_{J+1}$ and
\begin{align}
f_{J+1},s_{J+1}
\in S_\eps \left(\frac{\langle\xi\rangle^{1-J}}{\rho_\eps}\right) \cap S_\eps(\rho_\eps^{-J};\rho_\eps)
= S_\eps(\nu_\eps\langle\xi\rangle^{-J}) \cap S_\eps(\rho_\eps^{-J};\rho_\eps), \label{eq:f-J+1-s-J+1}
\end{align}
where we have used Lemma \ref{lemma:computations} and Lemma \ref{lemma:properties of the roots}.
Using these estimates and Theorem~\ref{thm:composition}, we obtain
$$ R_{J+1}\in \Psi_\eps(\langle\xi\rangle^{-J}) \cap \Psi_\eps(\rho_\eps^{-J};\rho_\eps)\,.$$
The correction estimates \eqref{eq:f-J+1-s-J+1} also preserve the asserted bounds for $F_{J+1}-F_0$ and $S_{J+1}-S_0$, concluding the proof.
\end{proof}

\subsection{Sobolev bounds on the propagators}
Using Lemma \ref{lemma:factorisation} we have 
$$
\mathcal L_{\lambda,\eps} = -\eps \alpha^{0,0} (\partial_s-F_J(s))(\partial_s-S_J(s)) + l.o.t.
$$

We now study the propagators associated to our factorised problems. We define the slow propagator $U_S(s,t)$ evolving forward for $s \geq t$ by 
$$
\partial_s U_S(s,t)=S_J(s)U_S(s,t), \quad U_S(t,t)=\Id\,,
$$
and the fast propagator $U_F(s,t)$ evolving backward  for $s \leq t$ by 
$$
\partial_s U_F(s,t)=F_J(s)U_F(s,t), \quad U_F(t,t)=\Id \,.
$$
The existence of such propagators may be justified using a standard Galerkin argument, and so we restrict ourselves to \emph{a priori} estimates. 
Throughout this section, we seek to analyse the boundedness properties of $U_F, U_S$ on  classical Sobolev spaces and weighted Sobolev spaces  induced by the norm 
$$\|B\|^2_{H_{\rho_\eps}^r }=\|\rho_\eps(D)B\|^2_{H^r}\,.$$
We begin with the following lemma.
\begin{lemma}
\label{lemma:sobolev bounds propagators}
Let $r \in \RR$. Then, uniformly for  $\lambda \in K\subset \CC$ in a compact set, there exists a constant $C_r$, so that the slow and fast propagators satisfy the following estimates for $ s \in [0,T]$
\begin{align*}
    \|U_S(s,0)B\|_{H^r} \leq (1+C_r\eps)\|B\|_{H^r}\,,  
\qquad
    \|U_F(s,T) B\|_{H^r} \leq C_r\e^{- c\eps^{-1}(T- s)}\|B\|_{H^r} \,.
\end{align*}
Furthermore, for any $q>0, r_1,r_2 \in \RR$, there exists $C_{r_1,r_2,q}>0$ so that if $ 0\leq s \leq T- q$, we have
$$
\|U_F(s,T)B\|_{H^{r_1}} \leq C_{r_1,r_2,q}\e^{-c\eps^{-1}(T-s)}\|B\|_{H^{r_2}}\,,
$$
and for any $r \in \RR$ there exists $C_r >0$ so that for $s \in [0,T]$ we have
$$
\|U_S(s,0) B\|_{H_{\rho_\eps}^r} \leq C_r \|B\|_{H_{\rho_\eps}^r}.
$$
\end{lemma}
\begin{proof}
We first prove the result for the slow propagator. We begin by applying $\langle D \rangle^{r}$ to the equation, obtaining 
$$
\partial_s \langle D \rangle^{r} B=S_J(s)\langle D \rangle^{r}B+[\langle D \rangle^{r},S_J(s)]B.
$$
Taking the $L^2$ inner product with $\langle D \rangle^r B$, we obtain
\begin{align}
\label{eq:energy evolution slow sobolev}
\frac{1}{2}\partial_s \|\langle D \rangle^{r} B\|_{L^2}^2 & =\Re \langle S_J(s) \langle D \rangle^{r} B, \langle D \rangle^{r} B\rangle +\Re \langle [\langle D \rangle^{r}, S_J(s)]B, \langle D \rangle^{r} B\rangle
\\
& = \Re \langle \Op_a (q_s (s)) \langle D \rangle^{r} B, \langle D \rangle^{r} B\rangle \notag
\\
& \quad +  \Re \langle (S_J -\Op_a (q_s (s))) \langle D \rangle^{r} B, \langle D \rangle^{r} B\rangle  \notag
\\
& \quad + \Re \langle [\langle D \rangle^{r}, S_J(s)]B, \langle D \rangle^{r} B\rangle \notag
\end{align}
For the first term we use Lemma \ref{lemma:properties of the roots} and the sharp G\r{a}rding inequality (Theorem \ref{thm:sharp garding}), applied to $\Op_a (-q_s (s)) - c\Op_a (\mu_\eps (\xi)) + c\eps$ to deduce
\begin{align*}
    & \Re \langle \Op_a (-q_s (s)) \langle D \rangle^{r} B, \langle D \rangle^{r} B\rangle 
    \\
    & \qquad \geq c\|  \Op_a(\mu_\eps (\xi)^{1/2}) \langle D\rangle^r B\|_{L^2}^2  - c \eps \| \langle D\rangle^r B\|_{L^2}^2 - C \|\Op_a\left(\nu_\eps(\xi)^{1/2}\right)  \langle D\rangle^r B\|_{L^2}^2 \,. 
\end{align*}
For the second term we use Lemma \ref{lemma:factorisation} and we estimate
$$ |\Re \langle (S_J -\Op_a (q_s (s))) \langle D \rangle^{r} B, \langle D \rangle^{r} B\rangle | \leq C \|  \nu_\eps^{1/2}(D) \langle D\rangle^r B \|_{L^2}^2 \,.$$
For the third term we write
$$
\Re \langle [\langle D \rangle^{r}, S_J(s)]B, \langle D \rangle^{r} B\rangle=\Re \langle [\langle D \rangle^{r}, S_J(s)]\langle D \rangle^{-r} \langle D \rangle^{r} B, \langle D \rangle^{r} B\rangle
$$
and notice that $ [\langle D \rangle^{r},S_J(s)]\langle D \rangle^{-r} \in \Psi_\eps(\nu_\eps)$ to deduce that 
$$|\Re \langle [\langle D \rangle^{r}, S_J(s)]B, \langle D \rangle^{r} B\rangle| \leq C\| \Op_a\left(\nu_\eps(\xi)^{1/2}\right) \langle D \rangle^{r} B\|_{L^2}^2\,.$$
All in all we deduce that 
\begin{align*}
\frac{1}{2}\partial_s \|\langle D \rangle^{r} B\|_{L^2}^2 & \leq -c\|  \Op_a(\mu_\eps (\xi)^{1/2}) \langle D\rangle^r B\|_{L^2}^2  +C_1 \eps \| \langle D\rangle^r B\|_{L^2}^2 + C_2 \|\Op_a\left(\nu_\eps(\xi)^{1/2}\right)  \langle D\rangle^r B\|_{L^2}^2 
\\
& \leq  -\frac{c}{2}\|  \Op_a(\mu_\eps (\xi)^{1/2}) \langle D\rangle^r B\|_{L^2}^2  +C_3 \eps \| \langle D\rangle^r B\|_{L^2}^2 \,,
\end{align*}
where in the last inequality, we have used Lemma \ref{lemma:computations} to deduce that for any $\eta >0$, chosen to be sufficiently small depending on $c$, we have
$$ 
\|\nu_\eps^{1/2}(D) \langle D\rangle^r B\|_{L^2}^2 \leq \eta \|\mu_\eps^{1/2}(D) \langle D\rangle^rB\|_{L^2}^2 + \eta^{-1} \eps \|\langle D\rangle^r B\|_{L^2}^2 \,,
$$
yielding the desired Sobolev boundedness result. The bounds for $H^r_{\rho_\eps}$ follow by the same argument by replacing $\langle D \rangle^r$ with $\rho_\eps(D)\langle D \rangle^r$. 

Next, we move on to the fast propagator bounds.  We begin by computing the action of $D^\alpha, |\alpha|\leq m \in \NN$ on the evolution equation
$$
\partial_s D^\alpha B=F_J(s) D^\alpha B+[D^\alpha, F_J(s)]B.
$$
First, note that in view of Lemma \ref{lemma:factorisation}, Lemma \ref{lemma:properties of the roots},  and the sharp G\r{a}rding inequality (Theorem \ref{thm:sharp garding}) we deduce the existence of a constant $C>0$ so that 
$$
\Re \langle F_J(s) D^\alpha B, D^\alpha B\rangle \geq c\|\rho_\eps^{1/2}(D) D^\alpha B\|_{L^2}^2-C\|D^\alpha B\|_{L^2}^2.
$$
Since $\rho_\eps(\xi) \geq \eps^{-1}$, this implies for $\eps $ small enough that 
$$
\Re \langle F_J(s) D^\alpha B, D^\alpha B\rangle \geq \frac{c}{2}\|\rho_\eps^{1/2}(D) D^\alpha B\|_{L^2}^2.
$$
Thus, it remains to deal with the commutator term. Since $m \in \NN$, we may compute the commutator explicitly using the composition formula, giving 
$$
[ D^\alpha,F_J(s)]=\sum_{0<\beta \leq \alpha} \begin{pmatrix}
    \alpha
    \\
    \beta
\end{pmatrix} \Op_a(D_a^\beta Q_{f,J}(s))D^{\alpha-\beta},
$$
where $Q_{f,J}$ is the symbol of $F_J(s)$.  We note that $D_a^\beta Q_{f,J} \in S_\eps(\rho_\eps;\rho_\eps)$, for all $\beta$. Consequently, the composition formula (Theorem \ref{thm:composition}) and Lemma \ref{lemma:properties of the roots} imply that 
$$
\rho_\eps^{-1/2}(D) \Op_a(D_a^\beta Q_{f,J}(s)) \rho_\eps^{-1/2} (D) \in \Psi_\eps(1;\rho_\eps)\,.
$$
Hence, we may bound 
\begin{align*}
&|\langle \Op_a(D_a^\beta Q_{f,J}(s))D^{\alpha-\beta} B, D^\alpha B\rangle|\\
&\qquad =\left |\langle \left (\rho_\eps^{-1/2}(D) \Op_a(D_a^\beta Q_{f,J}(s)) \rho_\eps^{-1/2}(D) \right )\rho_\eps(D)^{1/2} D^{\alpha-\beta}B, \rho_\eps(D)^{1/2} D^\alpha B \rangle\right |\\
& \qquad \leq C \|\rho_\eps^{1/2}(D) D^{\alpha-\beta}B\| \|\rho_\eps^{1/2}(D) D^\alpha B\|_{L^2}\,.
\end{align*}
Set now $\mathcal{E}_m(s)=\sum_{0  \leq  |\alpha|\leq m}\eta^{|\alpha|} \|D^\alpha B(T-s)\|_{L^2}^2$, for a constant $\eta \ll 1$ to be picked later. Combining our previous bounds with Young's inequality yields the following energy inequality
\begin{align}
&\frac{\dd}{\dd s} \mathcal{E}_m(s)+\frac{c}{2} \sum_{0 \leq |\alpha|\leq m}\eta^{|\alpha|}\|\rho_\eps^{1/2}(D) D^{\alpha} B(T-s)\|_{L^2}^2\\
&\qquad\leq C\delta \sum_{0  \leq  |\alpha|\leq m} \eta^{|\alpha|}  \|\rho_\eps^{1/2}(D)D^\alpha B(T-s)\|_{L^2}^2 
\\
&\qquad \quad + \sum_{\substack{0\leq |\alpha|\leq m\\ 0<\beta\leq\alpha}} C(\delta) \eta^{|\alpha| } \|\rho_\eps^{1/2}(D)D^{\alpha-\beta} B(T-s)\|_{L^2}^2 \,.
\end{align}
We therefore choose $\delta >0$ sufficiently small depending on $c$  to absorb the first term on the right-hand side into the left-hand side, so that the following holds true
\begin{align*}
&\frac{\dd}{\dd s} \mathcal{E}_m(s)+\frac{c}{3} \sum_{0 \leq |\alpha|\leq m}\eta^{|\alpha|}\|\rho_\eps^{1/2}(D) D^{\alpha} B(T-s)\|_{L^2}^2 \leq C(\delta) \sum_{\substack{0\leq |\alpha|\leq m\\ 0<\beta\leq\alpha}} \eta^{|\alpha|} \|\rho_\eps^{1/2}(D)D^{\alpha-\beta} B(T-s)\|_{L^2}^2 \,,
\end{align*}
Furthermore, noting that the sum is non-trivial if $|\alpha|>0$, we deduce that 
\begin{align*}
\sum_{\substack{0\leq |\alpha|\leq m\\ 0<\beta\leq\alpha}} \eta^{|\alpha|}\|\rho_\eps^{1/2}(D) D^{\alpha-\beta} B(T-s)\|_{L^2}^2& =\sum_{\substack{0\leq |\alpha|\leq m\\ 0<\beta\leq\alpha}} \eta^{|\alpha|-|\alpha-\beta|}\eta^{|\alpha-\beta|} \|\rho_\eps^{1/2}(D) D^{\alpha-\beta} B(T-s)\|_{L^2}^2\\
& \leq C_m \eta \sum_{0\leq |\alpha|\leq m} \eta^{|\alpha|} \|\rho_\eps^{1/2}(D) D^\alpha B(T-s)\|_{L^2}^2\,,
\end{align*}
where we have used $\eta^{|\alpha|-|\alpha-\beta|} \leq \eta$ and replaced the sum over $\alpha- \beta$ by $\alpha$.
As such, taking $\eta$ small enough depending on $C_m$ and $C(\delta)$, we deduce the energy inequality
\begin{equation}
\label{eq:fast energy inequality}
\frac{\dd}{\dd s}\mathcal{E}_m(s) +\frac{c}{4} \sum_{0 \leq |\alpha|\leq m}\eta^{|\alpha|} \|\rho_\eps^{1/2}(D) D^\alpha B(T-s)\|_{L^2}^2 \leq 0.
\end{equation}
 Since $\rho_\eps^{1/2}\gtrsim\eps^{-1/2}$, integrating this inequality yields the $H^r \to H^r$ bound for the fast propagator, for $r=m \in \NN$. The result for any $r \geq 0$ follows by interpolation. We note that the same estimates hold true for the adjoint $U_F(s,T)^*$ and hence $\|U_F(s,T) B\|_{H^r} \leq C_r \e^{- c\eps^{-1}(T- s)}\|B\|_{H^r} $ holds true for any $r \in \RR$ by duality.

Thus, it remains to prove the general Sobolev bounds for the fast propagator. The key observation is this: integrating \eqref{eq:fast energy inequality} over $[0,s]$, we observe that 
$$
\frac{c}{4}\int_0^s \|B(T-t)\|_{H^{m+1/2}}^2 dt \leq \|B(T)\|^2_{H^m}.
$$
As such, there exists $t_0 \in [0,s]$ so that 
$ \|B(T-t_0)\|_{H^{m+1/2}}^2 \leq \frac{4}{cs} \|B(T)\|^2_{H^m}. $
Thus, by the $H^{m+1/2}\to H^{m+1/2}$ bound for the propagator and the inequality $T-s\leq T-t_0$, we deduce
$$
\|B(T-s)\|_{H^{m+1/2}}^2 \leq C\|B(T-t_0)\|_{H^{m+1/2}}^2 \leq \frac{4C}{cs}\|B(T)\|_{H^m}^2\,.
$$
Hence, given $0 \leq r_1<r_2 $, splitting the interval $[0,q]$ into $2(\lceil r_2 \rceil-\lfloor r_1 \rfloor)$ pieces of equal length, we deduce that 
$$
\|B(T-q)\|_{H^{r_2}} \leq C_{r_1,r_2}\left (\frac{\lceil r_2 \rceil-\lfloor r_1 \rfloor}{cq} \right )^{\lceil r_2 \rceil-\lfloor r_1 \rfloor}\|B(T)\|_{H^{r_1}} \,,
$$
yielding the desired bound. Once again, arguing for the adjoint proves the case when $r_1,r_2<0$. To prove the case $r_1>0$ and $r_2 <0$  we split the interval $[0,q]$ into two intervals of equal length and we consider the case $H^{r_2} \to L^2$ in the interval $[0,q/2]$ and $L^2 \to H^{r_1}$ in the interval $[q/2, q]$.
\end{proof}
\subsection{The homogeneous factorised problem}
We may now solve the approximate equation with boundary conditions 
\begin{align} \label{eq:equation-homogeneous-approx}
    \begin{cases}
        \eps \alpha^{0,0}(\partial_s-F_J(s))(\partial_s-S_{J}(s))B(s)=0
        \\
        B(0)=B_{0}\,,
        \\
        B(T)=B_T\,.
    \end{cases}
\end{align}
 We define $w(s):=(\partial_s-S_J(s))B(s)$. Then $ (\partial_s-F_J(s))w(s)=0.$
Since the fast propagator evolves backwards in time, we have
$$ w(s)=U_F(s,T)w(T),\qquad 0\leq s\leq T. $$
We now solve forward in time  from $s=0$ the equation $ (\partial_s-S_J(s))B(s)=w(s)$. By Duhamel's formula,
\begin{align*}
    B(s)& =U_S (s,0) B_{0} + \int_0^s U_S (s, t) w(t) \dd t\,,
    \\
    & = U_S (s,0) B_{0} + \int_0^s U_S (s, t) U_F (t,T) w(T) \dd t \,.
\end{align*}
Hence, imposing $B(T)=B_T$ we deduce 
\begin{align} \label{eq:comp-expl-form}
    \int_0^T U_S (T, t) U_F (t,T) w(T) \dd t = B_T - U_S (T,0) B_{0 } \,. 
\end{align}
We define the operator
\begin{align} \label{d:M-eps}
     M_{SF} =\int_0^T U_S(T,t)U_F(t,T) \dd t
\end{align}
and the following lemma proves that $M_{SF}:  H^r \to H^r_{\rho_\eps}$ is invertible so that we deduce
$$ w(T)= M_{SF}^{-1} (B_T- U_S (T,0)B_{0}) \,,$$
and therefore 
$$B(s) = U_S (s,0) B_{0} + \int_0^s U_S (s, t) U_F (t,T) M_{SF}^{-1}  (B_T - U_S(T,0) B_0) \dd t \,.
$$
The main proposition we aim to prove is the following.
\begin{proposition}
\label{prop:homogeneous decomposition}
Let $B_0\in H^{r_1}$, $B_T \in H^{r_2}$ with $r_1, r_2 \in \RR$ and $r \in \RR$. Then, the solution to \eqref{eq:equation-homogeneous-approx} is given by 
$$
B(s) = U_S (s,0) B_{0} + \int_0^s U_S (s, t) U_F (t,T) M_{SF}^{-1}  (B_T - U_S(T,0) B_0) \dd t \,.
$$
Furthermore, define $\|v\|_{\gamma,r}=\sup_{0\leq s\leq T}\|\e^{\gamma\eps^{-1}(T-s)}v\|_{H^r}$. For all sufficiently small $\eps>0$ and all $\gamma\geq0$ in a sufficiently small interval independent of $\eps$, the following estimates hold 
$$ \left \|\int_0^s U_S(s,t)U_F(t,T) M_{SF}^{-1} (B_T-U_S(T,0)B_0) \dd t \right \|_{\gamma,r}\leq C_r  \left ( \|B_T\|_{H^{r+1}}+\|B_0\|_{H^{r+1}} \right ) \,,$$
and for a fixed $\delta>0 $ the following holds  for all $s \leq T -\delta$
$$ \left \|\int_0^s U_S(s,t)U_F(t,T)M_{SF}^{-1}(B_T-U_S(T,0)B_0) \dd t \right \|_{H^{r}} \leq C_{r,r_1, r_2, \delta} \e^{-c\eps^{-1}(T-s)}\left (\|B_T\|_{H^{r_2}}+\|B_0\|_{H^{r_1}} \right )\,. $$
\end{proposition}
The following lemma is fundamental for the proof of the proposition.

\begin{lemma} \label{lemma:boundary term bounds}
The operator  $M_{SF}: H^r \to H^r_{\rho_\eps}$ is invertible for any $\eps$ sufficiently small and for any $r\in \RR$. Furthermore, for any $r\in \RR$ there exists a constant $C_r$ independent of $\eps$ so that
$$ \|M_{SF}^{-1}\|_{H^r_{\rho_\eps} \to H^r} \leq C_r\,. $$
\end{lemma}

Before starting the proof, we give a short description of the strategy. The key observation is that $M_{SF}$ behaves, at leading order, as  the pseudodifferential operator
$$
\cQ(s) := \Op_a \left(\frac{1}{q_f(s)-q_s(s)}\right).
$$
By Lemma \ref{lemma:properties of the roots}, we have
$$
\cQ (s)\in
\Psi_\eps(\rho_\eps^{-1};\rho_\eps) \,.
$$
We first observe that  by the propagator composition property we have
$$
\partial_s U_S(T,s)=\lim_{h \to 0} \frac{ U_S(T,s+h) -U_S(T,s)}{h}=\lim_{h \to 0}\frac{ U_S(T,s+h)\left (\Id-U_S(s+h,s) \right )}{h}=-U_S(T,s)  S_J(s) \,.
$$
Hence, using also that $\partial_s U_F (s,T) =F_J(s) U_F (s,T)$ we deduce 
\begin{align*}
    \partial_s (U_S(T,s)\cQ(s)U_F(s,T) ) & = U_S(T,s) ( -S_J(s)\cQ (s) +\partial_s\cQ (s) +\mathcal Q(s)F_J(s) ) U_F(s,T)
    \\
    & = U_S(T,s) ( \Id + \cR (s)) U_F(s,T)\,,
\end{align*}
where $\cR (s)$ is defined as 
$$\cR(s)= \partial_s\cQ (s) +[\cQ (s),S_J(s)] +\cQ (s) (F_J(s)-S_J(s) ) -\Id \,.$$
 Hence, integrating the previous equation over $s\in[0,T]$ and rearranging the terms, we obtain
\begin{align} \label{eq:difference-Q-M}
    \cQ(T) - M_{SF}  =  U_S(T,0)\cQ(0)U_F(0,T) + \int_0^T U_S(T,s) \cR(s)U_F(s,T)\,\dd s\,,
\end{align}
and we aim to prove that the terms on the right hand side are small and $\cQ(T)$ is invertible.

\begin{proof}[Proof of Lemma \ref{lemma:boundary term bounds}]
    We first study the reference operator $\cQ(s)$.
By Lemma \ref{lemma:properties of the roots} we have
$ \mathcal Q(s) \in \Psi_\eps(\rho_\eps^{-1};\rho_\eps) $ and
consequently, we deduce
\begin{equation} \label{eq:Q-boundedness}
\| \cQ(s)B\|_{H^r_{\rho_\eps}} \leq C_r\|B\|_{H^r} \,.
\end{equation}
We next show that $\mathcal Q(s)$ is a uniformly invertible map from $H^r $ to $H^r_{\rho_\eps}$. To this end, we define
$$
\cQ^\sharp(s) := \Op_a\bigl(q_f(s)-q_s(s)\bigr).
$$
Again by Lemma \ref{lemma:properties of the roots}, $ \cQ^\sharp(s) \in \Psi_\eps(\rho_\eps;\rho_\eps),$ and hence  $\| \cQ^\sharp(s)B\|_{H^r} \leq C_r\|B\|_{H^r_{\rho_\eps}}.$
The composition formula (Theorem \ref{thm:composition}) gives
\begin{align*}
\mathcal E_\ell(s) :=
\cQ^\sharp(s)\cQ(s)-\Id \in \Psi_\eps(\rho_\eps^{-1};\rho_\eps),
\qquad
\mathcal E_r(s) := \cQ(s)\cQ^\sharp(s)-\Id \in
\Psi_\eps(\rho_\eps^{-1};\rho_\eps).
\end{align*}
Since $\rho_\eps(\xi)^{-1}
\leq
\eps,$ Corollary \ref{corollary:sobolev boundedness} implies
\begin{align*}
\|\mathcal E_\ell(s)\|_{H^r\to H^r} + \|\mathcal E_r(s)\|_{H^r_{\rho_\eps}\to H^r_{\rho_\eps}} \leq C_r\eps.
\end{align*}
For $\eps>0$ sufficiently small, both $ \cQ^\sharp(s)\cQ(s) = \Id+\mathcal E_\ell(s) $
on $H^r$, and
$  \cQ(s)\cQ^\sharp(s) = \Id+\mathcal E_r(s) $
on $H^r_{\rho_\eps}$, are invertible by a Neumann-series argument. Therefore, $ (\cQ^\sharp(s)\cQ(s) )^{-1} \mathcal Q^\sharp(s) $ is a left inverse of $\cQ(s)$, while
$ \mathcal Q^\sharp(s) (\mathcal Q(s)\mathcal Q^\sharp(s) )^{-1} $ is a right inverse. Hence,
$ \mathcal Q(s):H^r \to H^r_{\rho_\eps} $
is a bounded isomorphism satisfying  
\begin{equation}
\label{eq:Q-inverse-bound}
\|\mathcal Q(s)^{-1}B\|_{H^r} \leq
C_r\|B\|_{H^r_{\rho_{\eps}}} \,,
\end{equation}
with a constant independent of $\eps$. We now recall $\cR(s)= \partial_s\cQ (s) +[\cQ (s),S_J(s)] +\cQ (s) (F_J(s)-S_J(s) ) -\Id \,.$ By Lemma \ref{lemma:properties of the roots} we have $\partial_s\mathcal Q(s) \in \Psi_\eps(\rho_\eps^{-1};\rho_\eps).$  

By Lemma \ref{lemma:factorisation}, $F_J(s)-S_J(s)-\mathcal Q^\sharp(s) \in \Psi_\eps(1;\rho_\eps)$.
Thus,
\begin{align*}
(F_J-S_J )\cQ -\Id &= (F_J-S_J- \cQ^\sharp )\cQ + (\cQ^\sharp\mathcal Q-\Id ) \in \Psi_\eps(\rho_\eps^{-1};\rho_\eps).
\end{align*}
Finally, since the relevant principal symbols are scalar, Theorem \ref{thm:composition} implies
$$
[\mathcal Q(s),F_J(s)] \in \Psi_\eps(\rho_\eps^{-1};\rho_\eps).
$$
Therefore, we deduce that $\mathcal R(s)
\in \Psi_\eps(\rho_\eps^{-1};\rho_\eps)$ and 
\begin{equation} \label{eq:R-H-Hrho}
\|\cR(s)B\|_{H^r_{\rho_\eps}} \leq C_r \|B\|_{H^r}\,.
\end{equation}
We now turn to proving that the terms on the right hand side of \eqref{eq:difference-Q-M} are small. In particular, using Lemma \ref{lemma:sobolev bounds propagators} and
\eqref{eq:Q-boundedness}, we obtain
\begin{align*}
\| U_S(T,0)\mathcal Q(0)U_F(0,T)B \|_{H^r_{\rho_\eps}} &\leq C_r \|\cQ(0)U_F(0,T)B\|_{H^r_{\rho_\eps}}
\\ 
&\leq C_r \|U_F(0,T)B\|_{H^r}
\\
&\leq C_r\e^{-cT/\eps}\|B\|_{H^r}.
\end{align*}
Similarly, using also \eqref{eq:R-H-Hrho},  we have
\begin{align*}
  \int_0^T \| U_S(T,s)\mathcal R(s)U_F(s,T)B \|_{H^r_{\rho_\eps}} \,\dd s &\leq
C_r \int_0^T \|\mathcal R(s)U_F(s,T)B\|_{H^r_{\rho_\eps}}\,\dd s \leq C_r \int_0^T \|U_F(s,T)B\|_{H^r}\,\dd s
\\
& \leq C_r \int_0^T \e^{-c(T-s)/\eps}\,\dd s \|B\|_{H^r}  \leq C_r\eps\|B\|_{H^r}.
\end{align*}
Thus, we have
\begin{align} \label{eq:estimate-M-Q}
\|M_{SF}-\mathcal Q(T)\|_{H^r\to H_{\rho_\eps}^r} \leq C_r\eps.
\end{align}
Therefore, by standard manipulation we get
$$
M_{SF} = \left( \Id+ (M_{SF}-\mathcal Q(T))\mathcal Q(T)^{-1} \right) \mathcal Q(T).
$$
Moreover, by \eqref{eq:Q-inverse-bound}, we have
\begin{align*}
\| (M_{SF}-\mathcal Q(T))\mathcal Q(T)^{-1}B \|_{H_{\rho_\eps}^r}  \leq C_r\eps \|\mathcal Q(T)^{-1}B\|_{H^r} \leq C_r\eps \|B\|_{H_{\rho_\eps}^r}.
\end{align*}
Hence,
$$
\|
(M_{SF}-\mathcal Q(T))\mathcal Q(T)^{-1}
\|_{H_{\rho_\eps}^r\to H_{\rho_\eps}^r}
\leq
C_r\eps.
$$
For $\eps>0$ sufficiently small, the operator
$ \Id+ (M_{SF}-\mathcal Q(T))\mathcal Q(T)^{-1} $
is therefore invertible on $H_{\rho_\eps}^r$ by a Neumann-series
argument. Consequently,
$$
M_{SF}^{-1} = \mathcal Q(T)^{-1} (\Id+ (M_{SF}-\mathcal Q(T))\mathcal Q(T)^{-1} )^{-1}
$$
is a bounded map from $H_{\rho_\eps}^r$ to $H^r$ with $ \|M_{SF}^{-1}\|_{H_{\rho_\eps}^r\to H^r} \leq C_r, $
with $C_r$ independent of $\eps$, concluding the proof.
\end{proof}

We are now ready to prove the proposition. 

\begin{proof}[Proof of Proposition \ref{prop:homogeneous decomposition}] 
By Lemma \ref{lemma:boundary term bounds} the formula is justified by the computations given in \eqref{eq:comp-expl-form}.
We now prove the second inequality.
    Fix $s \in [0,T-\delta]$, for $\delta>0$ fixed, and set $\sigma=\min\{r_1,r_2\}$. Then, by Lemma \ref{lemma:sobolev bounds propagators}, there holds 
$$
\|B_T-U_S(T,0)B_0\|_{H^\sigma} \leq C(\|B_T\|_{H^{r_2}}+\|B_0\|_{H^{r_1}}).
$$
Hence, for any $r\in\RR$, Lemma \ref{lemma:sobolev bounds propagators} gives the following bound, uniformly for $t\leq s\leq T-\delta$
$$
\|U_S(s,t)U_F(t,T)M_{SF}^{-1}B\|_{H^{r}} \leq C\|U_F(t,T)M_{SF}^{-1} B\|_{H^{r}} \leq C\e^{-c\eps^{-1}}\|M_{SF}^{-1} B\|_{H^{\sigma-1}} \leq C\e^{-c\eps^{-1}}\|B\|_{H_{\rho_\eps}^{\sigma-1}},
$$
where $C$ depends on $r,\sigma, \delta$. We note that $\|B\|_{H_{\rho_\eps}^{\sigma-1}} \leq C \eps^{-1}\|B\|_{H^{\sigma}}$, and so we conclude the proof of the second statement. Next, Lemma \ref{lemma:sobolev bounds propagators} and Lemma \ref{lemma:boundary term bounds} give
\begin{align*}
    \e^{\gamma \eps^{-1}(T-s)}\|U_S(s,t)U_F(t,T) M_{SF}^{-1} B\|_{H^r} & \leq  C\e^{-\eps^{-1}(c(T-t)-\gamma(T-s))}\|B\|_{H_{\rho_\eps}^r}\,.
\end{align*}
Hence, integrating this yields
$$
\|B\|_{H_{\rho_\eps}^r} \ \e^{\gamma \eps^{-1}(T-s)}\int_0^s \e^{-c\eps^{-1}(T-t)} \dd t  \leq C\|B\|_{H^{r+1}},
$$
completing the proof by choosing $\gamma< c$.
\end{proof}

\subsection{The inhomogeneous factorised problem}
In this section, we aim to solve the inhomogeneous PDE 
\begin{align} \label{eq:inhomogeneous}
\begin{cases}
    (\partial_s-F_J)(\partial_s-S_J)B=v\,,
    \\
     B(0)=B(T)=0\,.
\end{cases}
\end{align}
We define $w=(\partial_s -S_J )B$. Then, $(\partial_s-F_J)w=v$,  and so, using Duhamel's formula backward in time, we obtain
$$
w(s)=U_F(s,T)w(T)-\int_s^TU_F(s,t)v(t) \dd t.
$$
Thus, we have 
$$
B(s)=\int_0^sU_S(s,t)\left (U_F(t,T)w(T)-\int_t^TU_F(t,r)v(r) \dd r \right ) \dd t.
$$
Verifying the boundary conditions at $s=T$ is then equivalent to 
$$
0=\int_0^T U_S(T,t)U_F(t,T) w(T)\dd t -\int_0^T U_S(T,t)\int_t^TU_F(t,r)v(r) \dd r  \dd t.
$$
In other words, we have 
\begin{align} \label{eq:formula-wT}
    w(T)= M_{SF}^{-1}  \left (\int_0^T U_S(T,t)\int_t^TU_F(t,r)v(r) \dd r  \dd t \right ).
\end{align}
Hence, we deduce the following result.
\begin{proposition}
\label{prop:inhomogeneous}
Let $\cG_{\eps}: v \to B$ be the map defined so that $B$ is the solution to \eqref{eq:inhomogeneous}  with inhomogeneous datum $v$. Then, for any $r \in \RR$, there exists $C_r>0$ so that 
$$
\|\mathcal{G}_\eps v\|_{L^\infty([0,T];H^r)} \leq C_r \eps \|v\|_{L^\infty([0,T];H^r)}.
$$
Furthermore, define $\|v\|_{\gamma,r}=\sup_{0\leq s\leq T}\|\e^{\gamma\eps^{-1}(T-s)}v\|_{H^r}$. There exists a sufficiently small $\gamma>0$, independent of $\eps$, such that, for all sufficiently small $\eps>0$,
$$
\|\mathcal{G}_\eps v\|_{\gamma,r} \leq C_r \eps \|v\|_{\gamma,r}.
$$
\end{proposition}

\begin{proof}
We undertake the proof directly with the exponential weight, since the case without the weight is identical.
We begin by noting that we may rewrite the integral as
$$
\int_0^T \int_t^T U_S(T,\tau) U_S(\tau,t)U_F(t,\tau)v(\tau) \dd \tau \dd t=\int_0^T U_S(T,\tau)\int_0^\tau U_S(\tau,t) U_F(t,\tau)  \dd t \ v(\tau) \dd \tau.
$$
By \eqref{eq:Q-boundedness} and following \eqref{eq:estimate-M-Q}, it holds that  for $\tau \in [0,T]$
$$
\left |\int_0^\tau U_S(\tau,t) U_F(t,\tau) \dd t  \right |_{H^r \to H_{\rho_\eps}^r} \leq C.
$$
Hence, by Lemma \ref{lemma:sobolev bounds propagators} we have
$$
\left |\int_0^T U_S(T,\tau)\left (\int_0^\tau U_S(\tau,t) U_F(t,\tau)  \dd t \right ) v(\tau) \dd \tau \right |_{H_{\rho_\eps}^r} \leq C |v|_{L^\infty([0,T];H^r)}.
$$
Thus, by formula \eqref{eq:formula-wT} and Lemma \ref{lemma:boundary term bounds} we may bound
$$
|w(T)|_{H^r} \leq C |v|_{L^\infty([0,T];H^r)}.
$$
Therefore, using also Lemma \ref{lemma:sobolev bounds propagators} we obtain the bound
\begin{align*}
&\sup_{s \leq T}\e^{\gamma \eps^{-1}(T-s)}\left |\int_0^s U_S(s,t)U_F(t,T)w(T) \dd t  \right |_{H^r} \\
&\qquad \leq C  |v|_{L^\infty([0,T];H^r)} \sup_{s \leq T}\e^{\gamma\eps^{-1}(T-s)}\int_0^s \e^{-c_0\eps^{-1}(T-t)} \dd t \leq C\eps|v|_{\gamma,r},
\end{align*}
so long as $\gamma < c_0$.
Similarly, we bound
\begin{align*}
\sup_{s \leq T} \e^{\gamma \eps^{-1}(T-s)}& \left | \int_0^s U_S(s,t)\int_t^TU_F(t,\tau) v(\tau)  \dd \tau \dd t\right |_{H^r} 
\\
& \leq C \sup_{s \leq T}\int_0^s\int_t^T \e^{-c\eps^{-1}(\tau-t)}\e^{-\gamma \eps^{-1}(s-\tau)} \dd \tau \dd t \ |v|_{\gamma,r}\\
& \leq C\eps |v|_{\gamma,r},
\end{align*}
so long as $\gamma<c$. Thus the proof is complete.
\end{proof}

\subsection{The solution decomposition}
Finally, we consider the full PDE
\begin{align}
\begin{cases} \label{eq:full equation}
\mathcal L_{\lambda,\eps} B=0, 
\\
B(0)=B_0 \in \mathcal{X} \,,
\\
B(T)=B_T \in L^2 (\TT^2)\,.
\end{cases}
\end{align}
From the previous results we make rigorous the following property 
$$ B(s) \approx U_S(s,0) B_0 \,.$$
More precisely, we prove the following.

\begin{lemma}
\label{lemma:full decomposition}
For any $\eps \in (0, \eps_0)$ there exists a unique  solution in $L^2 (\TT^2 \times (0,T))$ to \eqref{eq:full equation} that may be decomposed as 
$$
B(s)=U_S(s,0)B_0+\cE_1 B_T+\cE_2 B_0,
$$
where, for any $r\in\RR$ and $s\in[0,T-\delta]$, with $\delta>0$, the operators $\cE_1,\cE_2$ satisfy the following  estimates, with a constant $C_{r,\delta}$ depending on $r,\delta$
$$
\|\cE_2 B_0(s)\|_{H^r} \leq C_{r,\delta}\eps\|B_0\|_{H^{r+2-J}}, \quad \|\cE_1 B_T(s)\|_{H^r} \leq C_{r,\delta} \e^{-c\eps^{-1}}\|B_T\|_{H^{r+2-J}} \,.
$$
\end{lemma}
\begin{proof}
Uniqueness follows from computations similar to those in the proof of Lemma \ref{lemma:notion of solution}, using  weighted coercivity and the ellipticity of the operator. It therefore remains to prove the estimates for the decomposition.
    By Lemma \ref{lemma:factorisation}, we may write 
$$
\mathcal L_{\lambda,\eps}=-\eps \alpha^{0,0} (\partial_s-F_J)(\partial_s-S_J)+\eps \alpha^{0,0}R_J,
$$
with $R_J \in \Psi_\eps(\langle \xi \rangle^{1-J})$, for $J$ large but fixed. First, let $B_1$ solve the homogeneous factorised problem 
$$
(\partial_s-F_J)(\partial_s-S_J)B=0, \quad B(0)=B_0, B(T)=B_T.
$$
Then, setting $B=B_1+B_2$, it follows that $B_2$ solves
$$
(\partial_s -F_J)(\partial_s-S_J)B_2=R_J(B_1+B_2), \quad B_2(0)=B_2(T)=0.
$$
Hence, it follows that we seek $B_2$ so that
$$
B_2=\mathcal{G}_\eps \left (R_J(B_1+B_2) \right ).
$$
However, note that since $R_J : H^r \to H^r$ is bounded uniformly, it follows that by Proposition \ref{prop:inhomogeneous}
$$
\|\mathcal{G}_\eps(R_J B_2)\|_{L^\infty([0,T];H^r)} \leq C \eps \|B_2\|_{L^\infty([0,T];H^r)}. 
$$
Hence, $B_2$ is simply given by 
\begin{equation}
\label{eq:neumann series}
B_2=(\Id-\mathcal{G}_\eps R_J)^{-1} \mathcal{G}_\eps R_J B_1=\sum_{n \geq 0}(\mathcal{G}_\eps R_J)^{n+1}  B_1\,.
\end{equation}
By Proposition \ref{prop:homogeneous decomposition}, it holds that $B_1=U_S(s,0)B_0+\cE_{2,1}B_0+\cE_{1,1} B_T$, where $\|\cE_{1,1}B_T(s)\|_{H^{r_1}} \leq C_s \e^{-c\eps^{-1}}\|B_T\|_{H^{r_2}}$,  $\|\cE_{2,1} B_0 (s)\|_{H^{r_1}} \leq C_s \e^{-c\eps^{-1}}\|B_0\|_{H^{r_2}}$, for any $r_1,r_2$. 
From the previous Neumann series expansion \eqref{eq:neumann series}, we deduce that 
$$
B_2(s)=\sum_{n \geq 0} (\mathcal{G}_\eps R_J)^{n+1} U_S(s,0)B_0+\sum_{n \geq 0} (\mathcal{G}_\eps R_J)^{n+1}\cE_{2,1}B_0+\sum_{n \geq 0} (\mathcal{G}_\eps R_J)^{n+1} \cE_{1,1}(B_T).
$$
We set $\cE_{2,2} B_0=\sum_{n \geq 0}(\mathcal{G}_\eps R_J)^{n+1}(U_S(s,0)B_0+\cE_{2,1}B_0)$, and $\cE_{2}=\cE_{2,1}+\cE_{2,2}$. From Proposition \ref{prop:inhomogeneous}, we deduce that 
$$
\|\cE_{2,2} B_0\|_{L^\infty([0,T-\delta];H^r)} \leq C\eps \|U_S(s,0)B_0+\cE_{2,1}B_0\|_{L^\infty([0,T];H^{r+1-J})}.
$$
Applying Lemma \ref{lemma:sobolev bounds propagators} and the bound on $\cE_{2,1}$ given by Proposition \ref{prop:homogeneous decomposition}, we further have 
$$
\|\cE_{2,2} B_0\|_{L^\infty([0,T-\delta];H^r)} \leq C \eps \|B_0\|_{H^{r+2-J}}.
$$
Finally, we set $\cE_{1,2}B_T=\sum_{n \geq 0} (\mathcal{G}_\eps R_J)^{n+1}\cE_{1,1}(B_T)$. Using Proposition \ref{prop:inhomogeneous}, we bound 
$$
\|\cE_{1,2} B_T\|_{\gamma,r} \leq C\eps \|R_J \cE_{1,1}B_T\|_{\gamma,r} \leq C\eps \|\cE_{1,1}B_T\|_{\gamma,r+1-J}.
$$
Finally, applying the decomposition from Proposition \ref{prop:homogeneous decomposition}, we bound this by
$C\eps  \|B_T\|_{H^{r+2-J}}$, completing the proof by setting $\cE_1=\cE_{1,1}+\cE_{1,2}$.
\end{proof}

\subsection{A proof of the perturbative Proposition \ref{prop:pseudodifferential abstract}}

Finally, we turn to the proof of our desired anisotropic estimates. By Lemma \ref{lemma:full decomposition}, it follows that upon taking $J$ large enough, we only need to study the behaviour of $U_S(s,0)$. We recall that $\Pi_{\rm div}$ denotes the Leray projector acting on two-dimensional vector fields, $\Pi_0$\footnote{We remark that $\Pi_0 \in \Psi_\eps (1)$, by writing its Fourier symbol as $\chi(\xi)$, with $\chi(\xi) \in C_c^\infty(\RR^2)$, equal to $1$ for $\xi=0$ and with support contained in a ball of radius $1/2$.} denotes the projection onto the mean of a vector field, $\Pi_h$ denotes the projection of a three-dimensional vector field onto its first two components, and $\Pi_v$ denotes the projection onto its last component. We also recall that we write
$ B=(b,b_v),$
where $b$ consists of the first two components of $B$, while $b_v$ denotes its last component. Finally,  with a slight abuse of notation, we extend the operators $\Pi_{\rm div}$ and $\Id-\Pi_{\rm div}$ to three-dimensional vectors by setting the last component of their outputs to zero. We shall now prove that for $r\in (0,1/2)$ and $ \delta>0$ sufficiently small so that $r+2 \delta <1/2$ the following estimates hold true
\begin{align}
    \|\PP U_S(s,0) \PP B\|_{H^{-r}}  &\leq (1+C\eps^{1/2})\|\PP B \|_{H^{-r}}\,, \\ 
    \|\PP U_S(s,0) (\Id-\PP) B\|_{H^{-r}} &\leq C\eps^{1/2}\|(\Id-\PP) B\|_{H^r}\,,
    \\
    \|(\Id-\PP) U_S(s,0) (\Id-\PP) B\|_{H^{r}}  &\leq (1+C\eps^{1/2})\|(\Id-\PP) B\|_{H^{r}}\,,\\ 
    \|(\Id-\PP) U_S(s,0) \PP B\|_{H^{r}} &\leq C\eps^{\frac{1-2r}{2}} \|\PP B\|_{H^{-r}}\,,
    \\
    \|\Pi_h U_S(s,0) \Pi_v B\|_{H^{r}}  &\leq C\eps^{\frac{1-2r-\delta}{2}}\|\Pi_v B\|_{H^{-r-\delta}}\,,
    \\
    \|\Pi_v U_S(s,0) \Pi_h B\|_{H^{-r-\delta}} &\leq C\eps^{1/2}\|\Pi_h B\|_{H^{-r}}\,,
    \\
    |\Av U_S(s,0)\Av B|  &\leq (1+C\eps^{1/2})|\Av B|\,,
    \\ 
    |\Av U_S(s,0)(\Id-\Av)B|&\leq C\eps^{\frac{1-r-\delta}{2}}\|B\|_{H^{-r-\delta}}\,,
    \\
    \|(\Id-\Av)U_S(s,0)\Av B\|_{H^r} &\leq C\eps^{\frac{1-r}{2}} |\Av B|\,.
\end{align}
These inequalities are a direct consequence of the following lemma.

\begin{lemma}
\label{lemma:slow propagator}
Let $\Pi_1(D), \Pi_2(D) \in \Psi_\eps(1) $ be $L^2$ self-adjoint, so that $\Pi_1\Pi_2=0$, $\Pi_i^2=\Pi_i$, $i=1,2$. The slow propagator $U_S$ satisfies the following estimates for any $r\in\RR$ and $i\neq j$
$$
\|\Pi_i U_S(s,0) \Pi_jB\|_{H^{r}} \leq C_r \eps^{1/2} \|\Pi_j B\|_{H^r}, \quad \|\Pi_i U_S(s,0) \Pi_jB\|_{H^{r+1}} \leq C_r \|B\|_{H^{r}},
$$
and for $i=j$
$$
\|\Pi_i U_S(s,0) \Pi_jB\|_{H^{r}} \leq (1+C_r\eps^{1/2})\|\Pi_j B\|_{H^r}\,.
$$
\end{lemma}
\begin{proof}
 We consider $r \in \ZZ$ and then we recover the estimates by interpolation.
    The key identity we need for this proof is the following symbol estimate on the commutator
$$
[\langle D \rangle^r \Pi_i, S_J]\langle D \rangle^{-r} \in \Psi_\eps(\nu_\eps).
$$
Indeed, this follows by recalling from Lemma \ref{lemma:factorisation} that $S_J -\Op_a(q_s) \in \Psi_\eps(\nu_\eps)$, and that $\Op_a(q_s) \in \Psi_\eps(\mu_\eps)$. Thus, the term $S_J - \Op_a (q_s)$  yields a commutator contribution in $\Psi_\eps(\nu_\eps)$ by the composition formula (Theorem \ref{thm:composition}). The term $\Op_a(q_s)$ is a scalar operator. Hence, the composition formula implies that 
$$
[\langle D \rangle^r \Pi_i, \Op_a(q_s)]\langle D \rangle^{-r} \in \Psi_\eps(\nu_\eps)\,.
$$
We first prove the results where the target and input regularities match. Applying $ \langle D \rangle ^r\Pi_i$ to the slow evolution equation yields 
\begin{equation}
\label{eq:slow evolution projected}
\partial_s \langle D \rangle ^r \Pi_i B=S_J \langle D \rangle ^r\Pi_iB +[\langle D \rangle ^r\Pi_i,S_J]B.
\end{equation}
Arguing exactly as in the proof of Lemma \ref{lemma:sobolev bounds propagators}, we thus deduce that 
$$
\partial_s \|\langle D \rangle ^r \Pi_i B\|_{L^2}^2+c\|\mu_\eps^{1/2}(D) \langle D \rangle ^r \Pi_i B\|_{L^2}^2\leq C \eps  \|  \langle D \rangle ^r B\|_{L^2}^2,
$$
so that for $\eps$ small, the results follow immediately. 

Next, we prove the result when $i \neq j$, and one seeks to gain a derivative.
Take the inner product of \eqref{eq:slow evolution projected} with $\langle D \rangle ^r  \Pi_i B$ to deduce 
$$
\partial_s \frac{1}{2} \|\langle D \rangle ^r \Pi_i B\|_{L^2}^2 -\Re \langle S_J \langle D \rangle^r \Pi_i B, \langle D\rangle^r \Pi_i B\rangle  \leq \left |\langle [\langle D \rangle^r \Pi_i, S_J]B, \langle D \rangle^r \Pi_i B  \rangle\right |.
$$
We estimate 
\begin{align*}
& \left |\langle [\langle D \rangle^r \Pi_i,S_J]B, \langle D \rangle^r \Pi_i B  \rangle\right | 
\\
&\qquad  =|\langle  \mu_\eps^{-1/2}(D)[\langle D \rangle^r \Pi_i ,S_J]\langle D \rangle^{-r+1} \mu_\eps^{-1/2}(D) \langle D \rangle^{r-1}\mu_\eps^{1/2}(D)B, \mu_\eps^{1/2}(D) \langle D \rangle^{r} \Pi_i B\rangle| \\
& \qquad  \leq C\|\mu_\eps^{1/2}(D)B\|_{H^{r-1}}\|\mu_\eps^{1/2}(D) \Pi_i B\|_{H^{r}},
\end{align*}
where we used that, by the composition formula, 
$$
 \mu_\eps^{-1/2}(D)[\langle D \rangle^r \Pi_i ,S_J]\langle D \rangle^{-r+1}\mu_\eps^{-1/2}(D) \in \Psi_\eps(1).
 $$
Hence, applying Young's inequality and the sharp G\r{a}rding inequality as in the proof of Lemma \ref{lemma:sobolev bounds propagators}, we deduce the existence of a constant $C>0$ so that 
\begin{equation}
\label{eq:projected energy inequality slow}
\partial_s \|\langle D \rangle^r \Pi_i B\|_{L^2}^2+\frac{1}{C}\|\mu_\eps^{1/2}(D) \langle D \rangle^r \Pi_i B\|_{L^2}^2 \leq C\eps \|\langle D \rangle^r \Pi_i B\|_{L^2}^2+C\|\mu_\eps^{1/2}(D) B\|_{H^{r-1}}^2.
\end{equation}
From Lemma \ref{lemma:sobolev bounds propagators}, we further know that 
$$
\int_0^s \|\mu_\eps^{1/2}(D) B(t)\|_{H^{r-1}}^2  \dd t\leq C \| \Pi_j B(0)\|_{H^{r-1}}^2.
$$
Hence, integrating \eqref{eq:projected energy inequality slow}, we deduce that 
$$
\| \Pi_i B (s)\|_{H^r} \leq C \|\Pi_j B(0)\|_{H^{r-1}}, 
$$
completing the proof. 
\end{proof}

Finally, we show the following result concerning weak-strong convergence of the slow propagator. 
\begin{lemma}
\label{lemma:weak-strong}
For any $r \in \RR$, $\delta \in (0,2]$, there holds the estimate
$$
\|(U_S(s,0)-\Id)B\|_{H^{r-\delta}} \leq C_{r,\delta}\eps^{\frac{\delta}{2}}\|B\|_{H^r}.
$$
\end{lemma}
\begin{proof}
We begin by noting the identity
$$
U_S(s,0)-\Id=\int_0^s S_J(t)U_S(t,0) \dd t.
$$
Furthermore, $S_J(t) \in \Psi_\eps(\mu_\eps(\xi))$, and since $\mu_\eps(\xi) \leq \eps \langle \xi \rangle^2$, it follows that 
$$
\|S_J(t) B\|_{H^{r-2}} \leq C \eps \|B\|_{H^r}. 
$$
Therefore, we bound 
$$
\|(U_S(s,0)-\Id)B\|_{H^{r-2}} \leq C\eps \|B\|_{H^r}.
$$
Since $\|(U_S(s,0)-\Id)B\|_{H^r}\leq C\|B\|_{H^r}$ by Lemma \ref{lemma:sobolev bounds propagators}, interpolation gives the desired result.
\end{proof}
We now have all the ingredients needed to prove the main result of this section.
\begin{proof}[Proof of Proposition \ref{prop:pseudodifferential abstract}]
In view of   Lemma \ref{lemma:full decomposition} and Lemma \ref{lemma:slow propagator}, it suffices to prove the weak-strong estimate. 
Furthermore, Lemma~\ref{lemma:slow propagator} shows that the off-diagonal terms gain powers of $\eps$, so it remains to estimate the diagonal terms. Lemma~\ref{lemma:weak-strong} gives the bound 
$$
\|\Pi (U_S(s,0)-\Id)\Pi B\|_{H^{r-\delta}} \leq \|(U_S(s,0)-\Id)\Pi B\|_{H^{r-\delta}} \leq C\eps^{\delta/2}\|\Pi B\|_{H^{r}},
$$
completing the proof.
\end{proof}

\appendix

\section{Preliminaries}
\subsection{The Moser lemmas}
The aim of this section is to prove the following two results concerning Moser's lemma, see \cite{Moser1965}.
\begin{lemma}
\label{lemma:quantitative-moser}
Let $d\geq1$, and suppose that
$\rho_0,\rho_1\in C^\infty(\mathbb T^d)$ are strictly positive and have the same mass
$
\int_{\mathbb T^d}\rho_0(x)\dd x = \int_{\mathbb T^d}\rho_1(x)\dd x.
$
Then, there exists a diffeomorphism $\kappa=\kappa(\rho_0,\rho_1):\TT^d \to \TT^d$ isotopic to the identity, so that 
\begin{equation}
\label{eq:moser-pullback}
\rho_1(\kappa(x))\det D\kappa(x)=\rho_0(x).
\end{equation}
Furthermore, we have the following quantitative estimate. For any $k\in \NN$ and $c,M >0$ there exists $C= C(d,k,c,M)>0$ such that for any $\rho_j,\widetilde\rho_j$ of equal mass satisfying
$$\rho_j,\widetilde\rho_j\geq c\,, \qquad \|\rho_j\|_{C^{k+2}}
+\|\widetilde\rho_j\|_{C^{k+2}}\leq M, 
\qquad j=0,1,$$ 
it holds that 
\begin{align*}
&\|\kappa(\rho_0,\rho_1)
-\kappa(\widetilde\rho_0,\widetilde\rho_1)\|_{C^k}+
\|\kappa(\rho_0,\rho_1)^{-1}
-\kappa(\widetilde\rho_0,\widetilde\rho_1)^{-1}\|_{C^k}
\leq
C\sum_{j=0}^1
\|\rho_j-\widetilde\rho_j\|_{C^{k+2}}\,.
\label{eq:moser-quantitative}
\end{align*}
\end{lemma}

\begin{proof}
By assumption, $\rho_0-\rho_1$ is mean-free on $\TT^d$. Hence, we may define $Z=\nabla \Delta^{-1}(\rho_0-\rho_1)$, and $\nabla \cdot Z=\rho_0-\rho_1$. Set $\rho(x,t)=(1-t)\rho_0 (x)+t \rho_1 (x)$. Then, $\rho$ is strictly positive, and so we may define $V(t,x)=\frac{Z(x)}{\rho(t,x)}$. Hence, a computation reveals that $\rho$ solves
$$
\partial_t \rho+\nabla \cdot (V \rho)=0,
$$
$\rho(\cdot,0)=\rho_0$. Hence, by uniqueness of solutions to the continuity equation with smooth coefficients, it follows that $\rho(X^t(x), t )\det DX^t(x)=\rho_0(x)$, where $X^t$ is the flow map generated by $V(t,x)$
$$
\frac{\dd}{\dd t} X^t(x)=V(t,X^t(x)), \quad X^0(x)=x\,.
$$
Since $\rho(\cdot, 1)=\rho_1 (\cdot )$, we simply take $\kappa(\rho_0,\rho_1)=X^1$.
We finally prove the quantitative statement. Let
$\widetilde V$ denote the vector field associated with
$(\widetilde\rho_0,\widetilde\rho_1)$. Standard elliptic estimates (see e.g. \cite{GilbargTrudinger2001}*{Chapter 9, \S~9.5, Theorem 9.11})
and the lower bound $\rho,\widetilde\rho\geq c$ immediately yield
$$
\sup_{t\in[0,1]}
\left(
\|V(t)\|_{C^{k+1}}
+
\|\widetilde V(t)\|_{C^{k+1}}
\right)
\leq C(d,k,c,M)
$$
and
$$
\sup_{t\in[0,1]}
\|V(t)-\widetilde V(t)\|_{C^k}
\leq
C(d,k,c,M)
\sum_{j=0}^1
\|\rho_j-\widetilde\rho_j\|_{C^{k+2}}.
$$
Hence, by a standard $C^k$ stability estimate for flows (see e.g.\ \cite{TaylorNonlinearODE}*{Proposition 2.2}) and by
Gr\"onwall's inequality, we obtain
\begin{align*}
\|X^1-\widetilde X^1\|_{C^k}
\leq
C(d,k,c,M)
\sum_{j=0}^1
\|\rho_j-\widetilde\rho_j\|_{C^{k+2}} \,.
\end{align*}
Applying the same estimate after exchanging the two densities and noting that $\kappa (\rho_0, \rho_1)^{-1} = \kappa (\rho_1, \rho_0)$ gives
the corresponding estimate for the inverse maps, concluding the proof. 
\end{proof}
As an immediate corollary, we obtain the following result.
\begin{corollary}
\label{cor:foliated-moser}
Let $d\geq2$, and suppose that
$
\rho\in C^\infty(\mathbb T^{d-1}\times\mathbb T),$ $
\int_{\mathbb T^{d-1}}
\rho(x,z) \dd x =1, \ \forall z\,,
$ and $\rho$ is strictly positive.
Then there exists a diffeomorphism $R:\TT^d \to \TT^d$ of the form
$
R(x,z)=\bigl(\kappa_{z}(x),z\bigr)
$
such that
\begin{equation}
\label{eq:foliated-moser-jacobian}
\det DR(x,z)=\rho(x,z).
\end{equation}
\end{corollary}

\begin{proof}
For every fixed $z$, let $\rho_z:\TT^{d-1}\to \RR_+$ be given by $\rho_z(x)= \rho(x,z)$. Then, $\rho_z$ and the constant function $1$ have equal mass, and $z \mapsto \rho_z$ is smooth. Therefore, applying Lemma \ref{lemma:quantitative-moser}, we obtain a smooth family of diffeomorphisms $\kappa_z:\TT^{d-1} \to \TT^{d-1}$ so that $\det D_x \kappa_z(x)=\rho(x,z)$.  Furthermore, since $\rho(x,z+1)=\rho(x,z)$, it follows that $\kappa_z$ is periodic in $z$ and the result follows.
\end{proof}
\subsection{Keller--Liverani perturbation theory}
The aim of the present section is to introduce a slight generalisation of the classical Keller--Liverani perturbation result \cite{Keller_Liverani}, which explicitly records uniformity over multiple parameters. In order to prove the result, we must first record the following from the paper \cite{Keller_Liverani}. 
\begin{theorem}[Keller--Liverani, \cite{Keller_Liverani}]
\label{thm:keller-liverani 1}
Let $(X,\|\cdot \|_{X})$ be a Banach space, embedding compactly into the Banach space $(Y,\|\cdot \|_{Y})$, so that furthermore $\|\cdot \|_{Y} \leq \|\cdot \|_{X}$. 
Let $\{\cL_{\eps}\}_{\eps \in [0,1)}$ be a collection of bounded operators on $X$, which extend to bounded operators on $Y$. Assume furthermore that there exist constants $C,\alpha, M>0$, $\alpha <M$, such that the following estimates hold for all $\eps\in[0,1)$ and $m\in\NN$:
\begin{enumerate}
    \item $\|\cL_{\eps}^m B\|_{Y} \leq CM^m\|B\|_{Y}$
    \item $\|\cL_{\eps}^m B\|_{X} \leq C\alpha^m \|B\|_{X}+CM^m\|B\|_{Y}$
    \item $\|\cL_{\eps} B-\cL_{0} B\|_{Y} \leq \eta(\eps)\|B\|_{X} $, where $\eta: \RR_+ \to \RR_+$ is some monotone, upper-semicontinuous function that tends to zero as its argument tends to zero.
\end{enumerate}
Then, if $\cL_0$ has an algebraically simple, isolated eigenvalue $\lambda_0$ with $\alpha<|\lambda_0|$, then the same holds for $\cL_\eps$ for all $\eps$ small enough, and the corresponding eigenvalue $\lambda_\eps$ satisfies $\lambda_\eps \to \lambda_0$ as $\eps \to 0$. In fact, let $\Gamma\subset\CC$ be a positively oriented simple contour in the intersection of the resolvent set of $\mathcal{L}_0$ and $\{|z| >\alpha\}$, containing $\lambda_0$ in its interior, so that the associated Riesz projector

$$
\Pi=\frac{1}{2 \pi i}\int_{\Gamma}(z-\cL_0)^{-1} \dd z
$$
has rank one. Then, for all $\eps>0$ small enough, $\Gamma \subset \rho(\cL_\eps)$, and the diffusive Riesz projector 
$$
\Pi_\eps=\frac{1}{2 \pi i}\int_{\Gamma}(z-\cL_\eps)^{-1} \dd z
$$
also has rank $1$ for all $\eps$ small enough.
\end{theorem}

From this, we then deduce the following generalisation.
\begin{theorem}[\cite{Keller_Liverani}]
\label{thm:keller liverani}
Let $(X,\|\cdot \|_{X})$ be a Banach space, embedding compactly into the Banach space $(Y,\|\cdot \|_{Y})$, so that furthermore $\|\cdot \|_{Y} \leq \|\cdot \|_{X}$. Let $(\mathcal{F},d_\mathcal{F})$, $(\mathcal{G},d_{\mathcal{G}})$ be two metric spaces, so that $(\mathcal{G},d_{\mathcal{G}})$ is compact.
Let $\{\cL_{u,g}\}_{u \in \mathcal{F},g \in \mathcal{G}}$ be a collection of bounded operators on $X$, which extend to bounded operators on $Y$, and suppose that, for every fixed $u\in\mathcal{F}$, the map $g \mapsto \cL_{u,g}$ is continuous in the operator norm topology on $X$ and $Y$. Finally, assume that the $\cL_{u,g}$ satisfy the following: there exist constants $C,\alpha, M>0$, $\alpha <M$, such that the following estimates hold for all $u\in\mathcal{F}$, $g\in\mathcal{G}$, and $m\in\NN$:
\begin{enumerate}
    \item $\|\cL_{u,g}^m B\|_{Y} \leq CM^m\|B\|_{Y}$
    \item $\|\cL_{u,g}^m B\|_{X} \leq C\alpha^m \|B\|_{X}+CM^m\|B\|_{Y}$
    \item $\|\cL_{u,g} B-\cL_{v,g} B\|_{Y} \leq \eta(d_{\mathcal{F}}(u,v))\|B\|_{X} $, where $\eta: \RR_+ \to \RR_+$ is some monotone, upper-semicontinuous function that tends to zero as its argument tends to zero.
\end{enumerate}
Assume further that there exist $u\in\mathcal{F}$, $g\in\mathcal{G}$, and $\lambda_0\in\sigma_X(\cL_{u,g})$ corresponding to an algebraically simple eigenvalue, with $\alpha <|\lambda_0|<M$. Then, there exists an open set $G \subset \mathcal{G}$ containing $g$, so that for all $h \in G$, $\cL_{u,h}$ has an isolated, algebraically simple eigenvalue $\lambda_h \to \lambda_0$ as $h \to g$. Furthermore, there exists $\delta_{\mathcal F}>0$ so that for all $v \in B_{\delta_{\mathcal F}}(u)$, $\cL_{v,h}$ similarly has isolated, algebraically simple eigenvalues $\lambda_{h,v}$, and it holds \emph{uniformly} for $h \in G$ that $\lambda_{h,v} \to \lambda_h$, as $d_{\mathcal{F}}(u,v) \to 0$.
\end{theorem}
\begin{proof}
Since for fixed $u$, the map $g \mapsto \cL_{u,g}$ is continuous, the existence of an open set $G \ni g$ so that uniformly for $h \in \overline G$, $\cL_{u,h}$ has an algebraically simple eigenvalue $\lambda_h$, uniformly larger in magnitude than $\alpha$, follows by classical perturbation theory, see e.g. \cite{K76}. Indeed, fix a contour $\Gamma \subset \CC$, and define 
$$
\Pi_{v,h}=\frac{1}{2 \pi i}\int_{\Gamma}(z-\cL_{v,h})^{-1} \dd z.
$$
Then, we can pick $G$ so that for all $h \in \overline G$, $\mathrm{Rank}(\Pi_{u,h})=1$. We now claim that there exists $\delta_{\mathcal F}>0$ so that for all $v \in \mathcal{F}$ with $d_\mathcal{F}(u,v)\leq \delta_{\mathcal F}$ we have $\mathrm{Rank}(\Pi_{v,h})=1$. Indeed, suppose not. Then, there would exist a sequence $v_n, h_n$ so that $d_\mathcal{F}(v_n,u)\leq \frac{1}{n}$, but $\mathrm{Rank}(\Pi_{v_n,h_n}) \neq 1$. By compactness, we may find a (non-relabeled) subsequence so that $h_n \to h_\star$. But now, the operators $\cL_n=\cL_{v_n,h_n}$, $\cL_\infty=\cL_{u,h_\star}$ satisfy the assumptions of Theorem \ref{thm:keller-liverani 1}, with $\eps=\frac{1}{n}$, and furthermore we have
$$
\|\cL_{v_n,h_n}-\cL_{u,h_\star}\|_{X \to Y} \leq \|\cL_{v_n,h_n}-\cL_{u,h_n}\|_{X \to Y}+\|\cL_{u,h_n}-\cL_{u,h_\star}\|_{X \to Y} \to 0,
$$
as $n \to \infty$.
Thus, for all $n$ large enough, $\mathrm{Rank}(\Pi_{v_n,h_n})=1$, yielding a contradiction. Hence, we have ensured the existence of $\lambda_{h,v}$. Assume for a contradiction that $\lambda_{h,v}$ does not converge to $\lambda_{h}$ uniformly as $d_{\mathcal{F}}(u,v)\to 0$. Then, we again have a sequence of $v_n \to u$ and $h_n \in \overline G$ so that $|\lambda_{h_n,v_n} -\lambda_{h_n,u}|\gtrsim 1$. Up to a subsequence, $h_n \to h_\star$. Then, applying Theorem \ref{thm:keller-liverani 1} to $\cL_n=\cL_{v_n,h_n}$, $\cL_\infty=\cL_{u,h_\star}$, we deduce that
$$
|\lambda_{h_n,v_n} -\lambda_{h_n,u}| \to 0,
$$
as $n \to \infty$, completing the proof.
\end{proof}

\section{Elements of the pseudodifferential calculus}
\label{sec:pseudodifferential intro}
We begin by introducing the following class of symbols adapted to scales. We remark that, while the results cited throughout this section are stated on $\RR^d$, they are equally valid on $\TT^d$ via a standard argument, see e.g.\ \cite{RuzhanskyTurunen2010Torus}. 
Throughout this section, unless explicitly stated otherwise, every constant is assumed \emph{independent}  of $\eps$ and of $\lambda\in K_\lambda$, where $K_\lambda \subset \CC$ is a compact set.

\begin{definition}
\label{def:scale function}
Let $\ell_\eps: \RR^d \to [1,\infty)$ be a smooth function possibly depending on $\eps$. We say it is an admissible scale if for any multi-index $\beta \in \NN^d$ there holds 
$$
|\partial_\xi^\beta \ell_\eps(\xi)| \leq C_\beta \ell_\eps(\xi)^{1-|\beta|}.
$$
Furthermore, we call a smooth positive function $m_\eps(\xi)$ an admissible order function if there are constants $C,M>0$ so that 
$$
\frac{m_\eps(\zeta)}{m_\eps(\xi)} \leq C\left (1+\frac{|\zeta-\xi|}{\ell_\eps(\xi)^{1/2}} \right )^M \left (1+\frac{|\zeta-\xi|}{\ell_\eps(\zeta)^{1/2}} \right )^M,
$$
and there holds the estimate
$$
|\partial_\xi^\beta m_\eps(\xi)|\leq C_\beta m_\eps(\xi) \ell_\eps(\xi)^{-|\beta|}.
$$
We write $a_\eps(x,t,\xi,\lambda) \in S_\eps(m_\eps;\ell_\eps)$ if 
$$
\| \partial_\xi^\beta \partial_x^\alpha \partial_t^\gamma a_\eps(x,t,\xi,\lambda)\|_{L^\infty} \leq C_{\beta,\alpha,\gamma}m_\eps(\xi)\ell_\eps(\xi)^{-|\beta|}, 
$$
and set the corresponding seminorm 
$$
|a_\eps|_{m,\ell,K}=\max_{|\alpha|+|\beta|+|\gamma| \leq K}  \sup_{\xi \in \RR^d, \lambda \in K_\lambda}\ell(\xi)^{|\beta|}m(\xi)^{-1}\| \partial_\xi^\beta \partial_x^\alpha \partial_t^\gamma a_\eps(x,t,\xi,\lambda)\|_{L^\infty}.
$$
In the special case where $\ell_\eps=\langle \xi \rangle$, we simply write $a_\eps \in S_\eps(m_\eps)$.
Given a symbol, we may quantise it using the usual Kohn-Nirenberg quantisation on the torus. Indeed, given $f \in C^\infty(\TT^d)$, we define 
\begin{equation}
\label{eq:kohn-nirenberg}
\Op_x(a_\eps) f=\sum_{k \in \ZZ^d}a_\eps(x,t,2 \pi k,\lambda)\e^{2 \pi ix\cdot k} \widehat f(k),
\end{equation}
where the frequency variable is $\xi=2 \pi k$.
\end{definition}
We remark that for $a(x,\xi)=\sum_{|\alpha|\leq M} a_\alpha(x) \xi^\alpha$, the Kohn-Nirenberg quantisation \eqref{eq:kohn-nirenberg} recovers the linear differential operator 
$$
\Op_x(a)=\sum_{|\alpha|\leq M} a_\alpha(x)  D_x^\alpha,
$$
where $D_{x_j}=-i \partial_{x_j}$.
 We call an operator of the form $\Op_x(a)$ a \emph{pseudodifferential operator}, and we call $a$ its symbol. Given a pseudodifferential operator $A$,  whose symbol lies in $S_\eps(m_\eps;\ell_\eps)$, we write 
$$
A \in \Psi_\eps(m_\eps;\ell_\eps).
$$
As before, if $\ell_\eps =\langle \xi \rangle$, we shorten this to $A \in \Psi_\eps(m_\eps)$.\\

We remark that our present framework is somewhat distinct from the classical Weyl--H\"ormander calculus found e.g. in \cite{Hormander2007VolIII}. Instead, our framework is based on the one in \cite{Lerner2010Metrics}, which we have simplified here in order to avoid introducing the geometric notions of \cite{Lerner2010Metrics}. We now verify that our results follow from those stated in \cite{Lerner2010Metrics}.
For $X=(x,\xi)$ and $T=(y,\eta)$, set
$$
  g_{\eps,X}(T)
  :=|y|^2+\ell_\eps(\xi)^{-2}|\eta|^2.
$$
Its symplectic dual metric and Planck function are
$$
  g_{\eps,X}^{\sigma}(T)
  =\ell_\eps(\xi)^2|y|^2+|\eta|^2,
  \qquad
  h_{g_\eps}(X)=\ell_\eps(\xi)^{-1}.
$$
 The assumptions in
Definition~\ref{def:scale function} imply, uniformly in $\eps$, that $g_\eps$ is a split, $\sigma$-temperate H\"ormander metric satisfying the uncertainty principle and that every admissible
$m_\eps$ is a $g_\eps$-admissible weight in the sense of
\cite{Lerner2010Metrics}*{Definition~2.2.15}. Indeed,
$|\nabla\ell_\eps|\leq C$ gives local comparability when
$|\zeta-\xi|\lesssim\ell_\eps(\xi)$ and global polynomial
comparison. The derivative estimates on $m_\eps$ give
$g_\eps$-continuity, while the condition on $\frac{m_\eps(\zeta)}{m_\eps(\xi)}$ in Definition \ref{def:scale function} gives
$\sigma,g_\eps$-temperateness. Consequently, for each fixed
$(t,\lambda)$, the $(x,\xi)$-symbol
$a_\eps(\,\cdot\,,t,\,\cdot\,,\lambda)$ lies uniformly in
$S(m_\eps,g_\eps)$, and the same is true after any of the
displayed $t$-derivatives.

Associated to this class of pseudodifferential operators is a calculus, which allows us to compute expansions for the composition of two operators. Indeed, it is clear from \eqref{eq:kohn-nirenberg} that in general the composition of two pseudodifferential operators will not simply be given by a product of their symbols. However, the following theorem makes the deviation explicit. 
\begin{theorem}[\cite{Lerner2010Metrics}*{Theorems~2.3.7 and~2.3.19}]
\label{thm:composition}
Let $a_\eps \in S_\eps(m_\eps;\ell_\eps), b_\eps \in S_\eps(n_\eps;\ell_\eps)$. Then, the operator $\Op_x(a_\eps) \Op_x(b_\eps)$ is given by $\Op_x(a_\eps \# b_\eps)$, where 
$$
a_\eps \# b_\eps \sim \sum_{\alpha}\frac{1}{\alpha !}\partial_\xi^\alpha a_\eps \partial_x^\alpha b_\eps i^{-|\alpha|},
$$
in the sense that, for any $N \in \NN$, we have
$$
r_N=a_\eps \# b_\eps -\sum_{|\alpha| \leq N-1}\frac{1}{\alpha !}\partial_\xi^\alpha a_\eps \partial_x^\alpha b_\eps i^{-|\alpha|} \in S_\eps(m_\eps n_\eps \ell_\eps^{-N};\ell_\eps).
$$
 Furthermore, there exist an explicit constant $C>0$, independent of $a,b,\eps,\lambda$, and integers $L_a,L_b$, depending only on $m_\eps,n_\eps,K,N,d$ and independent of $\eps$, such that the following error estimate holds:
$$
|r_N|_{m_\eps n_\eps \ell_\eps^{-N},\ell_\eps,K} \leq C|a_\eps|_{m_\eps, \ell_\eps, K+L_a}|b_\eps|_{n_\eps,\ell_\eps,K+L_b}.
$$
\end{theorem}
\begin{proof}
For the metric $g_\eps$, Lerner's \cite{Lerner2010Metrics} reciprocal Planck weight is
$\lambda_{g_\eps}=\ell_\eps$. Hence
\cite{Lerner2010Metrics}*{Theorem~2.3.19} gives the displayed Kohn--Nirenberg expansion and the remainder
$(m_\eps n_\eps\ell_\eps^{-N},g_\eps)$, while
\cite{Lerner2010Metrics}*{Theorem~2.3.7} gives continuity in the
Fr\'echet symbol topologies, and therefore the stated finite-seminorm
estimate.
\end{proof}
Similarly, one may estimate the difference between the $L^2$ adjoint associated to a pseudodifferential operator and the operator obtained by quantising the pointwise adjoint. Indeed, we have the following result, which follows from \cite{Lerner2010Metrics}*{Theorem 2.3.18}, upon noting that for the Weyl quantisation, we have $(a^w)^\star=(a^\star)^w$. 
\begin{lemma}[\cite{Lerner2010Metrics}*{Theorem 2.3.18}]
\label{lemma:adjoint}
Let $a_\eps \in S_\eps(m_\eps;\ell_\eps)$. Then, 
$$
(\Op_x(a_\eps))^\star-\Op_x(a_\eps^\star) \in \Psi_\eps(m_\eps \ell_\eps^{-1};\ell_\eps),
$$
Furthermore, let $\tilde a_\eps$ denote the symbol of $(\Op_x(a_\eps))^\star-\Op_x(a_\eps^\star)$. There exist an explicit constant $C>0$, independent of $a,\eps,\lambda$, and an integer $L$, depending only on $m_\eps,K,d$ and independent of $\eps$, such that the following error estimate holds:
$$
|\tilde a_\eps|_{m_\eps \ell_\eps^{-1},\ell_\eps,K} \leq C|a_\eps|_{m_\eps, \ell_\eps, K+L}.
$$
\end{lemma}
Finally, we shall need the following results, which show that operators in the symbol class $\Psi_\eps(1)$ are bounded on Sobolev spaces.
\begin{theorem}[\cite{RuzhanskyTurunen2010Torus}*{Theorem 9.1}]
\label{thm:L^2 boundedness}
Let $A=\Op_x(a)$, with $a \in S_\eps(1;\ell_\eps)$. Then, there exists a constant $C>0$ depending only on finitely many seminorms of $a$ so that 
$$
\|A u\|_{L^2} \leq C\|u\|_{L^2},
$$
for any $u \in C^\infty(\TT^d)$.
\end{theorem}
As an immediate corollary, we have the following.
\begin{corollary}
\label{corollary:sobolev boundedness}
Let $A=\Op_x(a)$, with $a \in S_\eps(m)$, where $m$ is a strictly positive order function. Then, for any $s \in \RR$, there exists $C_s$ depending only on finitely many seminorms of $a$ so that 
$$
\|m^{-1}(D)A u\|_{H^s} \leq C_s \|u\|_{H^s},
$$
for any $u \in C^\infty(\TT^d)$.
\end{corollary}
\begin{proof}
We have that $\|m^{-1}(D)A u \|_{H^s}=\|\langle D \rangle^s m^{-1}(D)A u\|_{L^2}$. Furthermore, $\langle D \rangle^s  m^{-1}(D) A=\langle D \rangle^s  m^{-1}(D)A \langle D \rangle^{-s} \langle D \rangle^s$, and by means of the composition formula (Theorem \ref{thm:composition}) it follows that $\langle D \rangle^s  m^{-1}(D)A \langle D \rangle^{-s} \in \Psi_\eps(1)$. Hence, the result follows from Theorem \ref{thm:L^2 boundedness}.
\end{proof}
With these preliminary lemmas stated, we may now present the so-called sharp G\r{a}rding inequality, which gives a lower bound on expressions of the form 
$$
\langle \Op_x(a(x,\xi)) u, u \rangle_{L^2(\TT^d)},
$$
when $a(x,\xi)$ is a positive-definite Hermitian matrix. Indeed, if $a(x,\xi)=a(\xi)$, then it immediately follows by Plancherel that 
$$
\langle \Op_x(a(\xi)) u, u \rangle_{L^2(\TT^d)} \geq 0.
$$
The sharp G\r{a}rding inequality roughly states that the dependence of $a(x,\xi)$ on $x$ only introduces a lower-order error to this non-negative lower bound. Indeed, we have the following result, which is a variant of \cite{Lerner2010Metrics}*{Theorem 2.5.4}.
\begin{theorem}[sharp G\r{a}rding \cite{Lerner2010Metrics}*{Theorem 2.5.4}]
\label{thm:sharp garding}
Let $a_\eps \in S_\eps(m_\eps;\ell_\eps)$ be a matrix-valued symbol, so that its Hermitian part 
$$
\frac{a_\eps(x,\xi)+a_\eps(x,\xi)^\star}{2}
$$
is positive semi-definite. Then,  there exist a constant $C>0$, independent of $\eps$, and an integer $K\in\NN$ such that
$$
\Re \langle \Op_x(a_\eps) u, u \rangle_{L^2(\TT^d)} \geq -C|a_\eps |_{m_\eps,\ell_\eps,K}\|\left (\frac{m_\eps}{\ell_\eps}(D) \right )^{1/2} u\|_{L^2}^2.
$$
\end{theorem}
\begin{proof}
We begin by assuming that $a_\eps(x,\xi)$ is Hermitian. Set $q_\eps=(\frac{m_\eps}{\ell_\eps})^{1/2}$. Then, $q_\eps^{-1} a_\eps q_\eps^{-1}(x,\xi)$ is  a non-negative Hermitian matrix. Applying Theorem 2.5.4 from \cite{Lerner2010Metrics}, we deduce that there exists a constant $C>0$ so that
$$
\Re \langle \Op_x(q_\eps^{-1} a_\eps q_\eps^{-1}) u, u \rangle_{L^2} \geq -C |q_\eps^{-1} a_\eps q_\eps^{-1}|_{\ell_\eps,\ell_\eps,K}\|u\|_{L^2}^2.
$$
Applying this to $\Op_x(q_\eps)u$, noting that $\Op_x(q_\eps)$ is self-adjoint on $L^2$, and that, by the composition formula, $\Op_x(q_\eps^{-1}a_\eps q_\eps^{-1})\in \Psi_\eps(\ell_\eps,\ell_\eps)$, we deduce that 
$$
\Re \langle \Op_x(q_\eps)\Op_x(q_\eps^{-1}a_\eps q_\eps^{-1})\Op_x(q_\eps) u, u \rangle_{L^2} \geq -C|q_\eps^{-1}a_\eps q_\eps^{-1}|_{\ell_\eps,\ell_\eps,L}\|\Op_x(q_\eps)u\|_{L^2}^2. 
$$
Furthermore, by the definition of $q_\eps$, we see that $|q_\eps^{-1}a_\eps q_\eps^{-1}|_{\ell_\eps,\ell_\eps,L}\leq C |a_\eps|_{m_\eps,\ell_\eps,K}$, for some $K$ depending on $L$.
But now, note that from the definition of $q_\eps$, it follows that it is an admissible order function with respect to $\ell_\eps$, and $q_\eps^{-1} \in S_\eps(q_\eps^{-1};\ell_\eps)$. Hence, the composition formula of Theorem \ref{thm:composition} implies that $\Op_x(q_\eps) \Op_x(q_\eps^{-1}a_\eps q_\eps^{-1})\Op_x(q_\eps)-\Op_x(a_\eps) \in \Psi(m_\eps \ell_\eps^{-1};\ell_\eps)$. Thus, write now $R_\eps=\Op_x(q_\eps) \Op_x(q_\eps^{-1}a_\eps q_\eps^{-1})\Op_x(q_\eps)-\Op_x(a_\eps)$. Then, we estimate 
$$
|\langle R_\eps u,u \rangle|=|\langle \Op_x(q_\eps^{-1})R_\eps \Op_x(q_\eps^{-1}) \Op_x(q_\eps) u, \Op_x(q_\eps) u \rangle|\leq C\|\Op_x(q_\eps) u\|_{L^2}^2,
$$
where we used that the composition formula implies that $\Op_x(q_\eps^{-1})R_\eps \Op_x(q_\eps^{-1}) \in \Psi_\eps(1;\ell_\eps)$, and consequently, Theorem \ref{thm:L^2 boundedness} gives the desired upper bound. We remark that once again, the constant $C$ is bounded by an $\eps$-independent constant, times the symbol seminorm $|a_\eps|_{m_\eps,\ell_\eps;K}$. Finally, for general non-Hermitian symbols $a_\eps$, we note that 
\begin{align*}
\Re \langle \Op_x(a_\eps) u, u \rangle_{L^2}&=\frac{1}{2}\Re \langle (\Op_x(a_\eps)+\Op_x(a_\eps)^\star)u, u \rangle_{L^2}\\
&=\frac{1}{2}\Re \langle (\Op_x(a_\eps)+\Op_x(a_\eps^\star))u,u \rangle_{L^2}+\frac{1}{2} \Re \langle (\Op_x(a_\eps)^\star-\Op_x(a_\eps^\star))u,u \rangle_{L^2}.
\end{align*}
By Lemma \ref{lemma:adjoint}, $\Op_x(a_\eps)^\star-\Op_x(a_\eps^\star) \in \Psi_\eps(m_\eps \ell_\eps^{-1};\ell_\eps)$, so it may be estimated exactly as the term $R_\eps$, and hence we conclude the proof.
\end{proof}

\section*{Acknowledgments}

The research of MCZ was partially supported by the Royal Society University Research Fellowship URF\textbackslash R1\textbackslash 191492 and by the ERC/EPSRC Horizon Europe Guarantee EP/X020886/1. MS acknowledges support from the Chapman Fellowship at Imperial College London. DV acknowledges support from an Imperial College President's PhD Scholarship.

\section*{Declaration of AI Use}

The authors used ChatGPT 5.5 Pro and ChatGPT 5.6 Sol during the
development and preparation of this paper. The present work combines
ideas that the authors had been developing over the preceding two
years with insights arising from several extended conversations with
these systems, each spanning dozens of prompts. Below we describe the AI contributions that we regard as most significant.

The idea for the anisotropic space $X$ arose while we were attempting to build microlocal anisotropic Banach spaces in the spirit of \cites{BaladiTsujii2007AnisotropicHolderSobolev,DZ16} in a simplified model of Theorem~\ref{thm:time-periodic}, where we only retained the top left entry of the derivative cocycle. At the end of a long chat, ChatGPT 5.5 Pro suggested using the full derivative cocycle and eventually proposed a microlocal symbol equivalent to the norm of $X$.\footnote{At the time, this suggestion came across as ``hallucination'', since it was initially unable to justify what made retaining the full derivative cocycle advantageous over simply retaining the dominant top left entry of order $O(N^2)$.} Further exchanges helped formulate the suspension construction summarized in Lemma~\ref{lemma:flexibility of section maps}, after the authors proposed lifting the time-periodic dynamics to an autonomous flow.

For the resistive perturbation, the authors proposed using the Keller--Liverani framework, while ChatGPT suggested the boundary calculus of Grubb \cite{Grubb1996} and an initial, more indirect
stochastic argument. After checking this argument, the authors replaced it with the deterministic formulation of Lemma~\ref{lemma:notion of solution}. ChatGPT 5.6 Sol also assisted
with the section-return computation leading to \eqref{eq:inviscid section map} and with preliminary checks of the perturbative analysis in Section~\ref{sec:openness}. It was used intermittently during drafting to check mathematical consistency and exposition.

All AI-generated suggestions, computations, and references were independently verified and, where necessary, corrected or reformulated by the authors. The resulting document is entirely human written.

 \bibliographystyle{abbrv}
 \bibliography{DynamoBiblio.bib}

\end{document}